\documentclass[11pt]{article}
\usepackage{amsthm,amsmath,amsfonts,amssymb,bbm}
\usepackage[normalem]{ulem}
\usepackage[numbers]{natbib}
\usepackage[colorlinks,citecolor=blue,urlcolor=blue]{hyperref}
\usepackage{graphicx}
\usepackage[top=1in, bottom=1.25in, left=1.2in, right=1.2in]{geometry}
\theoremstyle{plain}

\newtheorem{Theorem}{Theorem}[section]
\newtheorem{Proposition}[Theorem]{Proposition}
\newtheorem{Corollary}[Theorem]{Corollary}
\newtheorem{Lemma}[Theorem]{Lemma}
\newtheorem{Definition}[Theorem]{Definition}
\newtheorem{Remark}[Theorem]{Remark}

\newtheorem{expl}[Theorem]{Example}

\usepackage[inline]{enumitem}
\usepackage{authblk}
\author[1]{Iza Danielewska}
\author[1]{Bartosz Ko{\l}odziejek}
\author[1]{Jacek Weso{\l}owski}
\author[2]{Xiaolin Zeng}
\affil[1]{Warsaw University of Technology, Poland}
\affil[2]{Universit\'e de Strasbourg, France}
\date{}
\title{Graphical Negative Multinomial and Multinomial Models\\ with Dirichlet-type priors}
\hypersetup{
 pdfauthor={Bartosz Ko\l odziejek, Jacek Weso\l owski, Xiaolin Zeng},
 pdftitle={Graphical Negative Multinomial and Multinomial Models with Dirichlet-type priors},
 pdfkeywords={},
 pdfsubject={},
 pdfcreator={Emacs 28.2 (Org mode 9.5.5)}, 
 pdflang={English}}

 \newcommand{\gbinom}[3]{\binom{#1}{#2}_{\!\!\!#3}} 
\newcommand{\gbinomm}[3]{\genfrac{[}{]}{0pt}{}{#1}{#2}_{\!#3}}

\begin{document}

\maketitle
\begin{abstract}
		Bayesian statistical graphical models are typically classified as either continuous and parametric (Gaussian, parameterized by the graph-dependent precision matrix with Wish\-art-type priors) or discrete and non-parametric (with graph-dependent structure of probabilities of cells and Dirichlet-type priors).  We propose to break this dichotomy by introducing two discrete parametric graphical models on finite decomposable graphs: the graph negative multinomial and the graph multinomial distributions  (the former related to the Cartier-Foata theorem for the graph generated  free quotient monoid).  These models interpolate between the product of univariate negative binomial laws and the negative multinomial distribution, and between the product of binomial laws and the multinomial distribution, respectively. We derive their Markov decompositions and provide related probabilistic representations. 
    
    We also introduce graphical versions of the Dirichlet and inverted Dirichlet distributions, which serve as conjugate priors for the two discrete graphical Markov models. We derive explicit normalizing constants for both graphical Dirichlet laws and establish their independence structure (a graphical version of neutrality), which yields a strong hyper Markov property for both Bayesian models. We also provide characterization theorems for graphical Dirichlet laws via respective graphical versions of neutrality, which extends previously known results.
\end{abstract}



\section{Introduction}
\label{sec:orgd75a8a6}
Graphical models belong to the basic toolbox of modern statistical analysis. Parametric graphical models are of great interest since they allow to reduce the number of unknowns in the model. Typically, Bayesian continuous graphical models are Gaussian and thus purely parametric, with a Wishart-type prior for the covariance matrix. On the other hand,  Bayesian discrete graphical models are purely non-parametric, with Dirichlet-type priors on the probabilities of cells in multi-way tables, \cite{Dawid1993}.
In this paper we break this dichotomy by   introducing two  purely parametric discrete graphical models: the graph negative multinomial model, denoted \(\operatorname{nm}_G\) and the graph multinomial model, denoted \(\operatorname{mult}_G\), together with conjugate Dirichlet-type priors on their parameters. The two models  are defined (first in the non-Bayesian setting) in terms of random vectors \(N=(N_i)_{i\in V}\), where \(V\) is a vertex set of a finite  graph \(G=(V,E)\), and each \(N_i\) takes values in \(\mathbb{N}:=\{0,1 ,\ldots\}\), see Definition \ref{def-graph-multinomial-and-negative-multinomial-distributions}.  Our considerations are restricted to decomposable graphs; potential generalizations to non-decomposable models pose significant challenges, especially for the Bayesian inference.  Let us mention also that, compared to the abundant literature for Gaussian graphical models, literature for discrete graphical models is relatively limited, see e.g. \cite{CL06,CM12,LRR19,TLX22,R23}, so our paper contributes to slightly bridging the gap. 

Here are the key features of the proposed models:
\begin{enumerate}[label={\arabic*)}, itemjoin={,}, itemjoin*={, and}]
	
	  \item 
    The distributions $\operatorname{nm}_G$ and $\operatorname{mult}_G$ are defined via the graph-multinomial coefficients. These arise as the coefficients in the power series expansion of powers of so-called multivariate independence polynomial $\Delta_G$ and $\delta_G$  of the graph $G$ (see, e.g., \cite{goldwurm1998clique,levit2005independence,Horn1,Horn2,Mult24}).  We establish several fundamental combinatorial properties and representations of these coefficients; notably, when $G$ is complete, they reduce to the classical multinomial coefficients.
    
	\item The graph negative multinomial \(\operatorname{nm}_G\) and graph multinomial \(\operatorname{mult}_G\) generalize to a decomposable graph setting the negative multinomial and multinomial  models. The latter two are particular cases when the underlying graph is complete. On the other hand, when \(G\) is a completely disconnected graph, \(\operatorname{nm}_G\) is a product of independent univariate negative binomial distributions and \(\operatorname{mult}_G\) is a product of independent univariate binomial distributions. In this sense, each of our models interpolates between the two respective boundary cases, while taking into account the graph structure.
	
	\item There is a significant reduction of the number of unknowns in the parametric models we propose. In a standard non-parametric discrete graphical model with $d$ variables valued in finite sets of sizes $r_1,\ldots,r_d$, the number of unknowns, which are probabilities of cells in $r_1\times \cdots \times r_d$-table, is of order  $\prod_{i=1}^d\,r_i$, while in \(\operatorname{nm}_G\) and  \(\operatorname{mult}_G\) the number of unknowns (parameters) is $d+1$. 
	
	\item Both models possess the global graph Markov property, i.e. if \(A,B,C \subset  V\) is a decomposition of \(G\) (see Definition \ref{def-decomposable-graph}), where the clique \(C\) separates \(A,B\),  and \(N=(N_i)_{i\in V}\) is \(\operatorname{nm}_G\) or  \(\operatorname{mult}_G\) then, conditionally on \((N_i)_{i\in C}\), random vectors \((N_i)_{i\in A}\) and \((N_i)_{i\in B}\) are independent.

	\item 
	The clique marginals of our models are the negative multinomial and multinomial distributions with suitable parameters. In particular, if $G$ is a tree, then our models are pair-wise Markov random fields (which is the usual model for the multivariate count data \cite{Wain, Counts11, Counts20}) with
	\[
	\mathbb{P}(N=n) \propto \exp\left( \sum_{i\in V} \phi_i(n_i) + \sum_{\{i,j\}\in E} \psi(n_{i},n_j)\right),\qquad n\in \mathrm{supp}(N)\subset\mathbb{N}^V,
	\]
	where $(\phi_i)_{i\in V}$ and $\psi$ are the
	potential functions associated with the nodes and edges of the graph $G=(V,E)$, respectively, given by (first cases correspond to the $\operatorname{nm}_G$ model, and the second cases correspond to the $\operatorname{mult}_G$)
	\[
	e^{\phi_i(n)} = \begin{cases}    x_i^{n}  \binom{r+n-1}{n}^{1-\mathrm{deg}_G(i)}, \\
		y_i^{n}  \binom{r}{n}^{1-\mathrm{deg}_G(i)},
	\end{cases}\quad\mbox{and}\quad e^{\psi(n,m)} = \begin{cases}
		\binom{r+n+m-1}{n,m},\\
		\binom{r}{n,m}, 
	\end{cases}
	\]
	where $\deg_G(i)$ denotes the degree of a vertex $i$ in $G$.

	\item 
We provide a probabilistic interpretation for both models. For the graph negative multinomial model, we employ the framework of the non-commutative free quotient monoid and invoke the Cartier--Foata generalization of the MacMahon Master theorem. In contrast, the graph multinomial model admits a relatively simple interpretation via an urn model and we relate it to the hardcore model from statistical physics (see, e.g., \cite{Scott_2005}).

	\item 
	For both discrete models, we develop their Bayesian versions by introducing convenient conjugate prior distributions. The conjugate prior for the  \(\operatorname{nm}_G\) model is called the graph Dirichlet distribution while the conjugate prior for the \(\operatorname{mult}_G\)
	model is called the graph inverted Dirichlet distribution. These priors interpolate between the product of univariate first-kind beta and Dirichlet distributions, and the product of univariate second-kind beta and inverted Dirichlet distributions, respectively. We derive explicit forms of their normalizing constants and  establish  graphical versions of neutrality properties, i.e., an intrinsic independence structure of the (inverted) Dirichlet distribution.
	
	\item We prove that both such parametric discrete graphical Bayesian models are strong hyper Markov. We call them respectively the graph Dirichlet negative multinomial and the graph inverted Dirichlet multinomial. They are natural  generalizations of Dirichlet negative multinomial distribution and Dirichlet multinomial distribution, respectively. These classical 
    compound models have many applications, in particular, in Bayesian document clustering, see e.g. \cite{DC1, DC2, DC3}.
	
	\item For both graph Dirichlet-type distributions we prove characterizations by suitable graphical versions of neutrality property. As a consequence, the graph Dirichlet (resp. graph inverted Dirichlet) distribution is the only distribution that is strong hyper Markov in the graph negative multinomial model (resp. the graph multinomial model). 
\end{enumerate}

The paper is organized as follows: 
In Section \ref{sec:org2cca608}, we introduce necessary definitions and notation for graph theory and graph Markov distributions. In particular, in its second part, we define graph polynomials $\Delta_G$ and $\delta_G$ and graph multinomial coefficients. These polynomials are fundamental objects for the models we introduce and analyze in the paper. In Section \ref{sec:orgce3b2b0}, we define graph negative  multinomial and multinomial graph models and discuss their properties. Section \ref{sec:org7b431d8} provides probabilistic interpretations of these models. We introduce graphical Dirichlet-type laws and analyze their properties in Section \ref{sec:orge36368d}. In Section \ref{six} we introduce our Bayesian models by imposing graph Dirichlet laws as priors for parameters of graph negative multinomial and graph multinomial models; in particular, we prove the conjugacy and hyper Markov properties. Characterizations of graphical Dirichlet distributions by respective versions of the neutrality (independence) property  are given in Section \ref{sec:org9c216d9}, where we also relate these results to  uniqueness of priors for graph (negative) multinomial models with strong hyper Markov properties. All proofs are gathered in Section \ref{sec:org27fcc87}.

\section{Definitions and notations}
\label{sec:org2cca608}
\subsection{Graph theoretic notions}
\label{sec:orgb4d09d8}
Let \(G=(V,E)\) be a finite undirected simple graph with vertex set \(V\) and edge set \(E \subset \{\{i,j\},\ i\ne j,\ i,j\in V\}\). We denote \(i\sim j\) when \(\{i,j\}\in E\).  We say that $i,j\in V$, $i\neq j$, communicate if either $i\sim j$ or there exist distinct $\ell_1,\ldots,\ell_m\in V\setminus\{i,j\}$ with $m\ge 1$, such that with $\ell_0=i$ and $\ell_{m+1}=j$, we have $\ell_k\sim \ell_{k+1}$, $k=0,1,\ldots,m$. The set $i\leftrightarrow j=\{\ell_k,\,k=0,\ldots,m+1 \}$ is called a path connecting $i$ and $j$. A path $i\leftrightarrow i$ with $m\ge 2$, if it exists, is called a cycle.  

We denote by \(\mathfrak{nb}_G(i)\) the set of neighbors of \(i\in V\), and by $\overline{\mathfrak{nb}}_G(i)$ its closure,  defined as,
\begin{equation}
	\label{eq-nbGi}
	\begin{aligned}
		\mathfrak{nb}_G(i)=\{j\in V:\,j\sim i\}\quad \text{ and }\quad\overline{\mathfrak{nb}}_G(i)=\mathfrak{nb}_G(i)\cup\{i\}.
	\end{aligned}
\end{equation}
The subscript $G$ is omitted when there is no ambiguity.  The complement graph of \(G\) is the graph \(G^*=(V,E^*)\), where \(E^*=\{\{i,j\},\ i\ne j,\ i,j\in V,\ \{i,j\}\not\in E\}\).

For \(A\subset V\) we denote by \(G_A=(A,E_A)\) the graph induced in \(G\) by \(A\), i.e., \(\{i,j\}\in E_A\) i.f.f.  \(i,j\in A\) and \(\{i,j\}\in E\).  Note that for any \(A\subset V\), we have $(G^*)_A=(G_A)^*$.

A graph \(G=(V,E)\) is complete if \(E=\{\{i,j\},\ i\ne j, \ i,j\in V\}\); it is denoted by \(K_V\). A subset \(C\subset V\) is a clique in \(G\) if \(G_C=K_C\). By \(\mathcal{C}_G\) we denote the set of cliques in \(G\). In particular, \(\emptyset\in\mathcal{C}_G\) and $\{v\}\in \mathcal C_G$ for all $v\in V$. A clique \(C\) in \(G\) is called maximal if it is not a proper subset of another clique. We denote by \(\mathcal{C}_{G}^+\) the collection of maximal cliques in \(G\).

\begin{Definition} \label{def-decomposable-graph}  Let \(G=(V,E)\) be an undirected graph.  
	\begin{enumerate}[label={\arabic*)}, itemjoin={,}, itemjoin*={, and}]
		\item For non-empty, disjoint subsets \(A,B,C\in V\), we say that	\(C\)  separates \(A,B\) in $G$ if, for all \(a\in A,b\in B\), any path \(a\leftrightarrow b\) intersects  \(C\).
		\item If $A\cup B\cup C=V$ and $C\in\mathcal{C}_{G}$ separates \(A\) and \(B\) in $G$, we say that $(A,B,C)$ is a decomposition of $G$.
		\item A graph $G=(V,E)$ is decomposable if either $G=K_V$ or there exists a decomposition $(A,B,C)$ of $G$ such that \(G_{A\cup C}\) and \(G_{B\cup C}\) are decomposable.
	\end{enumerate}
\end{Definition}
Decomposable graphs coincide with chordal graphs, meaning they have no induced cycles of length greater than 3.
It is known, see e.g., \cite{PeytonBlair}, that a graph \(G\) is decomposable if and only if the set of its maximal cliques \(\mathcal{C}^+_G\) allows a perfect ordering. This means that the elements of \(\mathcal{C}^+_G\) can be labeled \(C_1,\ldots,C_K\) in such a way that for any \(k\in\{2,\ldots,K\}\), there exists \(\ell\in\{1,\ldots,k-1\}\) such that
\begin{equation}
	\label{eq-perfect-ordering-max-cliques}
	\begin{aligned}
		\emptyset\neq S_k:=C_k\bigcap\,\left(\bigcup_{n =1}^{k-1}\,C_{n}\right)\subset C_{\ell}.
	\end{aligned}
\end{equation}
The sets \(S_2,\ldots,S_K\) defined in \eqref{eq-perfect-ordering-max-cliques} are called minimal separators (they need not be distinct). The family \(\mathcal{S}^-_G=\{S\subset V\colon \exists k\in\{2,\ldots,K\}, \,S=S_k\}\) is called the set of minimal separators. For \(S\in \mathcal{S}_G^-\), the number of indices \(k\) such that \(S=S_k\) is called the multiplicity of the minimal separator \(S\) and is denoted by \(\nu_S\). Both \(\mathcal{S}^-_G\) and \(\nu_S\) do not depend on the perfect ordering of \(\mathcal{C}_G^+\), and clearly \(\mathcal{S}^-_G\subset \mathcal{C}_G\).

Given a graph \(G=(V,E)\), if for each edge \(\{i,j\}\in E\), we choose one direction between ordered pairs \((i,j)\) and \((j,i)\), we obtain a directed graph (digraph). Such a digraph is denoted by \(\mathcal{G}=(V,\vec{E})\), where \(\vec{E} \subset \{(i,j),\ i\ne j,\ i,j\in V\}\) are the directed edges. A directed edge \((i,j)\) in \(\vec{E}\) is denoted \(i\to j\). We say that  \(G=(V,E)\) is the skeleton of \(\mathcal{G}=(V,\vec{E})\) if and only if  \(E=\{\{i,j\}:\;\mbox{either }\;i\to j\in \vec{E}\;
\mbox{or }\;j\to i\in\vec{E}\}\). We say that $(i=:\ell_0,\ell_1,\ldots,\ell_m,\ell_{m+1}=j)$ is a directed path from $i$ to $j$ in \(\mathcal{G}=(V,\vec{E})\)  if $\ell_k\to\ell_{k+1}\in\vec{E}$ for $k=0,1,\ldots,m$. In the case $j=i$, we have a directed cycle.

\begin{Definition}\label{def-pa-ch-nd-DAG}
	For a directed graph \(\mathcal G=(V,\vec{E})\), we introduce four functions called parents, children, descendants,  ancestors, denoted as \(\mathfrak{pa},\,\mathfrak{ch},\,\mathfrak{de},\,\mathfrak{an}\,\colon V\to \mathcal{P}(V)\), where \(\mathcal{P}(V)\) is the power set of \(V\), as follows
	\[\mathfrak{pa}(i)=\{j\in V\colon j\to i\},\qquad \mathfrak{ch}(i)=\{j\in V\colon i\to j\},\]
	\[\mathfrak{de}(i)=\bigcup_{k\ge 1}\,\mathfrak{ch}^k(i),\quad \mbox{where}\quad \mathfrak{ch}^1=\mathfrak{ch}\quad\mbox{and}\quad  \mathfrak{ch}^k(i)=\mathfrak{ch}(\mathfrak{ch}^{k-1}(i)),\;\;k\ge 2,\]
	\[\mathfrak{an}(i)=\bigcup_{k\ge 1}\,\mathfrak{pa}^k(i),\quad \mbox{where}\quad \mathfrak{pa}^1=\mathfrak{pa}\quad\mbox{and}\quad  \mathfrak{pa}^k(i)=\mathfrak{pa}(\mathfrak{pa}^{k-1}(i)),\;\;k\ge 2.\]
	Their closures,  \(\overline{\mathfrak{pa}},\,\overline{\mathfrak{ch}},\,\overline{\mathfrak{de}},\,\overline{\mathfrak{an}}\,\colon V\to \mathcal{P}(V)\) are defined by
	\[\overline{\mathfrak{pa}}(i)=\mathfrak{pa}(i)\cup\{i\},\quad \overline{\mathfrak{ch}}(i)=\mathfrak{ch}(i)\cup\{i\},\quad \overline{\mathfrak{de}}(i)=\mathfrak{de}(i)\cup\{i\},\quad \overline{\mathfrak{an}}(i)=\mathfrak{an}(i)\cup\{i\}.\]
	The non-descendants function $\mathfrak{nde}:V\to\mathcal P(V)$ is defined by
	\[\mathfrak{nde}(i) = V \setminus \overline{\mathfrak{de}}(i).\]
	\end{Definition}
It is clear that a directed graph \(\mathcal G\) is uniquely determined by its parent function \(\mathfrak{pa}\). Sometimes, to avoid ambiguity,  we add a subscript $\mathcal G$ to these functions, e.g., we write \(\mathfrak{pa}_{\mathcal{G}}\) instead of \(\mathfrak{pa}\). 

A sink for a directed graph \(\mathcal{G}\) is a vertex with no children.

A directed acyclic graph (DAG) is a digraph with no directed cycles.  A DAG \(\mathcal G\) is said to be moral if, when \(i\) and \(i'\) are both parents of \(j\), there is an (directed) edge between \(i\) and \(i'\). Equivalently, there are no induced directed subgraphs of the form \(i\rightarrow j \leftarrow i'\) (called a v-structure) in \(\mathcal G\); note that induced digraphs are defined similarly to the definition of induced (undirected) graphs. It turns out that all induced subgraphs of a (moral) DAG are (moral) DAGs. The skeleton of a moral DAG is necessarily decomposable,  but there exist DAGs with decomposable skeletons that are not moral.

A vertex \(i\in V\) is called simplicial in an undirected graph \(G\) if its neighbors \(\mathfrak{nb}_G(i)\) form a clique in \(G\).   A permutation of vertices $(v_1,\ldots,v_{|V|}) $ from $V$ is said to be a perfect elimination order if each $v_i$ is simplicial in $G_{\{v_i,\ldots,v_{|V|}\}}$. 
It is known that a graph is decomposable if and only if there exists a perfect elimination order of its vertices. Given a perfect elimination order $(v_1,\ldots,v_{|V|})$ one can easily construct a moral DAG in the following way: for any $1\le i<j\le |V|$, if $v_i\sim v_j$, then set $v_j\to v_i$. Moreover, any moral DAG can be constructed in this way for some perfect elimination order; see \cite{PeytonBlair}.

\subsection{Graph Markov distributions}
\label{sec:org0110ac4}
Let $V$ be a  finite set and  $A\subset V$. For $x=(x_i,\,i\in V)$, we write $x_A=(x_i,i\in A)$. 

Let $G=(V,E)$ be an undirected graph. We say that a discrete random vector \(X=(X_i)_{i\in V}\) satisfies the global Markov property with respect to $G$ (we will also say that the distribution of $X$ is $G$-Markov) if for any $A,B,S\subset V$ such that $S$ separates $A$ and $B$ in $G$, random vectors $X_A$ and $X_B$ are conditionally independent given $X_S$. 

If $G$ is decomposable, then the global Markov property is equivalent to the factorization 
\begin{equation}
	\mathbb{P}(X=x) \label{fact:undirected}
	= \frac{\prod_{C\in \mathcal{C}^+_G} \mathbb{P}(X_{C}=x_C)}{\prod_{S\in \mathcal{S}^-_G} \mathbb{P}(X_{S}=x_S)^{\nu_S}},\quad x\in\mathrm{supp}\,\mathbb P_X, 
\end{equation}
where $\mathbb P_X$ is the distribution of $X$.

We say that a distribution of $X$ is Markov with respect to a DAG $\mathcal G$ (defined by a parent function $\mathfrak{pa}$) if
\begin{equation}
	\mathbb{P}(X=x)=\prod_{i\in V} \mathbb{P}(X_i=x_i| X_{\mathfrak{pa}(i)}=x_{\mathfrak{pa}(i)}),\quad x\in\mathrm{supp}\,\mathbb P_X, \label{fact:directed}
\end{equation}
If \(\mathcal G\) is a moral DAG with  skeleton \(G=(V,E)\), then according to Lemma 3.21 of \cite{lauritzen1996graphical}, \(X\) is Markov with respect to \(\mathcal G\) if and only if \(X\) is Markov with respect to any moral DAG with skeleton $G$ if and only if  \(X\) is global Markov with respect to \(G\). 
Alternatively, factorizations \eqref{fact:undirected} for a decomposable \(G\) and \eqref{fact:directed} for a/every moral DAG \(\mathcal{G}\) with skeleton \(G\) are equivalent; see Proposition 3.28 of \cite{lauritzen1996graphical}.
\subsection{Graph polynomials}
\label{sec:orge67f81b}
For \(A\subset V\), we define \(\pi_A\colon \mathbb{R}^{V}\to \mathbb{R}\) by
\[\pi_A(x) = \prod_{i\in A} x_i,\qquad x\in \mathbb{R}^{V}.\]

\begin{Definition}[Clique polynomials]
	For a finite simple graph \(G=(V,E)\) define  two polynomials of $|V|$ variables,   \(\Delta_G\) and \(\delta_G\colon \mathbb{R}^{V}\to \mathbb{R}\), by
	\begin{equation}
		\label{eq-clique-polynomial-DeltaG}
		\begin{aligned}
			\Delta_G=\sum_{C\in \mathcal{C}_{G^*}}(-1)^{|C|} \pi_C
		\end{aligned}
	\end{equation}
	and
	\begin{equation}
		\label{eq-clique-polynomial-deltaG}
		\begin{aligned}
			\delta_G=\sum_{C\in \mathcal{C}_{G^*}}\pi_C,
		\end{aligned}
	\end{equation}
	where \(|C|\) denotes the cardinality of \(C\), and \(\mathcal{C}_{G^*}\) is the collection of cliques of the complement graph \(G^*\).
	\label{def-click-polynomial}
\end{Definition}

The polynomials $\Delta_G$ and $\delta_G$ are the so-called multivariate independence polynomials (or clique polynomials of the complement graph $G^*$), \cite{goldwurm1998clique,levit2005independence, Horn1, Horn2, Mult24}. Below, we list some special properties of $\Delta_G$ and $\delta_G$, which will be used in the sequel. 

\begin{Lemma}\label{lem-properties-of-clique-polynomials}
	Let \(G=(V,E)\) be a finite simple graph.
	\begin{enumerate}[label={\arabic*)}, itemjoin={,}, itemjoin*={, and}]
		\item If \(A \subset  V\), then
		\begin{equation}
			\label{eq-DeltaA}
			\begin{aligned}
				\Delta_G(x)_{|x_{V \setminus A}=0} = \Delta_{G_A}(x_A)=:\Delta_A(x),\  x_A \in \mathbb{R}^A,
			\end{aligned}
		\end{equation}
		where \(\Delta_G(x)_{|x_{V \setminus A}=0}\)  denotes \(\Delta_G\) evaluated at $x\in\mathbb R^V$  with coordinates \(x_i=0,\ \forall\, i\in V\setminus A\).
		\item  If \(A,B \subset  V\) are disconnected in $G$, i.e., there is no edge with one of its endpoints in \(A\) and the other in \(B\), then
		\begin{equation}
			\label{eq-ilocz}
			\begin{aligned}
				\Delta_{A\cup B}=\Delta_A \Delta_B.
			\end{aligned}
		\end{equation}
		\item For any \(C\in \mathcal{C}_G\),
		\begin{equation}
			\label{eq-Delta-G-C}
			\begin{aligned}
				\Delta_G=\Delta_{V\setminus C}-\sum_{i\in C} \pi_i \Delta_{\mathfrak{nb}_{G^*}(i)}\quad\mbox{and}\quad \delta_G=\delta_{V \setminus  C}+\sum_{i\in C} \pi_i \delta_{\mathfrak{nb}_{G^*}(i)},
			\end{aligned}
		\end{equation}
		where \(\pi_i:=\pi_{\{i\}}\).
        
		In particular, for any $i\in V$ 
		\begin{equation}
			\label{eq-Delta_GV-v}
			\begin{aligned}
				\Delta_G=\Delta_{V \setminus \{i\}}-\pi_i \Delta_{\mathfrak{nb}_{G^*}(i)}\quad \mbox{and}\quad \delta_G=\delta_{V \setminus \{i\}}+\pi_i \delta_{\mathfrak{nb}_{G^*}(i)}.
			\end{aligned}
		\end{equation}
		\item For a simplicial vertex $i_0$ in $G$, denote $\tilde V=V\setminus\{i_0\}$ and $\tilde G=G_{\tilde V}$. Define $\tilde x, \check{x}\in \mathbb{R}^{\tilde V}$ by $\tilde x_i = x_i/(1-x_{i_0})$, $\check{x}_i = x_i/(1+x_{i_0})$ for $i\in \mathfrak{nb}_G(i_0)$ and $\tilde x_i=x_i = \check{x}_i $ otherwise. Then
		\begin{align}\label{eq:Deltax}
			\Delta_G(x)=(1-x_{i_0})\Delta_{\tilde G}\left(\tilde x\right)\quad \mbox{and}\quad \delta_G(x)=(1+x_{i_0})\delta_{\tilde G}\left(\check x\right).
		\end{align}
	\end{enumerate}
\end{Lemma}

\begin{Definition}
	For a decomposable graph \(G=(V,E)\), define
	\begin{equation}
		\label{eq-support-MG}
		\begin{aligned}
			M_G=\left\{x\in(0,\infty)^V:\ \ \Delta_A(x) > 0\,\ \forall A \subset  V\right\}.
		\end{aligned}
	\end{equation}
	\label{def-support-of-graph-dirichlet}
\end{Definition}
In the sequel, we will only consider \(\Delta_G\) as a function defined on \(M_G\), and \(\delta_G\) as a function defined on \((0,\infty)^V\); that is, variable \(x\in M_G\) for \(\Delta_G(x)\) and  variable \(y\in (0,\infty)^V\) for \(\delta_G(y)\). 
\begin{Remark}
	\label{rmk-remark-on-the-support-set}
	We note that 
	\[\{x\in (0,\infty)^V\colon \sum_{i\in V}x_i < 1\}\subset M_G\subset(0,1)^V.\]
	Moreover, in the definition of \(M_G\), many of the conditions \(\Delta_A(x)>0\), \(A\subset V\), are redundant. For example, in the case of complete graph \(G=K_V\), all these conditions follow from \(\Delta_G(x)>0\). However, \(\Delta_G(x)>0\) does not imply that \(\Delta_A(x)>0\) for every \(A\subset V\)  for a general (decomposable) graph; for instance, when \(G\) is a graph with two maximal cliques \(\{1,2,3\}\) and \(\{3,4,5\}\). The set \(M_G\) is, in general, a complicated manifold, but under a proper change of  coordinates,  see  Definition \ref{def-factorizing-coordinates} and Lemma \ref{lem-clique-poly-new-coord} below,  it becomes the unit cube.
\end{Remark}
We define new coordinate systems related to moral DAGs of a given skeleton. 
\begin{Definition}[Factorizing coordinates]	\label{def-factorizing-coordinates}
	Let \(\mathcal{G}\)  be a moral DAG with skeleton \(G=(V,E)\).
	\begin{enumerate}[label={\arabic*)}, itemjoin={,}, itemjoin*={, and}]
		\item Define \(u^{\mathcal{G}}\colon M_G \to \mathbb{R}^V\) by
		\begin{equation}
			\label{eq-u-calG}
			\begin{aligned}
				u^{\mathcal{G}}_i:=u^{\mathcal{G}}_i(x)=1- \frac{\Delta_{\overline{\mathfrak{de}}(i)}(x)}{\Delta_{\mathfrak{de}(i)}(x)}.
			\end{aligned}
		\end{equation}
		\item Define \(w^{\mathcal{G}}\colon (0,\infty)^V\to \mathbb{R}^V\) by
		\begin{equation}
			\label{eq-w-calG}
			\begin{aligned}
				w^{\mathcal{G}}_i:=w^{\mathcal{G}}_i(y)= \frac{\delta_{\overline{\mathfrak{de}}(i)}(y)}{\delta_{\mathfrak{de}(i)}(y)}-1.
			\end{aligned}
		\end{equation}
	\end{enumerate}
\end{Definition}
In Lemma \ref{lem-clique-poly-new-coord} we show that \(u^{\mathcal{G}}\colon M_G \to (0,1)^V\), and \(w^{\mathcal{G}}\colon (0,\infty)^V\to (0,\infty)^V\), and that they are diffeomorphisms.
\begin{Remark}
	\label{rmk-other-def-u-w}
	 Let  $\mathcal G$ be a  moral DAG with skeleton $G=(V,E)$. Then for any $i\in V$, we have $\mathfrak{nb}_{G^*_{\overline{\mathfrak{de}}(i)}}(i)=\mathfrak{de}(i)\setminus\mathfrak{ch}(i)$. Consequently, by \eqref{eq-Delta_GV-v} with $G=G_{\overline{\mathfrak{de}}(i)}$, we get  
	\[u^{\mathcal{G}}_i=x_i \frac{\Delta_{\mathfrak{de}(i)\setminus \mathfrak{ch}(i)}(x)}{\Delta_{\mathfrak{de}(i)}(x)}\qquad \text{ and }\qquad w^{\mathcal{G}}_i=y_i \frac{\delta_{\mathfrak{de}(i)\setminus \mathfrak{ch}(i)}(y)}{\delta_{\mathfrak{de}(i)}(y)}.\]
\end{Remark}

\begin{Lemma}	\label{lem-clique-poly-new-coord}
	Let \(\mathcal{G}\) be a moral DAG with skeleton \(G=(V,E)\).  Let $u^{\mathcal{G}}$ and $w^{\mathcal{G}}$ be the functions introduced in Definition \ref{def-factorizing-coordinates}. Then
	\begin{enumerate}[label={\arabic*)}, itemjoin={,}, itemjoin*={, and}]
		\item \begin{equation}
			\label{eq-deltaG}
			\begin{aligned}
				\Delta_G=\prod_{i\in V} (1-u^{\mathcal{G}}_i)\qquad \text{ and }\qquad \delta_G=\prod_{i\in V}(1+w^{\mathcal{G}}_i);
			\end{aligned}
		\end{equation}
		\item  \(u^{\mathcal{G}}:M_G\to (0,1)^V\) is a diffeomorphism, and its inverse map is given by \(x^{\mathcal{G}}: (0,1)^V\to M_G\), where
		\begin{equation}
			\label{eq-x-in-u}
			\begin{aligned}
				x^{\mathcal{G}}_i(u)=u_i \prod_{j\in \mathfrak{ch}(i)}(1-u_j),\ i\in V;
			\end{aligned}
		\end{equation}
		\item \(w^{\mathcal{G}}\colon (0,\infty)^V\to (0,\infty)^V\) is a diffeomorphism, and its inverse map is given by \(y^{\mathcal{G}}\colon (0,\infty)^V\to (0,\infty)^V\), where
		\begin{equation}
			\label{eq-y-in-w}
			\begin{aligned}
				y^{\mathcal{G}}_i(w)=w_i \prod_{j\in \mathfrak{ch}(i)}(1+w_j),\ i\in V;
			\end{aligned}
		\end{equation}
		\item in particular,
		\begin{equation}
			\label{eq-ui-as-xu}
			\begin{aligned}
				u_i^{\mathcal{G}}=\frac{x_i}{\prod_{j\in \mathfrak{ch}(i)} (1-u_j^{\mathcal{G}})}\qquad \text{ and } \qquad  w_i^{\mathcal{G}}=\frac{y_i}{\prod_{j\in \mathfrak{ch}(i)}(1+w_j^{\mathcal{G}})},\quad i\in V.
			\end{aligned}
		\end{equation}
	\end{enumerate}
\end{Lemma}
\begin{expl}	\label{expl-1-2-3}
	Consider the chain \(G=1-2-3\). There are three moral DAGs with skeleton \(G\), namely
	\[\mathcal{G}_1 = 1\rightarrow2\rightarrow3,\quad 	\mathcal{G}_2 = 1\leftarrow2\leftarrow3,\quad \mathcal{G}_3=1\leftarrow 2\rightarrow 3.\]
	The corresponding factorizing coordinates, defined in \eqref{eq-u-calG},  are
	\begin{align*}
		(u^{\mathcal{G}_1}_1,u^{\mathcal{G}_1}_2,u^{\mathcal{G}_1}_3)&=  \left(\frac{x_1(1-x_3)}{1-x_2-x_3}, \frac{x_2}{1-x_3},x_3\right), \\
		(u^{\mathcal{G}_2}_1,u^{\mathcal{G}_2}_2,u^{\mathcal{G}_2}_3) &= \left(x_1, \frac{x_2}{1-x_1}, \frac{x_3(1-x_1)}{1-x_1-x_2} \right),\\
		(u^{\mathcal{G}_3}_1,u^{\mathcal{G}_3}_2,u^{\mathcal{G}_3}_3) & = \left(x_1, \frac{x_2}{(1-x_1)(1-x_3)}, x_3\right).
	\end{align*}
	We have
	\begin{align*}
		\Delta_G(x_1,x_2,x_3)  & = 1-x_1-x_2-x_3+x_1x_3  \\
		& = \left(1-\frac{x_1(1-x_3)}{1-x_2-x_3}\right)\left(1-\frac{x_2}{1-x_3}\right)(1-x_3)=\prod_{i=1}^3\,(1-u_i^{\mathcal G_1}) \\
		& = (1-x_1)\left(1-\frac{x_2}{1-x_1}\right)\left(1-\frac{x_3(1-x_1)}{1-x_1-x_2}\right)=\prod_{i=1}^3\,(1-u_i^{\mathcal G_2})\\
		& = 
		(1-x_1)\left(1-\frac{x_2}{(1-x_1)(1-x_3)}\right)(1-x_3)=\prod_{i=1}^3\,(1-u_i^{\mathcal G_3}).
	\end{align*}
\end{expl}

\subsection{Graph multinomial coefficients}\label{sec:coefficients}
For \(\alpha\in \mathbb{R}^V\) and \(A \subset  V\), we denote
\begin{equation*}
	\begin{aligned}
		|\alpha_A|=\sum_{i\in A}\alpha_i.
	\end{aligned}
\end{equation*}
If $A=V$, then we will sometimes drop the subscript $V$ and write $|\alpha|=\sum_{i\in V} \alpha_i$. 
\begin{Definition}\label{Ddel}
	Assume that $G$ is a finite simple graph.
    \begin{enumerate}[label={\arabic*)}, itemjoin={,}, itemjoin*={, and}]
    \item For any $r\in\mathbb{N}_+=\{1,2,\ldots\}$, we define the graph-multinomial coefficients  of the first type $\left\{\gbinom{r}{n}{\,\,G}\colon  n\in \mathbb{N}^V \right\}$ by 
	\[	\delta_G(y) ^r= \sum_{n \in\mathbb{N}^V} \gbinom{r}{n}{G} y^n,	\]
	where $y^n = \prod_{i\in V} y_i^{n_i}$. 
	\item For any real \(r>0\), we define the graph-multinomial coefficients of the second type $\left\{\gbinomm{|n|+r-1}{n}{G}\colon n\in \mathbb{N}^V \right\}$ by the formal power series expansion
	\[	\Delta_G(x)^{-r}= \sum_{n \in\mathbb{N}^V} \gbinomm{|n|+r-1}{n}{G} x^n.	\]
    \end{enumerate}
\end{Definition}
As we only consider finite graph, \cite{Dobrushin_1996} readily implies that \(\log \Delta_G(x)\) is absolutely convergent for \(\Vert x\Vert_\infty\) small enough, and thus, the above expansion is well defined. The same expansions were considered in \cite{Horn1,Horn2}, where the authors characterized  decomposable graphs by the inverse of their independence polynomials being Horn hypergeometric series. 

We note that when \(G\) is a complete graph, the above definitions are consistent with the definitions of the multinomial coefficients and the negative multinomial coefficients. For general \(G\), note also that $\gbinom{-r}{n}{\,G}=(-1)^{|n|} \gbinomm{|n|+r-1}{n}{G}$ for $r>0$.

In Section \ref{sec:org7b431d8}, we will see that the coefficients \(\gbinom{1}{n}{\,\,G}\) and \(\gbinomm{|n|}{n}{G}\) (see \eqref{eq:nn}) naturally arise in some combinatorial schemes for any graph $G$, however, our interest is mostly in the decomposable case. 

Since $\delta_G(y)^r$ is a multivariate polynomial for any $r\in \mathbb{N}_+$, the support of the graph-multinomial coefficients of the first type is a finite set. The following result describes their support precisely. 
\begin{Lemma}\label{lem:NGr}
Let $G$ be a finite simple graph and $r\in\mathbb{N}_+$. Then, $\gbinom{r}{n}{\,G}\neq 0$ if and only if $n\in \mathbb{N}_{G,r}$, where 
	\begin{align}\label{ngr}
		\mathbb{N}_{G,r}:=\{n\in\mathbb{N}^V\colon \max_{C\in\mathcal{C}_G^+} |n_C|\leq r\}.
	\end{align} 
\end{Lemma}
By comparing the coefficients of $\delta_G(x)^{r_1+r_2}$ and $\delta_G(x)^{r_1} \delta_G(x)^{r_2}$ (and similarly for $\Delta_G$), we obtain the following result.
\begin{Corollary}\label{cor:sums}
Let $r_i\in \mathbb{N}_+$ and $s_i>0$ for $i=1,2$. Then, 
\begin{multline*}
\gbinom{r_1+r_2}{n}{G}	 = \sum_{\substack{k\in\mathbb{N}_{G,r_1}\\ n-k\in\mathbb{N}_{G,r_2}}} \gbinom{r_1}{k}{G}\gbinom{r_2}{n-k}{G}\quad
\mbox{and}\\ \gbinomm{|n|+s_1+s_2-1}{n}{G}	 = \sum_{\substack{k\in\mathbb{N}^V\\n-k\in\mathbb{N}^V}} \gbinomm{|k|+s_1-1}{k}{G}\gbinomm{|n-k|+s_2-1}{n-k}{G}.
\end{multline*}
\end{Corollary}
Moreover, by Lemma \ref{lem-properties-of-clique-polynomials} 1), we obtain
\[
\gbinom{r}{n}{G}\Big|_{n_{A^c}=0}=\gbinom{r}{n_A}{G_A}\quad\mbox{and}\quad \gbinomm{|n|+r-1}{n}{G}\Big|_{n_{A^c}=0}=\gbinomm{|n_A|+r-1}{n_A}{G_A}.
\]
In the binomial coefficient $\binom{a}{b}$, we allow $a$ to be non-integer. In this case, we have 
\begin{align*}
	\binom{a}{b} = \frac{\Gamma(a+1)}{b!\,\Gamma(a-b+1)}. 
\end{align*}
We adopt the same convention for multinomial coefficients.
\begin{Lemma}\label{lem:constant-recustion}
Let $r\in \mathbb{N}_+$ and $s>0$.  If $i_0\in V$ is simplicial in $G=(V,E)$, then for $n\in\mathbb{N}^V$,
	\begin{multline*}
	\gbinom{r}{n}{G} = \gbinom{r}{n_{\tilde V}}{\tilde G} \binom{r-|n_{\mathfrak{nb}_G(i_0)}|}{n_{i_0}}\quad\mbox{ and }\\	\gbinomm{|n|+s-1}{n}{G} = \gbinomm{|n_{\tilde{V}}|+s-1}{n_{\tilde V}}{\tilde G} \binom{n_{i_0}+|n_{\mathfrak{nb}_G(i_0)}|+s-1}{n_{i_0}},
	\end{multline*}
	where $\tilde{V}=V\setminus\{i_0\}$ and $\tilde{G} = G_{\tilde{V}}$. 
\end{Lemma}

The above lemma allows us to find closed-form formulas for graph-multinomial coefficients for any decomposable graph. 
\begin{Lemma}\label{lem:explicit-constants}
Assume that $G=(V,E)$ is decomposable. Let $r\in \mathbb{N}_+$ and $s>0$. 
\begin{enumerate}[label={\arabic*)}, itemjoin={,}, itemjoin*={, and}]
\item 	Then, for $n\in \mathbb{N}^V$, 
	\begin{align*}
\gbinom{r}{n}{G} &= \frac{\prod_{C\in \mathcal{C}_G^+}\binom{r}{ n_C}}{\prod_{S\in \mathcal{S}_G^-}\binom{r}{n_S}^{\nu_S}}\qquad\mbox{ and }\qquad 
	\gbinomm{|n|+s-1}{n}{G} = \frac{\prod_{C\in \mathcal{C}_G^+} \binom{|n_C|+s-1}{n_C}}{\prod_{S\in \mathcal{S}_G^-}\binom{|n_S|+s-1}{n_S}^{\nu_S}}.
	\end{align*}
\item Let $\mathcal{G}$, with a parent function $\mathfrak{pa}$,  be a moral DAG with a skeleton $G$. Then,
	\begin{align*}
	\gbinom{r}{n}{G} = \prod_{i\in V} \binom{r-|n_{\mathfrak{pa}(i)}| }{n_i}\qquad\mbox{and}\qquad 	\gbinomm{|n|+s-1}{n}{G} = \prod_{i\in V} \binom{n_i+|n_{\mathfrak{pa}(i)}|  +s-1}{n_i}.
	\end{align*}
    \end{enumerate}
\end{Lemma}
By the above result, we see that the graph-multinomial coefficients are always nonnegative. 

\begin{Remark}
Define the ascending/descending Pochhammer symbols by
$$
(a)^{(k)}=a(a+1)\ldots(a+k-1),\quad (a)_{(k)}:=a(a-1)\ldots (a-k+1),\quad k\ge 1,\quad (a)^{(0)}=(a)_{(0)}=1.
$$
One can prove that 
	\begin{align*}
		\gbinom{r}{n}{G} = \frac{ \prod_{C\in \mathcal{C}_G^+}(r)_{\left(|n_C|\right)}}{n!\,\prod_{S\in \mathcal{S}_G^-}\left[(r)_{\left(|n_S|\right)}\right]^{\nu_S}}\qquad\mbox{and}\qquad 
		\gbinomm{|n|+s-1}{n}{G} = \frac{\prod_{C\in \mathcal{C}_G^+}\, (s)^{(|n_C|)}}{n!\, \prod_{S\in \mathcal{S}_G^-}\left[(s)^{(|n_S|)}\right]^{\nu_S}},
	\end{align*}
	where $r\in \mathbb{N}_+$, $s>0$ and $n!=\prod_{i\in V}\,n_i!$. 
\end{Remark}

\begin{Remark}\label{rem:Cd}
If $G$ is non-decomposable, the graph-multinomial coefficients can still be obtained, though they generally do not have compact formulas.

In \cite{Horn2}, the authors interpreted Reed's results \cite{Reed} on chromatic polynomials in terms of independence polynomials. 

Graph-multinomial coefficients for cycle graphs were considered in \cite{Horn2}: let $C_d$ be the cycle graph of length $d\geq 3$. For $r<0$, define
\[
\gbinom{r}{n}{C_d} = (-1)^{|n|}\gbinomm{|n|+r-1}{n}{C_d}. 
\]
Then, \cite[Theorem 4]{Horn2}, for $r\in \mathbb{N}\cup(-\infty,0)$, gives
\begin{align*}
\gbinom{r}{n}{C_d} = \frac{1}{n!}  \left(
\sum_{k=0}^d \, (-1)^{kd}
v_k(r) \prod_{i=1}^d \frac{(n_i)_{(k)}}{(r)_{(n_i+k)}}
\right)\, \prod_{i=1}^d (r)_{(n_i+n_{i+1})}
\end{align*}
for $n\in\mathbb{N}^V$, where $n_{d+1} = n_1$, $n! = \prod_{i=1}^n n_i!$, and $v_k(r) = \binom{r}{k}-\binom{r}{k-1}$ with the convention that for $r<0$, $\binom{r}{k}=(-1)^k \binom{-r+k-1}{k}$, $\binom{a}{-1}=0$. 
In particular, if $r=-1$, then for $n\in\mathbb{N}^V$, we obtain
\begin{align*}
\gbinomm{|n|}{n}{C_d} = \frac{1}{n!} \left(
\sum_{k=0}^d \,v_k(-1) \prod_{i=1}^d \frac{(n_i)_{(k)}}{(n_i+k)!}
\right)\,\prod_{i=1}^d (n_i+n_{i+1})! ,
\end{align*}
where $v_k(-1) = 2(-1)^k$ for $k>0$ and $v_0(-1)=1$. Although not immediately apparent from this formulation, the above formula simplifies to $\binom{n_1 + n_2 + n_3}{n_1, n_2, n_3}$ for $d = 3$, which corresponds to the complete (and hence decomposable) graph on 3 vertices.
\end{Remark}

\begin{Lemma}\label{lem:MGseries}
	Assume that $G$ is decomposable. For real $r>0$ and $x\in(0,\infty)^V$, the series  
	\begin{align}\label{eq:DeltaG}
        \sum_{n \in\mathbb{N}^V} \gbinomm{|n|+r-1}{n}{G} x^n
	\end{align}
	converges if and only if $x\in M_G$. In particular, if $x\in (0,\infty)^V\setminus M_G$, the series diverges. 
\end{Lemma}

\section{Negative multinomial and multinomial graphical models}
\label{sec:orgce3b2b0}

We define two discrete distributions arising from the multivariate power series expansions of $\delta_G^r$ and $\Delta_G^{-r}$. In Section \ref{sec:char-markov}, we show that these distributions naturally arise as the unique families of distributions with $G$-Markov properties having multinomial or negative multinomial clique marginals. If the graph $G$ is non-decomposable, this property no longer holds, although our formulas below still define valid probability distributions. Specifically, as shown in Remark \ref{rem:non-dec}, if $G=C_4$ is the cycle of length $4$, the $G$-multinomial distribution $\operatorname{mult}_G(r,x)$ is $G$-Markov when $r=1$, but it is not $G$-Markov when $r=2$. For these reasons, we restrict our definitions to decomposable graphs. 

\begin{Definition}	\label{def-graph-multinomial-and-negative-multinomial-distributions}
	Let \(G=(V,E)\) be a finite decomposable graph.
	\begin{enumerate}[label={\arabic*)}, itemjoin={,}, itemjoin*={, and}]
		\item A random vector \(N=(N_i)_{i\in V}\) has the \(G\)-negative multinomial distribution with parameters \(r>0\)  and \(x\in M_G\) if
		\begin{equation}
			\label{eq-PNn}
			\begin{aligned}
				\mathbb{P}(N=n)=\gbinomm{|n|+r-1}{n}{G}\,x^n\,\Delta^r_G(x),\quad n\in \mathbb{N}^V,
			\end{aligned}
		\end{equation}
		We write \(N\sim \operatorname{nm}_G(r,x)\). 
		\item A random vector \(N=(N_i)_{i\in V}\) has the \(G\)-multinomial distribution with parameters  \(r\in \mathbb{N}_+\), \(y\in (0,\infty)^V\), if 
		\begin{equation}
			\label{eq-PNn2}
			\begin{aligned}
				\mathbb{P}(N=n)=\gbinom{r}{n}{G}\,y^n\, \delta_G(y)^{-r},\quad n\in \mathbb N_{G,r}.
			\end{aligned}
		\end{equation}
We write \(N\sim \operatorname{mult}_G(r,y)\).
	\end{enumerate}
\end{Definition}
Since the graph-multinomial coefficients are non-negative and by Lemma \ref{lem:MGseries},  the above formulas define discrete probability distributions. 
\begin{Remark}
	Let us single out two special cases with $r=1$, being, respectively, graphical versions of
	\begin{itemize} \item the geometric law, $\operatorname{nm}_G(1,x)$:
		$$
		\mathbb P(N=n)=\gbinomm{|n|}{n}{G} x^n\,\Delta_G(x),\quad n\in\mathbb N^V,
		$$
		\item the Bernoulli law, $\operatorname{mult}_G(1,y)$:
		$$
		\mathbb P(N=n)=\tfrac{y^n}{\delta_G(y)},\quad \max_{c\in\mathcal C_G^+}\,|n_C|\le 1.
		$$
	\end{itemize}
\end{Remark}

\begin{expl}[Basic examples] $\,$
	\begin{itemize} \item Let \(G\) be a complete graph with vertex set \(V\), i.e., \(G=K_V\).
		\begin{enumerate}[label={\arabic*)}, itemjoin={,}, itemjoin*={, and}]
			\item If \(N\sim \operatorname{nm}_G(r,x)\) with  $r>0$ and  \(x\in M_G=\left\{x\in (0,\infty)^V\colon |x_V| < 1\right\}\), then
			\begin{equation}
				\label{eq-cnm}
				\begin{aligned}
					\mathbb{P}(N=n)=\binom{|n|+r-1 }{ n} (1-|x_V|)^r x^n,\quad n\in \mathbb{N}^V.
				\end{aligned}
			\end{equation}
			This is the negative multinomial distribution; denoted by \(\operatorname{nm}(r,x)\) or \(\operatorname{nm}_{K_V}(r,x)\).
			\item Similarly, if \(N\sim \operatorname{mult}_G(r,y)\) with $r\in\mathbb{N}_+$ and \(y\in (0,\infty)^V\), then
			\begin{equation}\label{multcl}\mathbb{P}(N=n)=\binom{r}{n} \left(\frac{1}{1+|y_V|}\right)^{r-|n|}\prod_{i\in V} \left( \frac{y_i}{1+|y_V|} \right)^{n_i},\quad n\in \mathbb{N}^V \mbox{ such that  }|n_V| \le r.\end{equation}
			This is the multinomial distribution with \(r\) trials and event probabilities \(p_i=y_i/(1+|y_V|)\) for \(i\in V\) and \(p_0=(1+|y_V|)^{-1}\); we denote it by \(\operatorname{mult}(r,y)\) or \(\operatorname{mult}_{K_V}(r,y)\).
		\end{enumerate}
		\item Let \(G\) be a completely disconnected graph, i.e. \(G=K_V^*\).
		\begin{enumerate}[label={\arabic*)}, itemjoin={,}, itemjoin*={, and}]
			\item If \(N\sim \operatorname{nm}_G(r,x)\) with $r>0$ and \(x\in M_G=(0,1)^V\), then
			\begin{equation}
				\label{eq-cnm2}
				\begin{aligned}
					\mathbb{P}(N=n)=\prod_{i\in V}\,\binom{r+n_i-1}{ n_i} x_i^{n_i} (1-x_i)^r,\quad\ n\in \mathbb{N}^V,
				\end{aligned}
			\end{equation}
			i.e., \(N\) is a random vector of independent univariate negative binomial coordinates.
			\item Similarly, if \(N\sim \operatorname{mult}_G(r,y)\) with $r\in \mathbb N_+$ and \(y\in (0,\infty)^V\), then
			\begin{equation}\label{multcl2}\mathbb{P}(N=n)=\prod_{i\in V}\,\binom{r}{n_i}\,p_i^{n_i}\,(1-p_i)^{r-n_i},\quad n\in\{0,1,\ldots,r\}^V,\end{equation}
			where $p_i=y_i/(1+y_i)$, $i\in V$, i.e., \(N\) is a random vector of independent univariate  binomial coordinates.
		\end{enumerate}
	\end{itemize}
\end{expl}
\begin{expl}\label{ex-1-2-3} 	\label{expl-basic-examples} Let \(G=(V,E)\) be the graph with \(V=\{1,2,3\}\) and \(E=\{\{1,2\},\{2,3\}\}\).
	\begin{enumerate}[label={\arabic*)}, itemjoin={,}, itemjoin*={, and}]
		\item If \(N\sim \operatorname{nm}_G(r,x)\) with $r>0$ and   \(x\in M_G=\{(x_1,x_2,x_3)\in(0,1)^3 \colon  x_2 < (1-x_1)(1-x_3)\}\), then for $n\in\mathbb N^3$,
		\begin{equation}
			\label{eq-1-2-3-nmG}
			\begin{aligned}
				\mathbb{P}(N=n) = \frac{(r)^{(n_1+n_2)} \,(r)^{(n_2+n_3)}}{(r)^{(n_2)}\,n_1!\,n_2!\,n_3!}\,  x_1^{n_1}\,x_2^{n_2}\,x_3^{n_3}(1-x_1-x_2-x_3+x_1x_3)^r.
			\end{aligned}
		\end{equation}
		\item If \(N\sim \operatorname{mult}_G(r,y)\) with $r\in\mathbb N_+$  and \(y\in (0,\infty)^V\) then for \(n\in\mathbb N^3\) such that and \(\max(n_1+n_2,n_2+n_3)\leq r\),
		\begin{equation}
			\label{eq-1-2-3-multG}
			\begin{aligned}
				\mathbb{P}(N=n) = \frac{(r)_{(n_1+n_2)} (r)_{(n_2+n_3)}}{(r)_{(n_2)}\,n_1! \,n_2!\,n_3!}\,  y_1^{n_1}\,y_2^{n_2}\,y_3^{n_3}\,(1+y_1+y_2+y_3+y_1y_3)^{-r}.
			\end{aligned}
		\end{equation}
	\end{enumerate}
\end{expl}

\subsection{Basic properties}\label{sec:basic}
\begin{Proposition}[Moment generating functions] 	\label{prop-moment-generating-functions} Let $G=(V,E)$ be a decomposable graph.
	\begin{enumerate}[label={\arabic*)}, itemjoin={,}, itemjoin*={, and}] 
		\item Let \(N\sim \operatorname{nm}_G(r,x)\). Its moment generating function \(L_N\) is
		\[L_N(\theta):=\mathbb{E}\left( e^{\left< \theta,N \right>} \right)=\left( \frac{\Delta_G(x)}{\Delta_G(xe^{\theta})} \right)^r,\quad \theta\in \Theta,\]
		where \(\Theta \subset \mathbb{R}^V\) contains a neighborhood of \(0\in \mathbb{R}^V\), \(xe^{\theta}=(x_i e^{\theta_i})_{i\in V}\) and \(\left< \theta,N \right>=\sum_{i\in V}\theta_i N_i\).
		\item Let \(N\sim \operatorname{mult}_G(r,y)\). Its moment generating function \(L_N\) is
		\[L_N(\theta)=\left( \frac{\delta_G(ye^{\theta})}{\delta_G(y)} \right)^r,\quad \theta\in \mathbb{R}^V,\]
		where \(y e^{\theta}=(y_ie^{\theta_i})_{i\in V}\).
	\end{enumerate}
\end{Proposition}
The proof is immediate. The following convolution properties are straightforward consequences of Proposition \ref{prop-moment-generating-functions}. 
\begin{Corollary}
	\label{coro-sum-of-nm} Let $G=(V,E)$ be a decomposable graph.
	\begin{enumerate}[label={\arabic*)}, itemjoin={,}, itemjoin*={, and}]
		\item If \(N_1\sim \operatorname{nm}_G(r_1,x)\) and \(N_2\sim \operatorname{nm}_G(r_2,x)\) are independent, then
		\[N_1+N_2 \sim \operatorname{nm}_G\left(r_1+r_2,x\right).\]
		\item If \(N_1\sim \operatorname{mult}_G(r_1,\,y)\) and \(N_2\sim \operatorname{mult}_G(r_2,\,y)\) are independent, then
		\[N_1+N_2 \sim \operatorname{mult}_G\left(r_1+r_2,\,y\right).\]
	\end{enumerate}
\end{Corollary}
The marginals of \(\operatorname{nm}_G\) (resp. \(\operatorname{mult}_G\)) on cliques in \(G\) have negative multinomial (resp. multinomial) distribution.
\begin{Theorem}\label{thm:marginal}\ 
	\begin{enumerate}[label={\arabic*)}, itemjoin={,}, itemjoin*={, and}]
		\item If \(N\sim \operatorname{nm}_G(r,x)\) and \(C\in \mathcal{C}_G\), then \(N_C=(N_i)_{i\in C}\) has the distribution \(\operatorname{nm}_{K_C}(r,x^C)\), where
		\begin{equation}
			\label{eq-xC}
			\begin{aligned}
				x^C_i=x_i\, \frac{\Delta_{\mathfrak{nb}_{G^*}(i)}(x)}{\Delta_{V\setminus C}(x)},\ i\in C.
			\end{aligned}
		\end{equation}
		\item If \(N\sim \operatorname{mult}_G(r,y)\) and \(C\in \mathcal{C}_G\), then \(N_C=(N_i)_{i\in C}\) has the distribution \(\operatorname{mult}_{K_C}(r,y^C)\), where
		\begin{equation}
			\label{eq-yC}
			\begin{aligned}
				y^C_i=y_i\, \frac{\delta_{\mathfrak{nb}_{G^*}(i)}(y)}{\delta_{V\setminus C}(y)}, \ i\in C.
			\end{aligned}
		\end{equation}
	\end{enumerate}
\end{Theorem}

\subsection{Bivariate marginals}\label{sec:bivariate}
In this section we discuss bivariate marginal distributions of $\operatorname{mult}_G$ and $\operatorname{nm}_G$. As a by-product, we will see that, in general, non-clique marginals fall outside our models, which complements the result of Theorem \ref{thm:marginal}. 

Assume that $G=(V,E)$ is a decomposable graph and fix two different vertices $i,j\in V$.
If $\{i,j\}\in E$, then the joint distribution of $(N_i,N_j)$ in our two models is described by Theorem \ref{thm:marginal} as we have $E\subset\mathcal{C}_G$. 

Suppose $\{i,j\}\notin E$ and let $S=\mathfrak{nb}_G(i)\cap \mathfrak{nb}_G(j)$.  Define $\mathcal{C}_i=\{ C\in\mathcal{C}_G^+\colon i\in C\}$ and $\mathcal{C}_j=\{ C\in\mathcal{C}_G^+\colon j\in C\}$.  For any $C\in\mathcal{C}_i$ and $C'\in\mathcal{C}_j$ we have $C\cap C'\subset S$. If $S=\emptyset$, then $C\cap C'$ is empty. Thus, by the global Markov property \eqref{fact:undirected}, $N_i$ and $N_j$ are independent.
The remaining case $S\neq\emptyset$ is most interesting and is described by the Theorem \ref{thm:pairwise}.  If $S\neq\emptyset$, the above reasoning implies that $S$ separates $i$ and $j$ in $G$. Indeed, the existence of a path connecting $i$ and $j$ with no intersection with $S$ would produce an induced cycle of length $\geq 4$ and  contradict decomposability of $G$. In such a case, $N_i$ and $N_j$ are conditionally independent given $N_S$.
The above argument can be restated using a marginalization in graphical models \cite{Fryd,Stu}.

We start with an example, which illustrates Theorem \ref{thm:pairwise} below.
\begin{expl}
	Let \(G\) be the graph in Example \ref{ex-1-2-3}, abusively this graph is denoted by \(1-2-3\).
	\begin{enumerate}[label={\arabic*)}, itemjoin={,}, itemjoin*={, and}]
		\item If  \(N\sim \operatorname{nm}_G(r,x)\), then 
		\begin{align}\label{eq:nm13}
			\mathbb{P}((N_1,N_3)=(n_1,n_3))=\binom{r+n_1-1}{ n_1}\binom{r+n_3-1}{n_3} x_1^{n_1}x_3^{n_3}\Delta(x)^r {}_2F_1(n_1+r,n_3+r;r;x_2)
		\end{align}
		for $n_1,n_3\in\mathbb{N}$, where $\Delta(x_1,x_2,x_3)=1-x_1-x_2-x_3+x_1x_3=\Delta_G(x)$ and ${}_2F_1$ is the hypergeometric function.
		\item If  \(N\sim \operatorname{mult}_G(r,y)\), then 
		\begin{align}\label{eq:mult13}
			\mathbb{P}((N_1,N_3)=(n_1,n_3))=\binom{r}{ n_1}\binom{r}{ n_3} \frac{y_1^{n_1}y_3^{n_3}}{\delta(y)^r} {}_2F_1(n_1-r,n_3-r;-r;-y_2)
		\end{align}
		for \(n_1,n_3\in \{0,1 ,\ldots,r\}\), where $\delta(y_1,y_2,y_3)=\Delta(-y)$. 
	\end{enumerate}
\end{expl}

\begin{Remark}
	The hypergeometric function ${}_2F_1(a,b,c,\cdot)$ is generally undefined if $c$ is a non-positive integer. However, if $a$ and $b$ are non-positive integers such that $\min\{a,b\}\geq c$, then the definition of the hypergeometric series makes sense as we argue in the proof of Lemma \ref{lem:hypergeo}. 
\end{Remark}

If $(N_1,N_3)$ satisfies \eqref{eq:nm13}, we write $(N_1,N_3)\sim \operatorname{nm}_{1-2-3}^{1,3}(x_1,x_2,x_3)$ and if $(N_1,N_3)$ satisfies \eqref{eq:mult13}, then we write $(N_1,N_3)\sim \operatorname{mult}_{1-2-3}^{1,3}(x_1,x_2,x_3)$. We are ready to formulate the main result of this section.
\begin{Theorem}\label{thm:pairwise} Let $G=(V,E)$ be a decomposable graph. Assume that $i, j\in V$, $i\neq j$, $\{i,j\}\notin E$ and $S=\mathfrak{nb}_G(i)\cap \mathfrak{nb}_G(j)\neq\emptyset$. 
	\begin{enumerate}[label={\arabic*)}, itemjoin={,}, itemjoin*={, and}]
		\item If \(N\sim \operatorname{nm}_G(r,x)\), then $(N_i, N_j)\sim \operatorname{nm}_{1-2-3}^{1,3}(\tilde x_1,\tilde x_2,\tilde x_3)$, where (recall \eqref{eq-xC})
		\begin{equation}\label{eq:tildex}
			\begin{gathered}
				\tilde x_1 =  x_i^{S\cup\{i\}} = 1- \frac{\Delta_{V\setminus S}(x)}{\Delta_{V\setminus(S\cup\{i\})}(x)},\qquad \tilde x_3 = x_j^{S\cup\{j\}}=  1- \frac{\Delta_{V\setminus S}(x)}{\Delta_{V\setminus(S\cup\{j\})}(x)},\\
				\tilde{x}_2  = \frac{\Delta_{V\setminus S}(x)(\Delta_{V\setminus S}(x)-\Delta_G(x))}{\Delta_{V\setminus (S\cup\{i\})}(x) \Delta_{V\setminus (S\cup\{j\})}(x)}.
			\end{gathered}
		\end{equation}
		\item  If \(N\sim \operatorname{mult}_G(r,y)\), then $(N_i, N_j)\sim \operatorname{mult}_{1-2-3}^{1,3}(\tilde y_1,\tilde y_2,\tilde y_3)$, where (recall \eqref{eq-yC})
		\begin{equation}\label{eq:tildey}
			\begin{gathered}
				\tilde y_1 =  y_i^{S\cup\{i\}} =  \frac{\delta_{V\setminus S}(y)}{\delta_{V\setminus(S\cup\{i\})}(y)}-1,\qquad \tilde y_3 = y_j^{S\cup\{j\}}=   \frac{\delta_{V\setminus S}(y)}{\delta_{V\setminus(S\cup\{j\})}(y)}-1,\\
				\tilde{y}_2  = \frac{\delta_{V\setminus S}(y)(\delta_G(y)-\delta_{V\setminus S}(y))}{\delta_{V\setminus (S\cup\{i\})}(y) \delta_{V\setminus (S\cup\{j\})}(y)}.
			\end{gathered}
		\end{equation}
	\end{enumerate}
\end{Theorem}

\subsection{Markov properties}
\label{sec:org7774c71}
\begin{Theorem}\label{thm:Markov}	\label{thm-strong-directed-hyper-markov-properties}
	Let \(\mathcal{G}\) be a moral DAG with skeleton \(G\).
	\begin{enumerate}[label={\arabic*)}, itemjoin={,}, itemjoin*={, and}]
		\item If \(N\sim \operatorname{nm}_G(r,x)\), then
		\[\mathbb{P}(N=n)=\prod_{i\in V} \operatorname{nb}(r+|n_{\mathfrak{pa}(i)}|,u^{\mathcal{G}}_i)(n_i),\quad n\in\mathbb N^V,\]
		where \(u^{\mathcal{G}}_i = u^{\mathcal{G}}_i(x)\) is defined in \eqref{eq-u-calG} and \(\operatorname{nb}(s,p)(k)\) is the probability of the value $k$ for the negative binomial distribution with parameters $s>0$ and $p\in(0,1)$, see \eqref{eq-cnm2}. In particular, 
		\[\mathbb{P}(N_i=n_i| N_{\mathfrak{pa}(i)}=n_{\mathfrak{pa}(i)})=\operatorname{nb}(r+|n_{\mathfrak{pa}(i)}|,u^{\mathcal{G}}_i)(n_i),\quad i\in V.
		\]
		\item If \(N\sim \operatorname{mult}_G(r,y)\), then
		\[\mathbb{P}(N=n)=\prod_{i\in V} \operatorname{bin}(r-|n_{\mathfrak{pa}(i)}|, w^{\mathcal{G}}_i)(n_i),\quad n\in \mathbb{N}_{G,r}.\]
		where $ \mathbb{N}_{G,r}$ is defined in  \eqref{ngr},  \(w^{\mathcal{G}}_i=w^{\mathcal{G}}_i(y)\) is defined in  \eqref{eq-w-calG} and \(\operatorname{bin}(s,p)(k)\) is the probability of the value $k$ for the binomial distribution with parameters $s\in\mathbb N_+$ and $p\in(0,1)$, see \eqref{multcl2}. In particular,
		\[\mathbb{P}(N_i=n_i| N_{\mathfrak{pa}(i)}=n_{\mathfrak{pa}(i)})=\operatorname{bin}(r-|n_{\mathfrak{pa}(i)}|,w^{\mathcal{G}}_i)(n_i),\quad i\in V.
		\]
	\end{enumerate}
\end{Theorem}

The representations of the probability mass functions for $\operatorname{nm}_G$ and $\operatorname{mult}_G$ given in the result below, in view of Theorem \ref{thm-strong-directed-hyper-markov-properties}, follow easily by comparing \eqref{fact:undirected} and \eqref{fact:directed}.
\begin{Proposition}	\label{prop-nm-mult-with-cliques}
	Let \(G=(V,E)\) be a decomposable graph.
	\begin{enumerate}[label={\arabic*)}, itemjoin={,}, itemjoin*={, and}]
		\item If \(N\sim \operatorname{nm}_G(r,x)\), then
		\begin{equation}
			\label{eq-nmG-C-S}
			\begin{aligned}
				\mathbb{P}(N=n)=\frac{\prod_{C\in \mathcal{C}_G^+} \operatorname{nm}(r,x^C)(n_C)}{\prod_{S\in \mathcal{S}_G^-}\left( \operatorname{nm}(r,x^S)(n_S) \right)^{\nu_S}},\quad n\in\mathbb N^V,
			\end{aligned}
		\end{equation}
		where \(\operatorname{nm}(r,v)\) is the negative multinomial distribution, see \eqref{eq-cnm}, and \(x^C\), $C\in\mathcal C_G$, is defined in \eqref{eq-xC}.
		\item If \(N\sim \operatorname{mult}_G(r,y)\), then
		\begin{equation}
			\label{eq-multG-C-S}
			\begin{aligned}
				\mathbb{P}(N=n)=\frac{\prod_{C\in \mathcal{C}_G^+} \operatorname{mult}(r,y^C)(n_C)}{\prod_{S\in \mathcal{S}_G^-}\left( \operatorname{mult}(r,y^S)(n_S) \right)^{\nu_S}},\quad n\in\mathbb N_{G,r},
			\end{aligned}
		\end{equation}
		where \(\operatorname{mult}(r, v)\) is the multinomial distribution, see \eqref{multcl}, and  \(y^C\), $C\in \mathcal C_G$, is defined in \eqref{eq-yC}.
	\end{enumerate}
\end{Proposition}

\subsection{Characterizations through global Markov properties}\label{sec:char-markov}
In this section, we demonstrate that both  discrete distributions we consider, emerge in a natural way from global Markov properties.  Specifically, we show that the distributions are determined by the structure of a decomposable graph and their clique marginals.
\begin{Theorem}\label{thm:char-markov}
	Let $N=(N_i)_{i\in V}$ be a random vector and let $G=(V,E)$ be a connected decomposable graph.  Suppose that the distribution of $N$ is $G$-Markov.
	\begin{enumerate}[label={\arabic*)}, itemjoin={,}, itemjoin*={, and}]
		\item If for each $C\in\mathcal{C}_G^+$  the marginal $N_C$ follows a negative multinomial $\operatorname{nm}$  distribution, then there exists $(r,x)\in(0,\infty)\times M_G$ such that $N\sim\operatorname{nm}_G(r,x)$.
		\item If for each $C\in\mathcal{C}_G^+$ the marginal $N_C$ follows a multinomial $\operatorname{mult}$ distribution,  then there exists $(r,y)\in\mathbb{N}_+\times (0,\infty)^V$ such that $N\sim\operatorname{mult}_G(r,y)$. 
	\end{enumerate}
\end{Theorem}

\begin{Remark}\label{rem:non-dec}
	It might be tempting to define distributions $\operatorname{mult}_G$ and $\operatorname{nm}_G$ for general, not necessarily decomposable, graphs $G$. In particular, the formulas for probabilities for $\operatorname{mult}_G$ indeed define a valid probability distributions for any graph $G$. However, we will show that for a special non-decomposable graph $G$, the distribution $\mathbb{P}(N=n)=\gbinom{r}{n}{\,G}y^n\delta_G(y)^{-r}$ is $G$-Markov if $r=1$, but this property does not hold for $r>1$. With the same $G$, if $\mathbb{P}(N=n)=\gbinomm{|n|+r-1}{n}{\,G}x^n\Delta_G(x)^{r}$, then it fails to be $G$-Markov for any $r>0$.  This motivates the restriction of our definitions to decomposable graphs.	
	
Assume that $\mathbb{P}(N=n)=\gbinom{r}{n}{\,C_4} y^n \delta_{C_4}(y)^{-r}$, where $C_4$ is the cycle of length $4$. The distribution of $N$ is $C_4$-Markov if $N_1$ and $N_3$ are conditionally independent given $(N_2,N_4)$, and if $N_2$ and $N_4$ are conditionally independent given $(N_1,N_3)$. In view of symmetry, it suffices to check only the first of these properties, which is equivalent to the following factorization
	\begin{align}\label{eq:C4Markov}
	\mathbb{P}(N=n) = \frac{\mathbb{P}(N_{\{1,2,4\}}=n_{\{1,2,4\}})\mathbb{P}(N_{\{2,3, 4\}}=n_{\{2,3,4\}})}{\mathbb{P}(N_{\{2,4\}}=n_{\{2,4\}})}.
	\end{align}
When $r=1$, the support of $N$ is $\mathbb{N}_{C_4,1}=\{ n\in\mathbb{N}^4\colon \max\{n_1+n_2, n_2+n_3, n_3+n_4, n_1+n_4\}\leq 1\}$. Therefore, 
\begin{align*}
\mathbb{P}(N=n) &= \frac{y^n}{\delta_{C_4}(y)} \mathbbm{1}_{\mathbb{N}_{G,1}}(n) =  \frac{\prod_{i=1}^4 y_i^{n_i}}{\delta_{C_4}(y)}\mathbbm{1}_{n_1\leq (1-n_2)\wedge(1-n_4)}\mathbbm{1}_{n_3\leq (1-n_2)\wedge(1-n_4)}\\
&=f(n_{\{1,2,4\}})\, g(n_{\{2,3,4\}}),
\end{align*}
which readily implies the factorization in \eqref{eq:C4Markov}. 

Now, consider $r=2$ and let $n=(1,1,1,1)$. The support of $N$, i.e., $\mathbb{N}_{C_4,2}$ has $24$ elements. We compute
	\begin{align*}
		\mathbb{P}(N_{\{1,2,4\}}=(1,1,1)) & = \left(\gbinom{2}{1,1,0,1}{\,C_4} y_1 y_2 y_4 + \gbinom{2}{1,1,1,1}{\,C_4} y_1y_2y_3y_4\right)/\delta_{C_4}(y)^2,\\
		\mathbb{P}(N_{\{2,3,4\}}=(1,1,1)) &= \left(\gbinom{2}{0,1,1,1}{\,C_4}  y_2 y_3 y_4 + \gbinom{2}{1,1,1,1}{\,C_4} y_1y_2y_3y_4\right)/\delta_{C_4}(y)^2,\\
		\mathbb{P}(N_{\{2,4\}}=(1,1)) & = \frac{\gbinom{2}{0,1,0,1}{\,C_4}  y_2 y_4 + \gbinom{2}{1,1,0,1}{\,C_4} y_1 y_2 y_4 + \gbinom{2}{0,1,1,1}{\,C_4} y_2 y_3 y_4+ \gbinom{2}{1,1,1,1}{\,C_4} y_1y_2y_3y_4}{\delta_{C_4}(y)^2}.
	\end{align*}
By collecting the monomial terms of  $\delta_{C_4}(y)^2=(1+y_1+y_2+y_3+y_4+y_1 y_3 + y_2 y_4)^2$, it is straightforward to verify that
	\begin{align*}
		\gbinom{2}{0,1,0,1}{C_4} = 4, \qquad \gbinom{2}{1,1,0,1}{C_4} = \gbinom{2}{0,1,1,1}{C_4} = \gbinom{2}{1,1,1,1}{C_4} = 2.  
	\end{align*}
Thus,  for $n=(1,1,1,1)$, the right-hand side of \eqref{eq:C4Markov} becomes 
	\[
 \frac{2 y_1(1+y_1) y_2 y_3(1+y_3) y_4}{(2+y_1+y_3+y_1 y_3) \delta_{C_4}(y)^2},
	\]
which does not equal $\mathbb{P}(N=(1,1,1,1))$. 	This shows that the distribution of $N$ is not $C_4$-Markov for $r=2$. 

If $\mathbb{P}(N=n)=\gbinomm{|n|+r-1}{n}{C_4}x^n\Delta_{C_4}(x)^{r}$, then, using Remark \ref{rem:Cd}, we can show that $\gbinomm{|n|+r-1}{n}{C_4}$ is not of the form $\prod_{C\in\mathcal{C}_{C_4}^+} f_C(n_C)$, which implies that it is not $C_4$-Markov  for any $r>0$. 
\end{Remark}

\section{Probabilistic interpretations}
\label{sec:org7b431d8}
 Multinomial and negative multinomial distributions have the following urn model interpretations:  Consider an urn containing infinitely many balls of \(m+1 \ge 2\) different colors, \(0,1 ,\ldots,m\). The proportions of the balls of each color are given by \(x_0,x_1 ,\ldots,x_m\). Balls are drawn from the urn one by one at random with replacement. If the sampling stops after \(r\)'th draw, $r\in \mathbb N_+$, then \(N=(N_0,\ldots,N_m)\) (where \(N_i\) representing the number of balls of color \(i\) drawn) follows the multinomial distribution \(\operatorname{mult}(r,x)\). 
On the other hand, if sampling stops when the number of white balls drawn (we label $0$ as the white color)  reaches the predetermined integer value \(r\in \mathbb N_+\), then the distribution of \(N\) is negative multinomial \(\operatorname{nm}(r,x)\).

Let us  provide probabilistic interpretations  for  the \(G\)-negative multinomial distribution and the \(G\)-multinomial distribution. We fix a finite, not necessarily decomposable, graph \(G=(V,E)\). Consider \(V\) as the alphabet, with its elements now referred to as letters. Let \(\mathfrak V\)  be the free monoid generated from  \(V\) through the operation of concatenation, that is, \(\mathfrak V\) is the collection of finite sequences of letters, called words.

The edges \(E^*\) of the complement graph \(G^*\) define a commutation rule for concatenation:
\begin{equation}
	\label{eq-commutation-rule}
	\begin{aligned}
		i_1i_2=i_2i_1\quad \mbox{ i.f.f. } \quad\{i_1,i_2\}\in E^*.
	\end{aligned}
\end{equation}
Two words \(w,w'\in \mathfrak V\) are adjacent if there are \(w_1,w_2\in \mathfrak V\) and \(\{i_1,i_2\}\in E^*\) such that
\[w=w_1i_1i_2w_2 \text{ and }w'=w_1i_2i_1w_2.\]
This further defines an equivalence  relation \(\sim\) on \(\mathfrak V\): two words \(w,w'\in \mathfrak V\) are equivalent if there is a sequence of words \(w_0,w_1 ,\ldots,w_k\in\mathfrak V\) such that consecutive words in the sequence are adjacent and \(w_0=w,\ w_k=w'\). Let \(L=\mathfrak V/\sim\) be the quotient monoid, with elements in \(L\) denoted by \([w]\), where \(w\in\mathfrak V\) is any representative of the equivalence class \([w]\). We say that \(L\) is the free  quotient monoid induced by \(G=(V,E)\).

For a clique \(C\in \mathcal{C}_{G^*}\), we denote \([C]=[\pi_C]\), the equivalence class of the product of letters from \(C\). This notation is well-defined since all vertices in \(C\) commute, recall \eqref{eq-commutation-rule}. Next, we consider the free quotient  monoid algebra \(\mathbb{Z}[L]\), that is, the set of formal sums \(\sum_{[w]\in L}z_w [w]\), where \(z_w\) is an integer. The following theorem can be found in \cite{cartier1969commutation}, see also \cite{krattenthaler2006theory,viennot1986heaps}.
\begin{Theorem}[Cartier--Foata]\label{CFT}
	For any graph \(G=(V,E)\), consider \(L\), the free  quotient monoid induced by \(G\). The element $\sum_{C\in \mathcal{C}_{G^*}} (-1)^{|C|} [C]$ has a unique inverse in \(\mathbb{Z}[L]\); it is $\sum_{[w]\in L}[w]$ which is both left and right inverse, i.e.
	\begin{equation}
		\label{eq-cartier-foata-monoid-ring}
		\begin{aligned}
			\left( \sum_{C\in \mathcal{C}_{G^*}} (-1)^{|C|} [C] \right)^{-1} = \sum_{[w]\in L}[w].
		\end{aligned}
	\end{equation}
	In particular, \(\Delta_G(x)^{-1}=\sum_{[w]\in L}\pi_w(x)\), where \(\pi_w(x)=\prod_{i\in w}x_i\)  for \(w\in \mathfrak V\)  and the equality holds for all $x\in\mathbb R^V$ such that the right-hand side converges.
	\label{thm-cartier-foata}
\end{Theorem}

We obtain the following representation of the $\operatorname{nm}_G$ distribution. 
\begin{Theorem}	\label{thm-nm-proba-interpretation}
	For a finite graph \(G=(V,E)\) and \(x\in M_G\), let \(\ell\) be a random variable valued in \(L\), the  free quotient monoid induced by \(G\), with distribution
	\[\mathbb{P}(\ell=[w])=\frac{\pi_w(x)}{\sum_{[w']\in L}\pi_{w'}(x)},\qquad [w]\in L,\]
	where \(\pi_w(x)=\prod_{i\in w}x_i\)  and $x\in M_G$. If \(\varepsilon\colon L\to \mathbb{N}^V\) is such that \(\varepsilon([w])\) counts the number of occurrences of each letter from $V$ in \([w]\in L\), then \(\varepsilon(\ell)\sim \operatorname{nm}_G(1,x)\).
\end{Theorem}
In particular, we obtain
\begin{align}\label{eq:nn}
\gbinomm{|n|}{n}{G} = \left|\{[w]\in L\colon\varepsilon([w])=n \}\right|.
\end{align}

In case of \(\operatorname{mult}_G(1,y)\) distribution the situation is much simpler, we do not need the free quotient monoid model. Instead, we rely only on  sampling cliques of $G^*$. 
\begin{Theorem}	\label{thm-mult-proba-interpretation}
	For a finite graph \(G=(V,E)\) and \(y\in (0,\infty)^V\), let \(\ell\) be a random variable valued in \(\mathcal{C}_{G^*}\), with distribution
	\[\mathbb{P}(\ell=C)=\frac{\pi_C(y)}{\delta_G(y)},\qquad  C\in \mathcal{C}_{G^*}.\]
	If $\varepsilon\colon \mathcal C_{G^*}\to\mathbb N_{G,1}$ is such that \(\varepsilon(C)\) counts the number of  occurrences of each vertex of $V$ in $C\in \mathcal C_{G^*}$ (these counts are in $\mathbb N_{G,1}$)  then \(\varepsilon(\ell)\sim \operatorname{mult}_G(1,y)\).
\end{Theorem}
\begin{Remark}
	\label{rmk-non-decomposable-no-proba-inter}
 The result of Theorem \ref{thm-nm-proba-interpretation} indicates that \(\operatorname{nm}_G(1,x)\) corresponds to the random experiment of drawing elements \([w]\) from \(L\) with probability proportional to \(\pi_w(x)\).  Thus, in view of 1) of Corollary \ref{coro-sum-of-nm}, one can interpret \(\operatorname{nm}_G(r,x)\) as the distribution of sums of counts of all letters obtained in \(r\) independent repetitions of such experiment. Similarly, in view of 2) of Corollary \ref{coro-sum-of-nm}, one can interpret \(\operatorname{mult}_G(r,y)\) as the distribution of sums of counts of vertices in $r$ independent experiments described in Theorem \ref{thm-mult-proba-interpretation}.

In fact, \(\operatorname{mult}_G(1,y)\) corresponds to taking a sample of random independent set (hardcore model) of \(G\) with parameters \(y\); see the discussion below.
\end{Remark}

\subsection{Relation to hardcore model}
Given a graph \(G=(V,E)\) and parameters \(\lambda=(\lambda_{i})_{i\in V} \in (0,\infty)^V\), the hardcore model on \(G\) involves choosing a random independent set \(I\subset V\) (a set is independent if no two of its vertices are adjacent) with probability
\[ \frac{1}{Z_{G}(\lambda)} \prod_{i\in I}\lambda_{i}, \]
where the partition function of the hardcore model, \(Z_{G}(\lambda)\), is defined by
\[ Z_{G}(\lambda) = \sum_{I-\text{independent set}} \prod_{i\in  I}\lambda_{i}, \]
see e.g. \cite{Scott_2005}. 
By definition, an independent set \(I\) of \(G\) is a clique of \(G^{*}\), making the hardcore model naturally related to our discussion. In fact, \(Z_{G}(\lambda) = \delta_{G}(\lambda)\). The probabilistic interpretation of \(\operatorname{mult}_{G}(1,\lambda)\) in Theorem~\ref{thm-mult-proba-interpretation} is exactly the hardcore model. Therefore, \(\varepsilon(I)\), which counts the occurrences of each vertex in an independent set \(I\), follows \(\operatorname{mult}_{G}(1,\lambda)\). Similarly, \(\operatorname{mult}_{G}(r,\lambda)\) corresponds to sampling \(r\) independent hardcore models and counting the occurrences of each vertex.

In particular, the global graph Markov property of \(\operatorname{mult}_{G}\) on a decomposable graph implies that, for the hardcore model, conditioned on a clique \(C\) that separates \(A\) and \(B\) in $G$, the hardcore models viewed in \(A\) and \(B\) are independent. The graph negative inverted Dirichlet distribution provides a natural Bayesian conjugate prior distribution to the Gibbs measure of the hardcore model.

Shearer's condition (which strengthens the Lov\'asz local lemma \cite{Alon_2008}), equivalent to the absolute convergence of the Mayer expansion of \(\log Z_{G}(\lambda)\) (see \cite{Dobrushin_1996} and \cite{Scott_2005}), is particularly relevant.  Since one of Shearer's polynomials coincides with our clique polynomial \(\delta_{G}\), these results provide conditions for locating the zeroes of the clique polynomial \(\delta_{G}\) outside a certain ball. See also \cite{BISSACOT_2011} for an improvement of Lovász local lemma using cluster expansion. 
These results, which apply to general graphs, offer insights into the set \(M_{G}\) (recall \eqref{eq-support-MG} and the fact that $\Delta_G(x)=\delta_G(-x)$). Finally, we draw the reader’s attention to the algorithm introduced in \cite{Moser_2010}. 

\section{Graph Dirichlet and graph inverted Dirichlet laws}
\label{sec:orge36368d}
\subsection{Definitions, relations to Dirichlet/beta type distributions}
\label{sec:org9c5ed7a}
Now, we introduce two graph related Dirichlet-type distributions by the form of densities with respect to the Lebesgue measure on $\mathbb R^V$.
\begin{Definition}	\label{def-graph-dir-inverted-dir}
	Let \(G=(V,E)\) be a finite decomposable graph.
	\begin{enumerate}[label={\arabic*)}, itemjoin={,}, itemjoin*={, and}]
		\item A probability measure on \(\mathbb{R}^V\) is the \(G\)-Dirichlet distribution with parameters \((\alpha,\beta)\), where \(\alpha\in (0,\infty)^V\) and \(\beta > 0\), if it has the  density
		\begin{equation}
			\label{eq-density-G-Dir}
			\begin{aligned}
				f(x)=K_G(\alpha,\beta) \Delta_G(x)^{\beta-1}x^{\alpha-1} \mathbbm{1}_{M_G}(x),\qquad x\in \mathbb{R}^V,
			\end{aligned}
		\end{equation}
		where \(x^{\alpha-1}=\prod_{i\in V}x_i^{\alpha_i-1}\), the support \(M_G\) is defined in \eqref{eq-support-MG} and \(K_G(\alpha,\beta)\) is a normalizing constant. This measure is denoted by \(\operatorname{Dir}_G(\alpha,\beta)\).
		\item A probability measure on \(\mathbb{R}^V\) is the \(G\)-inverted Dirichlet distribution with parameters \((\alpha,\beta)\), where \(\alpha\in(0,\infty)^V\) and \(\beta > \max_{C\in \mathcal{C}_G^+}|\alpha_C|\), if it has the  density
		\begin{equation}
			\label{eq-density-G-iDir}
			\begin{aligned}
				f(y)=k_G(\alpha,\beta)\delta_G(y)^{-\beta} y^{\alpha-1} \mathbbm{1}_{(0,\infty)^V}(y),\qquad y\in \mathbb{R}^V,
			\end{aligned}
		\end{equation}
		where \(y^{\alpha-1}\) is \(\prod_{i\in V}y_i^{\alpha_i-1}\) and \(k_G(\alpha,\beta)\) is a normalizing constant. This measure is denoted by \(\operatorname{IDir}_{G}(\alpha,\beta)\).
	\end{enumerate}
\end{Definition}
In Theorem \ref{thm-norcons} we provide explicit expression for the normalizing constants \(K_G\) and \(k_G\). In particular, it means that these distributions are well-defined. The following examples justify the terminology in Definition \ref{def-graph-dir-inverted-dir}.

\begin{expl} 	\label{expl-g-dir-examples}$\,$
	\begin{enumerate}[label={\arabic*)}, itemjoin={,}, itemjoin*={, and}]
		\item If \(G\) is a complete graph, then \(\mathcal{C}_{G^*}=\{\emptyset, \ \{i\}, i\in V\}\), so
		\[\Delta_G(x)=1-|x_V|,\ \ \delta_G(y)=1+|y_V|\]
		with \(|x_V|=\sum_{i\in V}x_i\) and \(M_G\) is the unit simplex \(\{x\in (0,\infty)^V,\ |x_V| < 1\}\). Equations \eqref{eq-density-G-Dir} and \eqref{eq-density-G-iDir} become
		\[f_{\operatorname{Dir}_G}(x)=K_G(\alpha,\beta)\prod_{i\in V}x_i^{\alpha_i-1}(1-|x_V|)^{\beta-1} \mathbbm{1}_{M_G}(x)\]
		and
		\[f_{\operatorname{IDir}_G}(y)=k_G(\alpha,\beta) \prod_{i\in V} y_i^{\alpha_i-1}(1+|y_V|)^{-\beta} \mathbbm{1}_{(0,\infty)^V}(y).\]
		In the former formula we recognize  the density of the Dirichlet distribution with parameters \((\alpha,\beta)\), consequently, 
		\[K_G(\alpha,\beta)=\frac{\Gamma(|\alpha_V|+\beta)}{\Gamma(\beta)\prod_{i\in V}\Gamma(\alpha_i)},\]
		and in the latter the density of the inverted Dirichlet distribution with parameters \((\alpha,\beta)\),  consequently, 
		\[k_G(\alpha,\beta)=\frac{\Gamma(\beta)}{\Gamma(\beta-|\alpha_V|)\prod_{i\in V}\Gamma(\alpha_i)}.\]
		\item If \(G\) is the fully disconnected graph, then \(\mathcal{C}_{G^*}=\{A:\ A \subset  V\}\), so
		\[\Delta_G(x)=\prod_{i\in V}(1-x_i),\ \ \delta_G(y)=\prod_{i\in V}(1+y_i),\]
		and \(M_G=(0,1)^V\). Thus \eqref{eq-density-G-Dir} and \eqref{eq-density-G-iDir} become
		\[f_{\operatorname{Dir}_G}(x)=K_G(\alpha,\beta) \prod_{i\in V}\left[ x_i^{\alpha_i-1}(1-x_i)^{\beta-1} \mathbbm{1}_{(0,1)}(x_i) \right]\]
		and
		\[f_{\operatorname{IDir}_G}(y)=k_G(\alpha,\beta)\prod_{i\in V}\left[ y_i^{\alpha_i-1}(1+y_i)^{-\beta} \mathbbm{1}_{(0,\infty)}(y_i) \right].\]
		So \(\operatorname{Dir}_G(\alpha,\,\beta)\) is a product of independent univariate Beta distributions of the first kind,  \(\operatorname{B}_I(\alpha_i,\beta),\ i\in V\), and \(\operatorname{IDir}_G(\alpha,\,\beta)\) is a product of independent univariate Beta distribution of the second kind \(\operatorname{B}_{II}(\alpha_i,\beta-\alpha_i)\), $i\in V$. In particular,
		\[K_G(\alpha,\beta)=\prod_{i\in V} \frac{\Gamma(\alpha_i+\beta)}{\Gamma(\alpha_i)\Gamma(\beta)},\ \ \ k_G(\alpha,\beta)=\prod_{i\in V} \frac{\Gamma(\beta)}{\Gamma(\alpha_i)\Gamma(\beta-\alpha_i)}.\]
	\end{enumerate}
\end{expl}

\begin{Remark}\label{hyperDir}
	Another graph-related Dirichlet-type distribution known in the literature as the hyper-Dirichlet law was introduced for decomposable graphs in \cite{Dawid1993}. Its density is of the form
	$$
	f(x)\propto \tfrac{\prod_{C\in \mathcal C^+_G}\,\prod_{k_C\in\mathcal I_C}x_{k_C}^{\alpha_{k_C}}}{\prod_{S\in\mathcal S^-_G}\,\left(\prod_{k_S\in\mathcal I_S}\,x_{k_S}^{\alpha_{k_S}}\right)^{\nu_S}},
	$$
	where $\mathcal I_i$, $i\in V$, are finite sets with $s_i=|\mathcal I_i|\ge 2$, $i\in V$, $\mathcal I_A=\times_{i\in A}\,\mathcal I_i$,  $x_{k_A}^{\alpha_{k_A}}=\prod_{j\in A}\,x_{k_j}^{\alpha_{k_j}}$, $A\subset V$. The density $f$ is positive for all $x$'s from a certain manifold in a unit simplex $T_s$ of  dimension $s=\prod_{i\in V}\,s_i$.  It is worth to mention that the hyper-Dirichlet and $G$-Dirichlet are different distributions. In particular, the dimension of $G$-Dirichlet is $|V|$ while in the case of the hyper-Dirichlet it is $s\ge 2^{|V|}$.   Nevertheless,  they both reduce to Dirichlet distributions (of different dimensions), when $G$ is a complete graph.
\end{Remark}

\begin{Remark}[Relation between \(\operatorname{Dir}_G\) and \(\operatorname{IDir}_G\)]
	\label{rmk-no-relation-GDir-GIDir}
	The two types of Beta distributions are related via the Beta-Gamma algebra: if \(U,V\) are independent Gamma random variables,  then \(U/(U+V)\) is Beta of type $I$ and \(U/V\) is Beta of type $II$. More generally, the Dirichlet distribution and the inverted Dirichlet distribution are linked via the Beta-Gamma algebra. Specifically, let \(X_0,X_1 ,\ldots,X_N\) be independent Gamma random variables with the same scale parameter. Define \(Y_i=X_i/(\sum_{j=0}^N X_j)\) and \(Z_i=X_i/X_0\) for \(i=1,\ldots,N\). Then, \(Y=(Y_i)_{i=1}^N\) follows a Dirichlet distribution and \(Z=(Z_i)_{i=1}^N\) follows an inverted Dirichlet distribution. 
	
	One might wonder whether this relation can be generalized to the case of a decomposable graph. In general, this is not true. For instance, it can be verified that for the graph \(1-2-3\), it is impossible to couple these two models using independent Gamma random variables.
\end{Remark}

\subsection{Normalizing constants}
\label{sec:orgb65e48a}
\begin{Theorem}
	Let \(\mathcal{G}\) be a moral DAG with skeleton \(G\) with a parent function \(\mathfrak{pa}\). Let \(m\) be the number of connected components of \(G\).
	\begin{enumerate}[label={\arabic*)}, itemjoin={,}, itemjoin*={, and}]
		\item For \(\alpha\in (0,\infty)^V\) and \(\beta> 0\), we have
		\begin{equation}
			\label{eq-KG-per-i}
			\begin{aligned}
				K_G(\alpha,\beta)=\prod_{i\in V} \frac{\Gamma(|\alpha_{\overline{\mathfrak{pa}}(i)}|+\beta)}{\Gamma\left( |\alpha_{\mathfrak{pa}(i)}|+\beta \right) \Gamma(\alpha_i)},
			\end{aligned}
		\end{equation}
		which can be also written as
		\begin{equation}
			\label{eq-KG-C-S}
			\begin{aligned}
				K_G(\alpha,\beta)=\frac{\prod_{C\in \mathcal{C}_G^+}\Gamma(|\alpha_C|+\beta)}{\Gamma^m(\beta)\prod_{i\in V}\Gamma(\alpha_i) \prod_{S\in \mathcal{S}_G^-}\left(\Gamma(|\alpha_S|+\beta)\right)^{\nu_S}}.
			\end{aligned}
		\end{equation}
		\item For \(\alpha\in (0,\infty)^V\) and \(\beta > \max_{C\in \mathcal{C}_G^+}|\alpha_C|\), we have
		\begin{equation}
			\label{eq-kG-per-i}
			\begin{aligned}
				k_G(\alpha,\beta)=\prod_{i\in V} \frac{\Gamma(\beta-|\alpha_{\mathfrak{pa}(i)}|)}{\Gamma(\beta-|\alpha_{\overline{\mathfrak{pa}}(i)}|)\Gamma(\alpha_i)},
			\end{aligned}
		\end{equation}
		which can be also written as
		\begin{equation}
			\label{eq-kG-C-S}
			\begin{aligned}
				k_G(\alpha,\beta)=\frac{\Gamma(\beta)^m \prod_{S\in \mathcal{S}_G^-} \left(\Gamma(\beta-|\alpha_S|)\right)^{\nu_S}}{\prod_{i\in V}\Gamma(\alpha_i)\prod_{C\in \mathcal{C}_G^+}\Gamma(\beta-|\alpha_C|)}.
			\end{aligned}
		\end{equation}
	\end{enumerate}
	\label{thm-norcons}
\end{Theorem}

\subsection{Independence structure}
\label{sec:org6c5fe5d}
\begin{Theorem}
	Let \(\mathcal{G}\) be a moral DAG with skeleton \(G\) with a parent function \(\mathfrak{pa}\).
	\begin{enumerate}[label={\arabic*)}, itemjoin={,}, itemjoin*={, and}]
		\item If \(X\sim \operatorname{Dir}_G(\alpha,\beta)\), and \(U^{\mathcal{G}}=u^{\mathcal{G}}(X)\) as defined in \eqref{eq-u-calG}, then random variables \(U^{\mathcal{G}}_i,\;i\in V\)  are independent. Moreover, for every \(i\in V\),
		\[U^{\mathcal{G}}_i \sim \operatorname{B}_I(\alpha_i,\beta+|\alpha_{\mathfrak{pa}(i)}|).\]
		\item If \(Y\sim \operatorname{IDir}_G(\alpha,\beta)\), and \(W^{\mathcal{G}}=w^{\mathcal{G}}(Y)\) as defined in \eqref{eq-w-calG}, then  random variables  \(W^{\mathcal{G}}_i,\;i\in V\)  are independent. Moreover, for every \(i\in V\),
		\[W^{\mathcal{G}}_i \sim \operatorname{B}_{II}(\alpha_i,\beta-|\alpha_{\overline{\mathfrak{pa}}(i)}|).\]
	\end{enumerate}
	\label{thm-independent-u-w}
\end{Theorem}
\begin{Definition}\label{def-G-Markov}	
	Let \(G=(V,E)\) be a simple undirected graph, we say that a random vector \(X=(X_i)_{i\in V}\) is \(G\)-one-dependent  if, for all \(A,B\) disconnected in \(G\), i.e. there is no edge between vertices in \(A\) and vertices in \(B\), we have
	\[X_A \text{ and }X_B \text{ are independent.}\]
\end{Definition}
\begin{Theorem}\label{thm-G-Markov-GDir}
	Let \(G\) be a finite decomposable graph. If  the distribution of \(X\) is either \(\operatorname{Dir}_G\) or \(\operatorname{IDir}_G\), then \(X\) is \(G\)-one-dependent.
\end{Theorem}

One may expect that if $X\sim \operatorname{Dir}_G$ and $C\subset V$, then $X_C\sim \operatorname{Dir}_{G_C}$, in particular $X_C$ should have Dirichlet distribution for $C\in\mathcal{C}_G$. However, this is not always the case. In particular, if $G=1-2-3$ and $X\sim\operatorname{Dir}_G(\alpha,\beta)$, then
\begin{align*}
f_{(X_1,X_2)}(x_1,x_2)&\propto x_1^{\alpha_1-1} x_2^{\alpha_2-1} \frac{(1-x_1-x_2)^{\alpha_3+\beta-1}}{(1-x_1)^{\alpha_3}} \mathbbm{1}_{0<x_1, x_2, x_1+x_2<1}\\
&\propto \frac{ f_{\operatorname{Dir}_{G_{\{1,2\}}}(\alpha_{\{1,2\}},\beta+\alpha_3)}(x_1,x_2)}{(1-x_1)^{\alpha_3}}.
\end{align*}
Below we present a surrogate.  Recall the definition of the mappings $M_G\ni x\mapsto x^C$ and $(0,\infty)^V\ni y\mapsto y^C$ from \eqref{eq-xC} and \eqref{eq-yC}. 
\begin{Theorem}\label{thm:dirC}
	Let $G$ be a  finite decomposable graph and $C\in\mathcal{C}_G$.
	\begin{enumerate}[label={\arabic*)}, itemjoin={,}, itemjoin*={, and}]
		\item 	If $X\sim \operatorname{Dir}_G(\alpha,\beta)$, then  $X^C\sim\operatorname{Dir}_{K_C}(\alpha_C,\beta)$.  Moreover, 
		$X_{V\setminus C}$ and $X^C$ are independent.
		\item 	If $Y\sim \operatorname{IDir}_G(\alpha,\beta)$, then $Y^C\sim\operatorname{IDir}_{K_C}(\alpha_C,\beta)$.  Moreover, $Y_{V\setminus C}$ and $Y^C$ are independent.
	\end{enumerate}
\end{Theorem}

\begin{Remark}\label{rem:inproof}
	Setting $C=\{i\}$ for a simplicial vertex $i\in V$, by \eqref{eq-xC}, we have $x^C_{i}=x_{i}$.
	Thus, if $X\sim \operatorname{Dir}_G(\alpha,\beta)$ (resp. $\operatorname{IDir}_G(\alpha,\beta)$), then for any simplicial vertex $i\in V$ one has $X_{i}\sim \operatorname{B}_I(\alpha_{i},\beta)$ (resp. $\operatorname{B}_{II}(\alpha_{i},\beta-\alpha_{i})$). 
\end{Remark}

\section{\texorpdfstring{$\operatorname{nm}_G$ and $\operatorname{mult}_G$ Bayesian models with  $\operatorname{Dir}_G$ and $\operatorname{IDir}_G$ priors}{nmG and multG Bayesian models with  Dir and IDir priors}}\label{six}
Consider a random vector $(N,X)$ such that $\mathbb P_X=\nu$ (the a priori distribution imposed on the parameter) and $\mathbb P_{N|X=x}=\mu_x$, $x\in\mathrm{supp}\,\nu$. We assume that observations $N^{(1)},\ldots,N^{(k)}$ are conditionally independent given $X$, where $X\sim\nu$ and $\mathbb P_{N^{(i)}|X}=\mathbb P_{N|X}$, $i=1,\ldots,k$.  

For a decomposable $G=(V,E)$ we will consider two Bayesian models with the graphical Markov structure. They are generated by 
\begin{enumerate}[label={\arabic*)}, itemjoin={,}, itemjoin*={, and}]
	\item $\mathbb P_{N|X}=\operatorname{nm}_G(r,X)$ and $\mathbb P_X=\operatorname{Dir}_G(\alpha,\beta)$; 
	\item $\mathbb P_{N|Y}=\operatorname{mult}_G(r,Y)$ and $\mathbb P_Y=\operatorname{IDir}_G(\alpha,\beta)$.
\end{enumerate}
\subsection{Conjugacy}

A standard requirement for a Bayesian model generated by $\mathbb P_{N|X}$ and $\mathbb P_X=\nu_\tau$, where $\tau$ is a parameter of the a priori law, is that the posterior distribution $\mathbb P_{X|N^{(1)},\ldots,N^{(k)}}$ is of the form $\nu_{g(\tau,N^{(1)},\ldots,N^{(k)})}$, i.e., $\tau$, the value of the parameter of the a priori law, is updated by the observation $N^{(1)},\ldots,N^{(k)}$. Then we say that $\nu_\tau$ is a conjugate prior in the model. Note also that $\mathbb P_{N^{(i)}}=\mathbb P_N$, $i=1,\ldots,k$. 

It appears that both models in which we are interested possess the conjugacy property. Moreover, the marginal distributions of $N$ are identifiable. 

\begin{Proposition}[Conjugate priors]	\label{prop-conjugate-priors}
	Let \(G=(V,E)\) be finite and decomposable. Assume that observations $(N^{(j)})_{j=1}^k$ follow a Bayesian model generated by $\mathbb P_{N|X}$ and $\mathbb P_X$.
	\begin{enumerate}
		\item  If $\mathbb P_{N|X}=\operatorname{nm}_G(r,X)$ and $\mathbb P_X=\operatorname{Dir}_G(\alpha,\beta)$, the posterior distribution is  \[\mathbb P_{X|N^{(1)},\ldots,N^{(k)}}=\operatorname{Dir}_G(\alpha+N^{(1)}+\ldots+N^{(k)},\beta+kr)\] and the  distribution of $N$ is 
		\begin{equation}\label{eq:DirNM}
			\mathbb{P}(N=n)= \gbinomm{|n|+r-1}{n}{G} \frac{K_G(\alpha,\beta)}{K_G(\alpha+n,\beta+r)},\quad n\in\mathbb{N}^V
		\end{equation}
		with $K_G$ given in \eqref{eq-KG-C-S} 
		(we say that $N$ follows a $G$-Dirichlet negative multinomial distribution).
		\item If $\mathbb P_{N|X}=\operatorname{mult}_G(r,X)$ and $\mathbb P_X=\operatorname{IDir}_G(\alpha,\beta)$, the posterior distribution is \[\mathbb P_{X|N^{(1)},\ldots,N^{(k)}}=\operatorname{IDir}_G(\alpha+N^{(1)}+\ldots+N^{(k)},\beta+kr)\]
		and the distribution of $N$ is
        \begin{equation}\label{eq:IDirMULT}
			\mathbb{P}(N=n)= \gbinom{n}{r}{G} \frac{k_G(\alpha,\beta)}{k_G(\alpha+n,\beta+r)},\quad n\in\mathbb{N}_{G,r}
		\end{equation}
		with $\mathbb{N}_{G,r}$ defined in \eqref{ngr} and $k_G$ given in \eqref{eq-kG-C-S} 
		(we say that $N$ follows a $G$-Dirichlet multinomial distribution).
	\end{enumerate}
\end{Proposition}

\begin{Remark}
	\label{rmk-posterior-likelihood} \	
	\begin{enumerate}[label={\arabic*)}, itemjoin={,}, itemjoin*={, and}]
		\item As we shall see in Section \ref{sec:org9c216d9}, both distributions \(\operatorname{Dir}_G\) and \(\operatorname{IDir}_G\) are strong hyper Markov \cite{Dawid1993}. Combining this with the fact that both \(\operatorname{nm}_G\) and \(\operatorname{mult}_G\) are \(G\)-Markov, Proposition 5.6 of \cite{Dawid1993} implies that distributions \eqref{eq:DirNM} and \eqref{eq:IDirMULT} also satisfy the global graph Markov property.
		\item In the case of a complete graph \(G\)  the distribution defined in \eqref{eq:DirNM} is known as the Dirichlet negative multinomial distribution.
		
		\item In the case of a complete graph \(G\)  the distribution \eqref{eq:IDirMULT} is closely related  to the Dirichlet multinomial distribution. Note that for the Dirichlet multinomial, the Dirichlet distribution is a prior on probabilities \(p\)  and here  we parameterize \(p\) by odds ratios, i.e. \(p_i=y_i/(1+\sum_{i=1}^n y_i)\), $i\in V$. Thus, the inverted Dirichlet distribution, as a prior after such reparametrization $p\mapsto y$, is conjugate to multinomial distribution.
	\end{enumerate}
\end{Remark}

\section{Characterizations of prior distributions through neutralities and the hyper Markov properties}
\label{sec:org9c216d9}
As a corollary of Theorem \ref{thm-strong-directed-hyper-markov-properties} and Theorem \ref{thm-independent-u-w}, we show that the graph negative multinomial (resp. graph multinomial distribution) with their parameter \(X\) (resp. \(Y\)) following  a graph Dirichlet distribution (resp. graph inverted Dirichlet distribution), are strong hyper Markov. The strong hyper Markov property is important in statistical graphical models, because it allows a simple decomposition of the Bayesian analysis into a collection of smaller problems, see e.g., \cite[Corollary 5.5]{Dawid1993}.  For our models the strong hyper Markov property turns out to be equivalent to neutrality properties of the Dirichlet and Inverted Dirichlet distributions. We begin with a review of neutrality properties of a Dirichlet distribution. In the next subsection, we interpret them as strong hyper Markov properties in the Bayesian models we consider.

Characterizations of the Dirichlet distribution through independence properties have a rich history. There are two distinct lines in this area, which are not directly related. One focuses on non-parametric prior in discrete models (\cite{Dawid1993}) and involves characterization via DAG-generated local and global independence of parameters,  \cite{HGC95,GH97,BW09,MW16}. This line can be understood within a framework of complete neutrality in contrast to characterizations based on neutrality with respect to partitions \cite{DR71,F73,JM80,BW07,SW14}. The independence properties considered below pertain to DAG-generated complete neutrality. 

We elaborate on this latter problem, which has so far been studied only in the context of complete graphs. Furthermore, we establish a connection between neutrality properties and strong hyper Markov property, see Theorem \ref{thm:simpleHyper}. Let $V=\{1,\ldots,d\}$ with $d\geq2$, and $G=K_V$ be a complete graph. Then, the set of DAGs with skeleton $G$ is in one-to-one correspondence with the set $S_d$ of permutations of $\{1,\ldots,d\}$ via the transformation $\mathcal G\equiv \sigma$ if and only if $\mathfrak{pa}_{{\mathcal G}}(\sigma(1))=\emptyset$ and  $\mathfrak{pa}_{{\mathcal G}}(\sigma(i))=\{\sigma_1,\ldots,\sigma_{i-1}\}$ for $i=2,\ldots,d$. Let $X$ be a random vector valued in the unit simplex $T_d=\{(x_1,\ldots,x_d)\in(0,\infty)^d\colon |x_V|<1\}=M_{G}$. According to \eqref{eq-u-calG}, the random  vector $u^{\mathcal{G}}(X)$ is given by
\[
\left(X_{\sigma(1)},\frac{X_{\sigma(k)}}{1-\sum_{i=1}^{k-1} X_{\sigma(i)}},\; k=2,\ldots,d \right). 
\]
If $X$ is such that the components of the above vector are independent, then we say that $X$ is completely neutral with respect to $\sigma$. 
It is well known (see, e.g., \cite{JM80}) that if a $T_d$-valued random vector $X$ is completely neutral with respect to any permutation $\sigma\in S_d$, then $X$ follows a Dirichlet distribution. In Theorem \ref{thm:charact}, we generalize this result to arbitrary decomposable graphs, along with an analogous characterization of the $G$-inverted Dirichlet distribution.   

The main result in this section is the following result. 
\begin{Theorem}\label{thm:charact}
	Let $G=(V,E)$ be a decomposable connected graph with at least two vertices.
	\begin{enumerate}[label={\arabic*)}, itemjoin={,}, itemjoin*={, and}]
		\item Let $X$ be an $M_G$-valued random vector. If for every moral DAG $\mathcal{G}$ with skeleton $G$  random variables $u^{\mathcal{G}}_i(X)$, $i\in V$, are independent, then $X$ has a $G$-Dirichlet distribution.
		\item Let $Y$ be a $(0,\infty)^V$-valued random vector.  If for every moral DAG $\mathcal{G}$ with skeleton $G$ random variables $w^{\mathcal{G}}_i(Y)$, $i\in V$, are independent, then $Y$ has a $G$-inverted Dirichlet distribution.
	\end{enumerate}
\end{Theorem}

Let us mention that for the Dirichlet distribution, stronger characterization results are available in the literature. Complete neutrality with respect to all DAGs, as given above, clearly implies that for all $i\in V$, 
$X_{i}$ and $X_{V\setminus\{i\}}/(1-X_i)$ are independent.
The latter condition is referred to as {$X_i$ being neutral in the vector $X$}. It is known that if $X_i$ is neutral in $X$ for each $i\in V$, then $X$ necessarily follows the Dirichlet distribution, see e.g., Theorem 1 in \cite{JM80}. This provides an example of a characterization of the Dirichlet by through neutrality with respect to partitions of $V$ (here there are $|V|$ partitions: $\pi_i=\{\{i\},\,V\setminus\{i\}\}$, $i\in V$). In general, by neutrality of $X$ with respect to a partition $\pi=(P_1,\ldots,P_k)$ of $V$ is defined as the independence of the following $k+1$ random vectors:
\[
\left(\sum_{i\in P_j}X_i\right)_{j=1,\ldots,k},\quad \left(\frac{X_\ell}{\sum_{i\in P_1}X_i}\right)_{\ell\in P_1},\ldots,\left(\frac{X_\ell}{\sum_{i\in P_k}X_i}\right)_{\ell\in P_k}.
\]
The literature explores various sets of partitions that are sufficient for characterizing the Dirichlet distribution. Notably, in the case of the complete graph, properties of complete neutrality with respect to only two appropriately chosen permutations (i.e., independence of the coordinates of $u^\mathcal G$ for only two DAGs) are sufficient to characterize the standard Dirichlet distribution; see \cite{GH97}. This result was extended in \cite{MW16} to characterize the hyper-Dirichlet distribution (see Remark \ref{hyperDir}) using global and local independence properties related to a separating family of DAGs.

\subsection{Hyper Markov properties}
We recall some definitions from \cite{Dawid1993}, translated to  the Bayesian setting we consider. We fix a DAG \(\mathcal{G}\) with skeleton \(G=(V,E)\), and a random variable $N$ with a conditional \(G\)-negative multinomial distribution \(\mathbb P_{N|X}= \operatorname{nm}_G(r,X)\), where \(X\sim \operatorname{Dir}_G(\alpha,\beta)\). Given \(X\), the conditional distribution of \(N\), \(\mathbb{P}_{N|X}\),  can be identified with a collection of random variables 
\begin{equation}\label{strhipMar}\{\mathbb{P}(N=n|X),\ n\in \mathbb{N}^V\}=: (p_n(X))_{n\in \mathbb{N}^V}=p(X)=p.\end{equation}
For \(A \subset V\), the marginal of the conditional distribution of \(N\) given \(X\) on the set \(A\), \(\mathbb{P}_{N_A|X}\), can be identified with a collection of random variables 
\[\{\mathbb{P}(N_A=n_A|X),\ n_A\in \mathbb{N}^A\}=: (p^A_{n_A}(X))_{n_A\in \mathbb{N}^A}=p_A.\]
For \(A,B \subset V\), given $X$, the conditional distribution of \(N_B\) with respect to \(N_A\) can be identified with a collection of random variables
\[\left\{\frac{\mathbb{P}(N_{A\cup B}=n_{A\cup B}|X)}{\mathbb{P}(N_A=n_A|X)},\ n_{A\cup B}\in \mathbb{N}^{A\cup B}\right\}=:(p^{B|A}_{n_{A\cup B}}(X))_{n_{A\cup B}\in \mathbb{N}^{A\cup B}}=p_{B|A}.\]

Several notions of hyper Markov properties of Bayesian graphical models were introduced in terms of these families of r.v. in \cite{Dawid1993}.
\begin{Definition}	\label{def-hyper-Markovs}
	Consider the family of r.v. \(p=(p_n(X))_{n\in \mathbb{N}^V}\) as above. We say that
	\begin{enumerate}[label={\arabic*)}, itemjoin={,}, itemjoin*={, and}]
		\item \(p\) is hyper Markov over \(G\) if for all decompositions \((A,B,C)\)   of \(G\) (recall Definition \ref{def-decomposable-graph})
		\begin{equation}
			\label{eq-hyper-markov}
			\begin{aligned}
				p_{A\cup C} \text{ and } p_{B\cup C} \text{ are conditionally independent given }p_{C}.
			\end{aligned}
		\end{equation}
		\item \(p\) is directed hyper Markov over \(\mathcal{G}\) if for all \(i\in V\),
		\begin{equation}
			\label{eq-directed-hyper-markov}
			\begin{aligned}
				p_{\overline{\mathfrak{pa}}(i)} \text{ and }p_{\mathfrak{nde}(i)} \text{ are conditionally independent given }p_{\mathfrak{pa}(i)}.
			\end{aligned}
		\end{equation}
		\item \(p\) is strong hyper Markov over \(G\) if for all decompositions \((A,B,C)\)  of \(G\), 
		\begin{equation}
			\label{eq-str-hyper-markov}
			\begin{aligned}
				p_{B\cup C|A\cup C} \text{ and  }p_{A\cup C} \text{ are independent}.
			\end{aligned}
		\end{equation}
		\item \(p\) is strong directed hyper Markov over \(\mathcal{G}\) if for all \(i\in V\), 
		\begin{equation}
			\label{eq-str-dir-hyper-markov}
			\begin{aligned}
				p_{\{i\}|\mathfrak{pa}(i)} \text{ and }p_{\mathfrak{nde}(i)} \text{ are independent}.
			\end{aligned}
		\end{equation}
	\end{enumerate}
\end{Definition}
In \cite{Dawid1993}, it is proved that, for a fixed finite decomposable graph $G$,
\begin{enumerate}[label={\arabic*)}, itemjoin={,}, itemjoin*={, and}]
	\item \eqref{eq-str-dir-hyper-markov} holds for all moral DAG \(\mathcal{G}\) (with skeleton \(G\)) is equivalent to \eqref{eq-str-hyper-markov}  holds for \(G\),
	\item \eqref{eq-directed-hyper-markov} holds for one moral DAG \(\mathcal{G}\) (with skeleton \(G\)) is equivalent to \eqref{eq-hyper-markov} holds for \(G\),
	\item \eqref{eq-str-hyper-markov} holds for \(G\) implies \eqref{eq-hyper-markov} holds for the same \(G\).
\end{enumerate}

\begin{Corollary}
	\label{coro-hyper-Markov-for-Gnm}
	Given a finite decomposable \(G\) and any moral DAG \(\mathcal{G}\) with skeleton \(G\),  Bayesian models generated by 
	\begin{enumerate}[label={\arabic*)}, itemjoin={,}, itemjoin*={, and}]
		\item  $\mathbb P_{N|X}=\operatorname{nm}_G(r,X)$ and $\mathbb P_X=\operatorname{Dir}_G(\alpha,\beta)$

        and 
		\item $\mathbb P_{N|Y}=\operatorname{mult}_G(r,Y)$ and $\mathbb P_Y=\operatorname{IDir}_G(\alpha,\beta)$
	\end{enumerate}
	are strong directed hyper Markov over \(\mathcal{G}\).
\end{Corollary}
\begin{proof}
	This is an immediate consequence of Theorem \ref{thm-strong-directed-hyper-markov-properties} together with formulas \eqref{eq-u-calG} and \eqref{eq-w-calG} and  independence properties from  Theorem \ref{thm-independent-u-w}.
\end{proof}

As a consequence, both Bayesian graphical models verify all notions of hyper Markov properties in Definition \ref{def-hyper-Markovs}.

Now, we reformulate the strong hyper Markov property of the priors. 

\begin{Theorem}\label{thm:simpleHyper}
	Let $G$ be a connected decomposable graph. 
	\begin{enumerate}[label={\arabic*)}, itemjoin={,}, itemjoin*={, and}]
		\item Let us fix $r>0$ and assume that $\mathbb{P}_{N|X}=\operatorname{nm}_G(r,X)$. The random vector $p$, as defined in \eqref{strhipMar},  is strong (directed) hyper Markov  if and only if for every moral DAG $\mathcal{G}$ with skeleton $G$ random vector $u^{\mathcal{G}}(X)$, defined in \eqref{eq-u-calG},  has independent components. 
		\item Let us fix $r\in\mathbb N_+$ and assume that $\mathbb P_{N|Y}=\operatorname{mult}_G(r,Y)$. The random vector $p$, as defined in \eqref{strhipMar} with $Y$ instead of $X$,   is strong (directed) hyper Markov if and only if for every moral DAG $\mathcal{G}$ with skeleton $G$ random vector $w^{\mathcal{G}}(Y)$, defined in \eqref{eq-w-calG}, has independent components. 
	\end{enumerate}
\end{Theorem}

\section{Proofs}
\label{sec:org27fcc87}

\subsection{Proofs from Section \ref{sec:orge67f81b}}
\begin{proof}[Proof of Lemma \ref{lem-properties-of-clique-polynomials}]
	\begin{enumerate}[label={\arabic*)}, itemjoin={,}, itemjoin*={, and}]
		\item Clearly, with $A^c=V\setminus A$  we have
		\[\left. \Delta_G(x) \right|_{x_{A^c} = 0} = \sum_{\substack{C\in \mathcal{C}_{G^*} \\ C\cap A^c = \emptyset}} (-1)^{|C|} \pi_C(x)\]
		and
		\[
        \{C\in \mathcal{C}_{G^*}\colon \ C\cap A^c = \emptyset \}=\{C\in \mathcal{C}_{G^*}\colon \ C \subset A\}= \mathcal{C}_{G_A^*},\]
		thus the result follows by referring to the fact that $(G^*)_A=(G_A)^*$.
		\item We first prove that
		\begin{equation}
			\label{eq-clique-in-union}
			\begin{aligned}
				\{C\in \mathcal{C}_{G^*}\colon \ C \subset A\cup B\} = \{C_A \cup C_B\colon \ C_A \in \mathcal{C}_{G_A^*} \text{ and } C_B \in C_{G_B^*}\}.
			\end{aligned}
		\end{equation}
		Given \(C\in \mathcal{C}_{G^*}\) such that \(C\subset A\cup B\),  we have \(C=C_A \cup C_B\) with \(C_A: = C\cap A\in \mathcal{C}_{G_A^*},\ C_B: = C \cap B\in \mathcal{C}_{G_B^*}\), which proves the inclusions from the left to the right in \eqref{eq-clique-in-union}. On the other hand, given any \(C_A\in \mathcal{C}_{G_A^*},\ C_B\in \mathcal{C}_{G_B^*}\), clearly \(C_A\cup C_B \subset A\cup B\). If \(C_A \cup C_B \notin \mathcal{C}_{G^*}\), then \(G_{C_A\cup C_B}\)  has at least one edge, but there are no edges in \(G_{C_A}\) and \(G_{C_B}\). Therefore  this edge connects \(C_A\subset A\) and \(C_B\subset B\), i.e. it connects \(A\) and \(B\), a contradiction. Thus \eqref{eq-clique-in-union} is proved.
		
		Using \eqref{eq-clique-in-union} we can write
		\[\Delta_{A\cup B}= \sum_{\substack{ C\in \mathcal{C}_{G^*} \\ C \subset  A\cup B} }(-1)^{|C|} \pi_C = \sum_{\substack { C_A \in \mathcal{C}_{G_A^*} \\ C_{B} \in \mathcal{C}_{G_B^*} }}(-1)^{|C_A\cup C_B|} \pi_{C_A\cup C_B}.\]
		Note that \(C_A\cap C_B = \emptyset\) implies that \(\pi_{C_A \cup C_B} =\pi_{C_A} \pi_{C_B}\) and \(|C_A\cup C_B| =|C_A| +|C_B|\), thus we have
		\begin{equation*}
			\begin{aligned}
				\Delta_{A\cup B} & = \sum_{\substack { C_A \in \mathcal{C}_{G_A^*} \\ C_B \in \mathcal{C}_{G_B^*} }}(-1)^{|C_A|+| C_B|} \pi_{C_A}\pi_{C_B} 
			 =\left( \sum_{C_A \in \mathcal{C}_{G_A^*}} (-1)^{|C_A|} \pi_{C_A} \right) \left( \sum_{C_B \in \mathcal{C}_{G_B^*}} (-1)^{|C_B|} \pi_{C_B} \right)\\
             &= \Delta_A \Delta_B.
			\end{aligned}
		\end{equation*}
		\item  By 1) we have
		\[\Delta_{(\overline{\mathfrak{nb}}_G(i))^c} = \sum_{\substack {C'\in \mathcal{C}_{G^*} \\ C'\cap \overline{\mathfrak{nb}}_G(i) =\emptyset}} (-1)^{|C'|} \pi_{C'}.\]
		Note that  \((\overline{\mathfrak{nb}}_G(i))^c = \mathfrak{nb}_{G^*}(i)\). Moreover, if \(C'\in \mathcal{C}_{G^*}\) and \(C' \subset  \mathfrak{nb}_{G^*}(i)\), then \(C'\cup \{i\}\in \mathcal{C}_{G^*}\). Therefore,
		\[-\pi_i \Delta_{(\overline{\mathfrak{nb}}_G(i))^c} = -\pi_i \sum_{\substack{C'\in \mathcal{C}_{G^*} \\ C' \subset  \mathfrak{nb}_{G^*}(i) } }(-1)^{|C'|} \pi_{C'} = \sum_{ \substack{ C' \in \mathcal{C}_{G^*} \\ i\in C'}}(-1)^{|C'|} \pi_{C'}.\]
		Thus, for $C\in\mathcal C_G$,
		\begin{equation}
			\label{eq-abeq}
			\begin{aligned}
				\Delta_{V \setminus C} - \sum_{i\in C} \pi_i \Delta_{(\overline{\mathfrak{nb}}_G(i))^c} = \sum_{\substack{C'\in \mathcal{C}_{G^*}\\ C\cap C' =\emptyset } } (-1)^{|C'|}\pi_{C'} + \sum_{i\in C} \sum_{\substack{C'\in \mathcal{C}_{G^*}\\ i\in C'}}(-1)^{|C'|}\pi_{C'}.
			\end{aligned}
		\end{equation}
		Since \(C\) is a clique in $G$, for \(\{i,j\}\in E\) and \(i,j\in C\), we have
		\[\{C'\in \mathcal{C}_{G^*}:\ i\in C'\} \cap \{C'\in \mathcal{C}_{G^*}:\ j\in C'\}=\emptyset,\]
		therefore
		\[\sum_{i\in C}\sum_{\substack{C'\in \mathcal{C}_{G^*}\\ i\in C'}}(-1)^{|C'|}\pi_{C'} = \sum_{\substack { C'\in \mathcal{C}_{G^*} \\ C\cap C' \ne \emptyset}}(-1)^{|C'|}\pi_{C'}.\]
		Plugging the above in \eqref{eq-abeq} yields  the first part of \eqref{eq-Delta-G-C}. The second one for $\delta$ follows from the fact that $\delta(x)=\Delta(-x)$. Taking $C=\{i\}$ in \eqref{eq-Delta-G-C} one gets \eqref{eq-Delta_GV-v}.
		
		\item We partition $\mathcal{C}_{G^*}$ into three disjoint sets:  $A_1= \{ C\in \mathcal{C}_{G^*} \colon i_0\notin C, C\cap \mathfrak{nb}_G(i_0)\neq \emptyset\}$, $A_2=\{ C\in \mathcal{C}_{G^*} \colon C\cap \overline{\mathfrak{nb}}_G(i_0)=\emptyset\}$ and $A_3=\{ C\in \mathcal{C}_{G^*} \colon i_0\in C\}$.
		Since $i_0$ is simplicial  in $G$, $\overline{\mathfrak{nb}}(i_0)$ is a maximal clique in $G$. Thus, if $i_0\in C\in \mathcal{C}_{G^*}$, then $C\cap \mathfrak{nb}_G(i_0)=\emptyset$ and so $A_3=\{C\cup\{i_0\}\colon C\in A_2\}$. Therefore, for $C':=C\cup\{i_0\}\in A_3$ with the unique $C\in A_2$, we have $\pi_{C'}(x)=-x_{i_0}\pi_C(x)$. Moreover, $A_1 =  \{ C\in\mathcal{C}_{\tilde G^*}\colon C\cap \mathfrak{nb}_G(i_0)\neq \emptyset\}$ and similarly
		$A_2 = \{ C\in\mathcal{C}_{\tilde G^*}\colon C\cap \mathfrak{nb}_G(i_0)=\emptyset\}$.
		Thus,
		\begin{align}\label{Dell}
			\Delta_G(x)=\sum_{C\in A_1\cup A_2\cup A_3}(-1)^{|C|}\pi_C(x) = \sum_{C\in A_1}(-1)^{|C|}\pi_C(x)+ (1-x_{i_0})\sum_{C\in A_2}(-1)^{|C|}\pi_C(x).
		\end{align}
		Note that $A_1\cup A_2=\mathcal{C}_{\tilde G^*}$ and thus
		\begin{align}\label{dell}
			\Delta_{\tilde G}(\tilde x) &= \sum_{C\in A_1} (-1)^{|C|}\pi_C(\tilde x) + \sum_{C\in A_2} (-1)^{|C|}\pi_C(\tilde x).
		\end{align}
		For $C\in A_2$, we have $\pi_C(\tilde x)=\pi_C(x)$ since $\tilde x_{V\setminus \overline{\mathfrak{nb}}_G(i_0)}=x_{{V\setminus \overline{\mathfrak{nb}}_G(i_0)}}$.  For $C\in A_1$, given that $i_0$ is simplicial, we have $|C\cap \mathfrak{nb}_G(i_0)|=1$ and therefore $\pi_C(\tilde x)=(1-x_{i_0})^{-1}\pi_C(x)$. On comparing right-hand sides of \eqref{Dell} and \eqref{dell} these two facts finish the proof  of \eqref{eq:Deltax}.
	\end{enumerate}
\end{proof}
\begin{proof}[Proof of Lemma \ref{lem-clique-poly-new-coord}]
	\begin{enumerate}[label={\arabic*)}, itemjoin={,}, itemjoin*={, and}]
		\item Thanks to \eqref{eq-ilocz},  without  loss of generality,  we can assume that \(G\) is connected. We use induction  with respect to  number of vertices of \(G\). If \(|V|=1\), then \eqref{eq-deltaG} is trivial.
		
		Let \(\mathcal{G}\) be a moral DAG with skeleton \(G=(V,E)\). Assume \eqref{eq-deltaG} holds for graphs (not necessarily connected) with number of vertices less than \(|V|\). Let \(i_0\) be the  source vertex in \(\mathcal{G}\), i.e., $\mathfrak{pa}(i_0)=\emptyset$. Denote \(V' = V \setminus \{i_0\}\), \(G'=G_{V'}\) and \(\mathcal{G}' = \mathcal{G}_{V'}\). Since \(G\) is connected, \(\mathfrak{de}(i_0)=V'\) and \(\overline{\mathfrak{de}}(i_0) = V\). By \eqref{eq-u-calG} we have
		\[u^{\mathcal{G}}_{i_0}(x) = 1- \frac{\Delta_G(x)}{\Delta_{G'}(x_{V'})},\]
		that is, \(\Delta_G(x)=(1-u^{\mathcal{G}}_{i_0}(x)) \Delta_{G'}(x_{V'})\). Since $|V'|<|V]$, by the induction assumption, 
		\[\Delta_{G'}(x_{V'}) = \prod_{i\in V'}(1-u^{\mathcal{G'}}_i(x_{V'})).\]
		Since  \(u^{\mathcal{G}}_i\) depends only on \(x_{\overline{\mathfrak{de}}(i)}\),  see \eqref{eq-u-calG}, we have
		\[u^{\mathcal{G}}_i(x)=u^{\mathcal{G}'}_i(x_{V'}),\ i\in V',\]
		and therefore
		\[\Delta_G(x)=(1-u^{\mathcal{G}}_{i_0}(x)) \prod_{i\in V'} (1-u^{\mathcal{G}'}_i(x_{V'}))=\prod_{i\in V} (1-u^{\mathcal{G}}_i(x)).\]
		The same argument works for \(w_i^{\mathcal{G}}\).
		
		\item We first prove \eqref{eq-x-in-u}.  Let again \(i_0\) be the  source vertex  in \(\mathcal{G}\) as in 1). Then, \(\overline{\mathfrak{de}}(i_0)\) contains all vertices in the connected component of \(G\) that contains \(i_0\). Denote this component by \(H=(V_H,E_H)\) and define \(H'=H_{V'}\) and \(H'' = H_{V''}\), where \(V' = V_H \setminus \{i_0\}\) and \(V'' = V_H \setminus \overline{\mathfrak{ch}}(i_0)\). By \eqref{eq-u-calG} and  \eqref{eq-Delta-G-C} of Lemma \ref{lem-properties-of-clique-polynomials}  applied for $G=H$ and $C=\{i_0\}$, we have
		\[u^{\mathcal{G}}_{i_0} = 1- \frac{\Delta_H(x_{V_H})}{\Delta_{H'}(x_{V'})} = 1- \frac{\Delta_{H'}(x_{V'})-x_{i_0} \Delta_{H''}(x_{V''})}{\Delta_{H'}(x_{V'})},\]
		therefore,
		\begin{equation}\label{xi0} x_{i_0}= u^{\mathcal{G}}_{i_0} \frac{\Delta_{H'}(x_{V'})}{\Delta_{H''}(x_{V''})}.\end{equation}
		Let \(\mathcal{H}'=\mathcal{G}_{V'}\) and \(\mathcal{H}'' = \mathcal{G}_{V''}\). By 1) proved just above,  we have
		\[\Delta_{H'}(x_{V'}) = \prod_{i\in V'}(1-u^{\mathcal{H}'}_i(x_{V'}))\quad \text{ and }\quad  \Delta_{H''}(x_{V''}) = \prod_{i\in V''} (1-u^{\mathcal{H}''}_i(x_{V''})).\]
		Note that \((u^{\mathcal{G}}_i(x))_{i\in V'} = (u^{\mathcal{H}'}_i(x_{V'}))_{i\in V'}\), \((u^{\mathcal{G}}_i(x))_{i\in V''} = (u^{\mathcal{H}''}_i(x_{V''}))_{i\in V''}\). Moreover, \(V' \setminus  V'' = \mathfrak{ch}(i_0)\).  Thus, in view of \eqref{xi0} we get 
		\[x^{\mathcal{G}}_{i_0}(u)= x_{i_0} = u_{i_0}^{\mathcal G} \frac{\prod_{i\in V'}(1-u^{\mathcal{G}}_i(x))}{\prod_{i\in V''}(1-u^{\mathcal{G}}_i(x))}=u_{i_0}^{\mathcal G} \prod_{i\in \mathfrak{ch}(i_0)}(1-u^{\mathcal{G}}_i(x)).\]
		Eq. \eqref{eq-x-in-u} now follows from inductive removal of  source vertices  in subsequent graphs, when repeating the above argument.
		
		Next for every \(x\in M_G\), by \eqref{eq-Delta-G-C} of Lemma \ref{lem-properties-of-clique-polynomials}, we have
		\[\Delta_{\overline{\mathfrak{de}}(i)}(x)= \Delta_{\mathfrak{de}(i)}(x)-x_i \Delta_{\mathfrak{de}(i) \setminus  \mathfrak{ch}(i)}(x),\]
		which implies that \(\Delta_{\overline{\mathfrak{de}}(i)}(x) < \Delta_{\mathfrak{de}(i)}(x)\) for all \(i\in V\); therefore, in view of \eqref{eq-u-calG}, we have \(u^{\mathcal{G}}(M_G) \subset  (0,1)^V\).
		
		It remains to prove that \(x^{\mathcal{G}}((0,1)^V) \subset  M_G\). We use induction with respect to the number of elements of \(V\). The case \(|V|=1\)  is trivial. Consider \(|V| > 1\). For arbitrary \(u\in (0,1)^V\), denote \(x=x^{\mathcal{G}}(u)\). By \eqref{eq-deltaG} we have \(\Delta_G(x) > 0\). Let \(j_0\) be arbitrary vertex in \(V\) and \(V'=V \setminus \{j_0\}\). We will prove that \(\Delta_G(\widehat{x}) = \Delta_{V'}(\widehat{x}) > 0\), where \(\widehat{x}\) is  obtained  by plugging \(x_{j_0}=0\) in $x$.
		
		Denote \(\mathcal{G}'=\mathcal{G}_{V'}\), and note that \(\mathcal{G}'\) is a moral DAG with skeleton \(G' = G_{V'}\). Define a vector \(\widehat{u}\in \mathbb{R}^{V'}\) by \(x^{\mathcal{G}'}(\widehat{u}) = x_{V'} = (x^{\mathcal{G}}_i(\widehat u))_{i\in V'}\), i.e., by the system of equations
		\[x^{\mathcal{G}}_i=\widehat{u}_i \prod_{j\in \mathfrak{ch}_{\mathcal{G}'}(i)}(1-\widehat{u}_j) = x_i= u_i \prod_{j\in \mathfrak{ch}_\mathcal{G}(i)}(1-u_i),\ i\in V'.\]
		Note that, in view of the fact that $x_{j_0}=0$, for \(i\notin \overline{\mathfrak{an}}_{\mathcal{G}}(j_0)\) we have \(\widehat{u}_i=u_i\) and for \(i\in \mathfrak{an}_{\mathcal{G}}(j_0)\) the above equality yields 
		\[\widehat{u}_i = u_i \prod_{j\in \mathfrak{ch}_{\mathcal{G}}(i)} \frac{1-u_j}{1-\widehat{u}_j},\]
		where, if  $j_0\in\mathfrak{ch}_{\mathcal G}(i)$ in the last formula, we insert \(\widehat{u}_{j_0}=0\). Consider the last formula as a recursion with respect to the ``directed" distance from \(j_0\) in \(\mathcal{G}\).  We pick \(j_1\) such that \(\mathfrak{ch}_{\mathcal{G}}(j_1)\cap \overline{\mathfrak{an}}_{\mathcal{G}}(j_0) = \{j_0\}\). Then, we have \(\widehat{u}_{j_1} = u_{j_1} (1-u_{j_0})\in(0,\,u_{j_1})\), since \(u_{j_0} > 0\). Consequently,  $\tfrac{1-u_{j_1}}{1-\widehat u_{j_1}}\in(0,1)$.  Iterating this procedure by taking $j_k$ such that $\mathfrak{ch}_\mathcal G(j_k)\cap \overline{\mathfrak{an}}_{\mathcal{G}}(j_0)\subset\{j_0,j_1,\ldots,j_{k-1}\}$  we conclude that \(\widehat{u}\in (0,1)^{V'}\) and \(\Delta_{G'}(x^{\mathcal{G'}}(\widehat{u})) > 0\) by the induction assumption. Finally, note that \(\Delta_{G'}(x^{\mathcal{G}'}(\widehat{u})) = \Delta_{G'}(x)= \Delta_G(\widehat{x})\).
		
		We can repeat the argument by assigning zero to more coordinates of \(x\). In this way we obtain that \(\Delta_{G_A}(x) > 0\) for any \(A \subset  V\), i.e., we have \(x=x^{\mathcal{G}}(u) \in M_G\). Consequently, \(u^{\mathcal{G}}\) is a bijection, and referring to its definition, we conclude that it is a diffeomorphism.
		
		\item Eq. \eqref{eq-y-in-w} follows from Eq. \eqref{eq-x-in-u} by the identity \(\delta_G(x)=\Delta_G(-x)\), it then also follows that \(y^{\mathcal{G}}\) is an automorphism on \((0,\infty)^V\).
		
		\item Immediate consequences of 2) and 3).
	\end{enumerate}
\end{proof}

\subsection{Proofs from Section \ref{sec:coefficients}}

\begin{proof}[Proof of Lemma \ref{lem:NGr}]
	Define $\varepsilon\colon \mathcal{C}_{G^*}\to \{0,1\}^V$ as the indicator function that represents the occurrence of each vertex in a clique. Specifically, for any $C\in\mathcal{C}_{G^*}$, we have $\varepsilon(C)_i = \mathbbm{1}_C(i)$ for $i\in V$.  A clique $C$ belongs to $\mathcal{C}_{G^*}$ if and only if $C$ contains at most one vertex from every maximal clique of $G$. Clearly, $\varepsilon(\mathcal{C}_{G^*})= \mathbb{N}_{G,1}$, and thus $\varepsilon^{-1}(n)\in\mathcal{C}_{G^\ast}$  if and only if $n\in\mathbb{N}_{G,1}$. Therefore, by comparing  \eqref{eq-clique-polynomial-deltaG} with Definition \ref{Ddel} (1), we observe that 
	\[
	\gbinom{1}{n}{G}= \mathbbm{1}_{\mathcal{C}_{G^*}}(\varepsilon^{-1}(n)) = \mathbbm{1}_{\mathbb{N}_{G,1}}(n).
	\]
    Since
\[
    \delta_G(x)^r = \left( \sum_{k\in\mathbb{N}_{G,1}} x^k\right)^r = \sum_{k_1\in\mathbb{N}_{G,1}} \ldots \sum_{k_r\in\mathbb{N}_{G,1}}  x^{\sum_{j=1}^r k_j},
    \]
    we obtain 
\[
	\gbinom{r}{n}{\,G} =\sum_{k_1\in \mathbb{N}_{G,1}} \ldots  \sum_{k_r\in\mathbb{N}_{G,1}} \mathbbm{1}_{n = \sum_{j=1}^r k_j}.
	\]
	Therefore, $\gbinom{r}{n}{\,G}\neq 0$ if  only if 
	\[
	n\in \left\{ k_1+\ldots+k_r\colon k_j\in\mathbb{N}_{G,1}, j=1,\ldots, r \right\} = \mathbb{N}_{G,r}. 
	\]
\end{proof}

\begin{proof}[Proof of Lemma \ref{lem:constant-recustion}]
	Take $x\in \mathbb{R}^V$ and define $\check{x}\in \mathbb{R}^{\tilde{V}}$ as in Lemma \ref{lem-properties-of-clique-polynomials} 4). We have
	\[
	\check{x}^{n_{\tilde V}} = \prod_{i\in \tilde{V}\setminus \mathfrak{nb}_G(i_0)} x_i^{n_i} \prod_{i\in\mathfrak{nb}_G(i_0)} \left( \frac{x_i}{1+x_{i_0}}\right)^{n_i} = x_{\tilde{V}}^{n_{\tilde V}} (1+x_{i_0})^{-|n_{\mathfrak{nb}_G(i_0)}|}.
	\]
	Thus, 	by the second equality in \eqref{eq:Deltax} of Lemma \ref{lem-properties-of-clique-polynomials} 4), we obtain 
	\begin{align*}
		\delta_G^r(x) &= (1+x_{i_0})^r\delta_{\tilde{G}}^r(\check{x})
		= (1+x_{i_0})^r \sum_{n_{\tilde V}\in \mathbb{N}^{\tilde V}} \gbinom{r}{n_{\tilde V}}{\tilde G} \check{x}^{n_{\tilde V}} 
		=  \sum_{n_{\tilde V}\in \mathbb{N}^{\tilde V}} \gbinom{r}{n_{\tilde V}}{\tilde G} x_{\tilde{V}}^{n_{\tilde V}} (1+x_{i_0})^{r-|n_{\mathfrak{nb}_G(i_0)}|} \\
		&=	 \sum_{n_{\tilde{V}}\in \mathbb{N}^{\tilde V} }\sum_{n_{i_0}=0}^{r-|n_{\mathfrak{nb}_G(i_0)}|}\,\gbinom{r}{n_{\tilde V}}{\tilde G}\,  \binom{r-|n_{\mathfrak{nb}_G(i_0)}|}{n_{i_0}} x_{\tilde{V}}^{n_{\tilde V}} x_{i_0}^{n_{i_0}} = \sum_{n\in \mathbb{N}^{V} } \gbinom{r}{n_{\tilde V}}{\tilde G} \binom{r-|n_{\mathfrak{nb}_G(i_0)}|}{n_{i_0}} x^n,
	\end{align*}
	where we used the convention that $\binom{a}{b}=0$ if $a<0$. In view of Definition \ref{Ddel} (1) and by the uniqueness of the coefficients, we obtain the first part of the assertion. 
	
	The second part follows similarly, using the identity $\Delta_G(x)=(1-x_{i_0})\Delta_{\tilde{G}}(\tilde{x})$ and the series expansion
	\[
	(1-x_{i_0})^{-r-|n_{\mathfrak{nb}_G(i_0)}|} = \sum_{n_{i_0}=0}^\infty  \binom{n_{i_0}+|n_{\mathfrak{nb}_G(i_0)}|+r-1}{n_{i_0}} x_{i_0}^{n_{i_0}}.
	\]
\end{proof}

\begin{proof}[Proof of Lemma \ref{lem:explicit-constants}]
	We proceed by induction on the number of vertices in the graph. The assertion is trivial if $|V|=1$. Suppose that the assertion holds for all decomposable graphs $G$ with $p$ vertices, and consider a graph with $p+1$ vertices. Then, there exists a simplicial vertex $i_0$, and by Lemma \ref{lem:constant-recustion}, 
\begin{equation}\label{(rn)}
	\gbinom{r}{n}{G} = \gbinom{r}{n_{\tilde V}}{\tilde G} \binom{r-|n_{\mathfrak{nb}_G(i_0)}|}{n_{i_0}} = \frac{\prod_{C\in\mathcal{C}_{\tilde G}^+} \binom{r}{n_C}}{\prod_{S\in\mathcal{S}_{\tilde G}^-} \binom{r}{n_S}^{\nu_S^{\tilde G}}}  \binom{r-|n_{\mathfrak{nb}_G(i_0)}|}{n_{i_0}}
	\end{equation}
	where $\tilde{V}=V\setminus\{i_0\}$, $\tilde{G} = G_{\tilde{V}}$, and the latter equality follows from the induction hypothesis. Since $i_0$ is simplicial in $G$, we have $\overline{\mathfrak{nb}}_G(i_0)\in\mathcal{C}^+_G$. We consider two cases: (a) $\mathfrak{nb}_G(i_0)\in\mathcal{C}_{\tilde G}^{+}$ and (b) $\mathfrak{nb}_G(i_0)\notin\mathcal{C}_{\tilde G}^{+}$. 
	
	In case (a), we have  $\mathcal{C}_G^+ = (\mathcal{C}_{\tilde G}^+\setminus \{ \mathfrak{nb}_G(i_0)\} )\cup \{\overline{\mathfrak{nb}}_G(i_0)\}$ and $\nu_S^G=\nu_S^{\tilde{G}}$ for all $S\in \mathcal{S}_G^- = \mathcal{S}_{\tilde G}^-$.  Thus, in view of \eqref{(rn)},
	\[
	\gbinom{r}{n}{G} = \frac{\prod_{C\in\mathcal{C}_{G}^+} \binom{r}{n_C}}{\prod_{S\in\mathcal{S}_{G}^-} \binom{r}{n_S}^{\nu_S^G}} 
	\frac{\binom{r}{n_{\mathfrak{nb}_G(i_0)}}}{\binom{r}{n_{\overline{\mathfrak{nb}}_G(i_0)}}} \binom{r-|n_{\mathfrak{nb}_G(i_0)}|}{n_{i_0}} =  \frac{\prod_{C\in\mathcal{C}_{G}^+} \binom{r}{n_C}}{\prod_{S\in\mathcal{S}_{G}^-} \binom{r}{n_S}^{\nu_S^G}},
	\]
	where the latter equality follows from a well-known identity for the multinomial coefficient. 
	
	In case (b), we obtain $\mathcal{C}_G^+ = \mathcal{C}_{\tilde G}^+\cup\{ \overline{\mathfrak{nb}}_G(i_0) \}$,  $S_0=\mathfrak{nb}_G(i_0)\in \mathcal{S}_G^-$, $\nu_{S_0}^G = \nu_{S_0}^{\tilde{G}}+1$ and $\nu_{S}^G = \nu_{S}^{\tilde{G}}$ for $S\in\mathcal{S}_{\tilde{G}}\setminus\{S_0\}$.  
	Thus, in view of \eqref{(rn)}, 
	\[
	\gbinom{r}{n}{G} =  \frac{\prod_{C\in\mathcal{C}_{G}^+} \binom{r}{n_C}}{\prod_{S\in\mathcal{S}_{G}^-} \binom{r}{n_S}^{\nu_S^G}} 
	\frac{\binom{r}{n_{\mathfrak{nb}_G(i_0)}}}{\binom{r}{n_{\overline{\mathfrak{nb}}_G(i_0)}}}
	\binom{r-|n_{\mathfrak{nb}_G(i_0)}|}{n_{i_0}}
	\]
	and the same argument follows. 
	The assertion for the graph-multinomial coefficient of the second type follows similarly. 
	
	Any moral DAG $\mathcal{G}$ can be constructed from a perfect elimination order of vertices $(v_1,\ldots,v_{|V|})$ by setting $v_j\to v_i$ if $i<j$ and $v_i\sim v_j$. Since each $v_i$ is simplicial in $G_{\{v_i,\ldots, v_{|V|}\}}$, the second part of the assertion follows by iterating formulas of Lemma \ref{lem:constant-recustion}.
\end{proof}

\begin{proof}[Proof of Lemma \ref{lem:MGseries}]

Take $x\in M_G$, and define $u\in \mathbb{R}^V$ by \eqref{eq-x-in-u}.
	By the definition of $u$ and Lemma \ref{lem:explicit-constants}, we have
	\[
	\gbinomm{|n|+r-1}{n}{G} x^n = \prod_{i\in V}\binom{n_i+|n_{\mathfrak{pa}(i)}|  +r-1}{ n_i} u_i^{n_i} (1-u_i)^{|n_{\mathfrak{pa}(i)}|}.
	\]
It follows that
	\[
	\sum_{n \in\mathbb{N}^V} \gbinomm{|n|+r-1}{n}{G} x^n  = \sum_{n \in\mathbb{N}^V}\,\prod_{i\in V}\binom{n_i+|n_{\mathfrak{pa}(i)}|  +r-1}{ n_i} u_i^{n_i} (1-u_i)^{|n_{\mathfrak{pa}(i)}|} =\prod_{i\in V} (1-u_i)^{-r},
	\]
 where the summation with respect to $n\in\mathbb N^V$ is taken following the  perfect ordering of vertices associated with $\mathcal G$, i.e., the $k$th sum is over $n_{i_k}$, where $i_k$ has no parents among $\{i_{k+1},\ldots,i_{|V|}\}$. 
	Moreover, by \eqref{eq-deltaG}, the product equals $\Delta_G(x)^{-r}$. 

Now, suppose that the series \eqref{eq:DeltaG} converges for some $x\in(0,\infty)^V\setminus M_G$. By Definition \ref{Ddel}, we conclude that 
$\Delta_G(x)\neq 0$.

Case 1: $\Delta_G(x)<0$. Since the series \eqref{eq:DeltaG} converges for $x$, it also converges for $x_t:=t x$ with $t \in (0,1]$. 
However, for $t$ sufficiently small, we have $x_t\in M_G$, which forces $\Delta_G(x_t)>0$. By the continuity of 
$\Delta_G$, there exists $t_0\in(0,1)$ such that $\Delta_G(x_{t_0})=0$. This contradicts the fact that convergence of the series implies $\Delta_G(x_{t_0})$ is nonzero.

Case 2: $\Delta_G(x)>0$. Since $x\not\in  M_G$, there exists $A\subset V$ such that $\Delta_A(x)\leq 0$. The convergence of the series  \eqref{eq:DeltaG} for $x$, implies it also converges for $x'$ defined by $x'_A = x_A$ and $x'_{A^c}=0$. Noting that $\Delta_G(x') = \Delta_A(x)$,  we see that the possibility $\Delta_A(x)=0$ is ruled out by the convergence condition. If instead $\Delta_A(x)<0$, then the argument from Case 1 applies.

In either situation a contradiction is reached, completing the proof. 
\end{proof}

\subsection{Proof from Section \ref{sec:basic}}

\begin{proof}[Proof of Theorem \ref{thm:marginal}]  We prove only the first part; the second follows along the same lines since $\delta_G(x)=\Delta_G(-x)$.

	Let \(\Theta \subset \mathbb{R}^V\) contains a neighborhood of 0, and \(\theta\in \Theta\) such that \(\theta_i=0\) if \(i\notin C\), we have
	\begin{equation}
		\label{eq-LNA}
		\begin{aligned}
			L_{N_C}(\theta_C)=L_N(\theta)=\left( \frac{\Delta_G(x)}{\Delta_G(xe^{\theta})} \right)^r.
		\end{aligned}
	\end{equation}
	By 3) of Lemma \ref{lem-properties-of-clique-polynomials} we have
	\[\Delta_G=\Delta_{V\setminus C}-\sum_{i\in C}\pi_{\{i\}} \Delta_{\mathfrak{nb}_{G^*}(i)}.
	\]
	Since \(C\) is a clique in \(G\), \(\mathfrak{nb}_{G^*}(i) \subset  V\setminus C\) for all \(i\in C\), and by our choice of \(\theta\), we also have \(\Delta_{V\setminus C}(xe^{\theta})=\Delta_{V\setminus C}(x)\). As a consequence,
	\[\frac{\Delta_G(x)}{\Delta_G(x e^{\theta})}=\frac{\Delta_{V\setminus C}(x)-\sum_{i\in C}\,x_i\Delta_{\mathfrak{nb}_{G^*}(i)}(x)}{\Delta_{V\setminus C}(x)-\sum_{i\in C}\,x_i e^{\theta_i}\Delta_{\mathfrak{nb}_{G^*}(i)}(x)}
	=\frac{1-\sum_{i\in C}x_i\frac{\Delta_{\mathfrak{nb}_{G^*}(i)}(x)}{\Delta_{V\setminus C}(x)}}{1-\sum_{i\in C}\,e^{\theta_i}\,x_i\frac{\Delta_{\mathfrak{nb}_{G^*}(i)}(x)}{\Delta_{V\setminus C}(x)}}=\tfrac{1-\sum_{i\in C}\,x^C_i}{1-\sum_{i\in C}\,e^{\theta_i}x^C_i},\]
	where \(x^C=(x_i^C)_{i\in C}\) is defined in \eqref{eq-xC}. 
	Eq. \eqref{eq-LNA} and the fact that $\Delta_C=1-\sum_{i\in C}\pi_{\{i\}}$ for $C\in\mathcal C_G$  imply
	\[L_{N_C}(\theta_C)=\left( \frac{\Delta_C(x^C)}{\Delta_C(x^Ce^{\theta_C})} \right)^r,\]
	in a neighborhood of \(0\in \mathbb{R}^C\).
\end{proof}

\subsection{Proofs from Section \ref{sec:bivariate}}
\begin{Lemma}\label{lem:hypergeo}
	Let $|z_S|=\sum_{v\in S} z_v$, where $S$ is a nonempty set. 
	\begin{enumerate}[label={\arabic*)}, itemjoin={,}, itemjoin*={, and}]
		\item If $c$ is not a non-positive integer, then 
		\begin{align*}
			\sum_{n_S\in\mathbb{N}^S} \frac{ (a)^{(|n_S|)} (b)^{(|n_S|)}   }{  (c)^{(|n_S|)} } \prod_{v\in S}\frac{ z_v^{n_v}}{n_v!} = {}_2F_1(a,b;c; |z_S|)
		\end{align*}
		provided the rhs above is well defined. 
		\item If $c$ is a negative integer and $a,b$ are non-positive integers  such that $\min\{a,b\}\geq c$, then 
		\begin{align*}
			\sum_{n_S\in\mathbb{N}^S\colon |n_S|\leq \min\{-a,-b\}} \frac{ (a)^{(|n_S|)} (b)^{(|n_S|)}   }{  (c)^{(|n_S|)} } \prod_{v\in S}\frac{ z_v^{n_v}}{n_v!} = {}_2F_1(a,b;c; |z_S|).
		\end{align*}
	\end{enumerate}
\end{Lemma}
\begin{proof}
	1) By definition of ${}_2F_1$ we have
	\begin{align*}
		{}_2F_1(a,b;c; |z_S|) &= \sum_{k=0}^\infty \frac{(a)^{(k)}(b)^{(k)}}{(c)^{(k)}} \frac{(\sum_{v\in S} z_v)^k}{k!}
		=\sum_{k=0}^\infty \frac{(a)^{(k)}(b)^{(k)}}{(c)^{(k)}}  \sum_{n_S\in \mathbb{N}^S} \mathbbm{1}_{|n_S|=k} \prod_{v\in S} \frac{ z_v^{n_v}}{n_v!}
	\end{align*}
	and the assertion follows from swapping the sums. 
	
	2) If $a$ and $b$ are non-positive integers, then the series in the definition of ${}_2F_1$ terminates and we have
	\[
	{}_2F_1(a,b;c; z) = \sum_{k=0}^{\min\{-a,-b\}} \frac{(a)^{(k)}(b)^{(k)}}{(c)^{(k)}} \frac{z^k}{k!}.
	\]
	If $c$ is a negative integer such that $\min\{a,b\}\geq c$, then the above formula still makes sense. Therefore, 2) follows from the same argument as 1). 
\end{proof}
\begin{proof}[Proof of Theorem \ref{thm:pairwise}]
	By the Markov property, $N_i$ and $N_j$ are conditionally independent given $N_S$. Thus,
	\begin{align}\label{eq:sum-nS}
		\mathbb{P}(N_i=n_i, N_j=n_j) =
		\sum_{n_S} \frac{\mathbb{P}(N_{S\cup \{i\}} = n_{S\cup\{i\}}) \mathbb{P}(N_{S\cup \{j\}} = n_{S\cup\{j\}})}{\mathbb{P}(N_S = n_{S})}.
	\end{align}
	We claim that $S\in\mathcal{C}_G$. Indeed, if $S$ is not a clique, then $G_{S\cup\{i,j\}}$ has an induced cycle of length at least $4$, which contradicts decomposability. In particular $S\cup\{i\}, S\cup\{j\}\in\mathcal{C}_G$. 
	By Theorem \ref{thm:marginal}, the distributions of $N_S$, $N_{S\cup \{i\}}$ and $N_{S\cup \{j\}}$ are the negative multinomial (resp. multinomial) distributions with parameters given by \eqref{eq-xC} (resp. \eqref{eq-yC}). 
	
	1) We have 
	\[
	\frac{\mathbb{P}(N_{S\cup \{i\}} = n_{S\cup\{i\}}) \mathbb{P}(N_{S\cup \{j\}} = n_{S\cup\{j\}})}{\mathbb{P}(N_S = n_{S})} = \frac{ \binom{ r+|n_S|+n_i-1}{n_{S\cup\{i\}}} \binom{ r+|n_S|+n_j-1}{n_{S\cup\{j\}}}}{\binom{ r+|n_S|-1}{n_S}} \tilde x_1^{n_i} \tilde x_3^{n_j} \prod_{v\in S} y_v^{n_v}  \Delta^r,
	\]
	where $\tilde x_1 = x_i^{S\cup\{i\}}$, $\tilde x_3 = x_j^{S\cup\{j\}}$,  $\Delta = \frac{ (1-|x^{S\cup\{i\}}|)(1-|x^{S\cup\{j\}}|)}{1-|x^S|}$ and 
	\[
	y_v = \frac{x_v^{S\cup\{i\}}x_v^{S\cup\{j\}}}{x_v^{S}} = x_v \Delta_{\mathfrak{nb}_{G^*}(v)}(x) \frac{\Delta_{V\setminus S}(x)}{\Delta_{V\setminus (S\cup\{i\})}(x) \Delta_{V\setminus (S\cup\{j\})}(x)},\qquad v\in S.
	\]
	By \eqref{eq-xC} and \eqref{eq-Delta-G-C}, we have $1-|x^C|=\frac{\Delta_G(x)}{\Delta_{V\setminus C}(x)}$ for any $C\in\mathcal{C}_G$. This implies that
	\[
	\Delta = \frac{\Delta_{V\setminus S}(x)\Delta_G(x)}{\Delta_{V\setminus (S\cup\{i\})}(x)\Delta_{V\setminus (S\cup\{j\})}(x)}
	\quad\mbox{and}\quad \tilde{x}_2:=|y_S| = \frac{\Delta_{V\setminus S}(x)(\Delta_{V\setminus S}(x)-\Delta_G(x))}{\Delta_{V\setminus (S\cup\{i\})}(x) \Delta_{V\setminus (S\cup\{j\})}(x)}.
	\]
	Moreover, by \eqref{eq-Delta_GV-v} we obtain
	\[
	\frac{\Delta_{V\setminus S}(x)}{\Delta_{V\setminus (S\cup\{k\})}(x)} = 1 - x_k \frac{\Delta_{\mathfrak{nb}_{G_{V\setminus S}^*}(k)}(x)}{\Delta_{V\setminus (S\cup\{k\})}(x)},\qquad k=i,j. 
	\]
	We claim that when $k\in \{i,j\}$ we have $\mathfrak{nb}_{G_{V\setminus S}^*}(k)=\mathfrak{nb}_{G^*}(k)$.
	Indeed, we always have $\mathfrak{nb}_{G_{V\setminus S}^*}(k)\subset\mathfrak{nb}_{G^*}(k)$;
	if $v\in \mathfrak{nb}_{G^*}(k)$, then $v\notin S=\mathfrak{nb}_G(i)\cap \mathfrak{nb}_G(j)$, therefore $v\in \mathfrak{nb}_{G_{V\setminus S}^*}(k)$.
	As a result \eqref{eq:tildex} for $\tilde x_1$ and $\tilde x_3$ follows. It can be verified by direct calculation that $\Delta=(1-\tilde{x}_1)(1-\tilde{x}_3)-\tilde x_2=\Delta_{1-2-3}(\tilde{x})$. 
	
	Moreover, 
	\begin{align*}
		\frac{ \binom{ r+|n_S|+n_i-1}{n_{S\cup\{i\}}} \binom{ r+|n_S|+n_j-1}{n_{S\cup\{j\}}}}{\binom{ r+|n_S|-1}{n_S}} =
		\binom{r+n_i-1}{n_i} \binom{r+n_j-1}{n_j} \frac{ (r+n_i)^{(|n_S|)} (r+n_j)^{(|n_S|)}   }{  n_S! (r)^{(|n_S|)} }.
	\end{align*}
	Proof of 1) is completed using Lemma \ref{lem:hypergeo} 1). 
	
	2) This case is treated similarly. By Theorem \ref{thm:marginal} we have
	\[
	\frac{\mathbb{P}(N_{S\cup \{i\}} = n_{S\cup\{i\}}) \mathbb{P}(N_{S\cup \{j\}} = n_{S\cup\{j\}})}{\mathbb{P}(N_S = n_{S})} = \frac{ \binom{r}{n_{S\cup\{i\}}} \binom{r}{n_{S\cup\{j\}}}}{\binom{r}{n_{S}}}  \tilde{y}_i^{n_i} \tilde{y}_j^{n_j} \prod_{v\in S} x_v^{n_v} \delta^{-r},
	\]
	where $\tilde{y}_1 = y_i^{S\cup\{i\}}$, $\tilde{y}_2 = y_j^{S\cup\{j\}}$, $\delta = \frac{(1+|y^{S\cup\{i\}}|)(1+|y^{S\cup\{j\}}|)}{1+|y^S|}$ and
	\[
	x_v = \frac{y_v^{S\cup\{i\}}y_v^{S\cup\{j\}}}{y_v^{S}} = y_v \delta_{\mathfrak{nb}_{G^*}(v)}(y) \frac{\delta_{V\setminus S}(y)}{\delta_{V\setminus (S\cup\{i\})}(y) \delta_{V\setminus (S\cup\{j\})}(y)},\qquad v\in S.
	\]
	By 3) of Lemma \ref{lem-properties-of-clique-polynomials}, we obtain
	\[
	\tilde y_2:= |x_S| = \frac{\delta_{V\setminus S}(y)(\delta_G(y)-\delta_{V\setminus S}(y))}{ {\delta}_{V\setminus (S\cup\{i\})}(y) \delta_{V\setminus (S\cup\{j\})}(y)}.
	\]
	Eq. \eqref{eq:tildey} is proved similarly as in 1). 
	Finally, 
	\begin{align*}
		\frac{ \binom{r}{n_{S\cup\{i\}}} \binom{r}{n_{S\cup\{j\}}}}{\binom{r}{n_{S}}} &= \binom{r}{n_i}\binom{r}{n_j} \frac{ (r-n_i)_{(|n_S|)} (r-n_j)_{(|n_S|)}}{(r)_{|n_S|}} \\
		&=\binom{r}{n_i}\binom{r}{n_j}(-1)^{|n_S|}\frac{ (n_i-r)^{(|n_S|)} (n_j-r)^{(|n_S|)}}{(-r)^{|n_S|}},
	\end{align*}
	where we used the fact that $(-a)_{(k)}=(-1)^k (a)^{(k)}$. The sum in \eqref{eq:sum-nS} ranges over $n_S\in\mathbb{N}$ with $|n_S|\leq \min\{r-n_i,r-n_j\}$. The proof is concluded with the application of Lemma \ref{lem:hypergeo} 2).
\end{proof}

\subsection{Proofs from Section \ref{sec:org7774c71}}

\begin{proof}[Proof of Theorem \ref{thm-strong-directed-hyper-markov-properties}]
	\begin{enumerate}[label={\arabic*)}, itemjoin={,}, itemjoin*={, and}]
		\item By Eq. \eqref{eq-deltaG}  and \eqref{eq-ui-as-xu} we have
		\[\Delta_G(x)=\prod_{i\in V}(1-u^{\mathcal{G}}_i)\quad\mbox{and}\quad x_i=u^{\mathcal{G}}_i \prod_{j\in \mathfrak{ch}(i)}(1-u^{\mathcal{G}}_j),\quad i\in V.\]
	Moreover, by Lemma \ref{lem:explicit-constants}, we have 
	\begin{align*}
		\gbinomm{|n|+r-1}{n}{G} = \prod_{i\in V} \binom{n_i+|n_{\mathfrak{pa}(i)}|  +r-1}{ n_i}
\end{align*}
Plugging the above into \eqref{eq-PNn}, we have for $n\in\mathbb{N}^V$,
\begin{align*}
		\mathbb{P}(N=n)& =\prod_{i\in V} \binom{n_i+|n_{\mathfrak{pa}(i)}|  +r-1}{n_i}\prod_{i\in V}(1-u_i^{\mathcal{G}})^r \prod_{i\in V} \left[ u^{\mathcal{G}}_i \prod_{j\in \mathfrak{ch}(i)} (1-u^{\mathcal{G}}_j) \right]^{n_i}\\
		&= \prod_{i\in V} \binom{n_i+|n_{\mathfrak{pa}(i)}|  +r-1}{n_i} (u^{\mathcal{G}}_i)^{n_i} (1-u^{\mathcal{G}}_i)^{ r +  |n_{\mathfrak{pa}(i)}|}.
\end{align*}
		\item Similarly,
		\[\delta_G(y)=\prod_{i\in V}(1+w_i^{\mathcal{G}})\quad\mbox{and}\quad\ y_i=w^{\mathcal{G}}_i \prod_{j\in \mathfrak{ch}(i)}(1+ w^{\mathcal{G}}_j),\quad i\in V\]
		and
		\[
		\gbinom{r}{n}{G} = \prod_{i\in V} \binom{r-|n_{\mathfrak{pa}(i)}| }{ n_i}.
		\]
		Plugging the above into \eqref{eq-PNn2}, we get for \(n\in \mathbb{N}_{G,r}\),
		\begin{equation*}
			\begin{aligned}
				\mathbb{P}(N=n) & = \prod_{i\in V} \binom{r-|n_{\mathfrak{pa}(i)}| }{ n_i}\prod_{i\in V} (1+w^{\mathcal{G}}_i)^{-r} \prod_{i\in V}\left[w^{\mathcal{G}}_i \prod_{j\in \mathfrak{ch}(i)} (1+ w^{\mathcal{G}}_j)\right]^{n_i}\\
				& = \prod_{i\in V} \binom{r-|n_{\mathfrak{pa}(i)}| }{ n_i} (w^{\mathcal{G}}_i)^{n_i} (1+w^{\mathcal{G}}_i)^{|n_{\mathfrak{pa}(i)}|-r}.
			\end{aligned}
		\end{equation*}
	\end{enumerate}
\end{proof}

\subsection{Proofs from Section \ref{sec:char-markov}}

\begin{proof}[Proof of Theorem \ref{thm:char-markov}]
	We proceed by induction on the number $K$ of maximal cliques of $G$. If $K=1$, then the assertion is trivially satisfied. Assume that the assertion holds true for all decomposable graphs $G$ with $|\mathcal{C}^+_G|= K$ and let $G$ be a graph with $K+1$ maximal cliques. 
	Let $(C_1,\ldots,C_{K+1})$ be a perfect ordering of its maximal cliques, and define
	\[
	R = C_{K+1}\setminus ( C_1\cup\ldots\cup C_K). 
	\]
	The induced subgraph $G'=G_{V\setminus R}$ is connected, decomposable, and has $K$ maximal cliques, $C^+_{G'} = \{ C_1,\ldots,C_K\}$. It is clear that if $N$ is $G$-Markov, then $N_{V\setminus R}$ is $G'$-Markov.  By \eqref{eq-perfect-ordering-max-cliques}, there exists $\ell\in\{1,\ldots,K\}$ such that $S_{K+1} = C_{K+1}\cap \left( \cup_{n=1}^K C_n\right)\subset C_\ell$. This implies that $S_{K+1} = C_{\ell}\cap C_{K+1}$ and we also have $R=C_{K+1}\setminus S_{K+1}$. 
	
	We now prove part 1); part 2) follows using analogous arguments. By the inductive hypothesis, there exists $(r, x')\in (0,\infty) \times M_{G'}$ such that $N_{V\setminus R}\sim \operatorname{nm}_{G'}(r, x')$. 	
	By Theorem \ref{thm:marginal}, we have
	\[
	\mathbb{P}(N_{C_\ell}=n_{C_\ell}) = \operatorname{nm}(r,(x')^{C_\ell})(n_{C_\ell}),
	\]
    where $(x')^{C_\ell}$ is defined in \eqref{eq-xC}. By the assumption 1),
	\[
	\mathbb{P}(N_{C_{K+1}}=n_{C_{K+1}}) = \operatorname{nm}(\tilde{r}, \tilde{x})(n_{C_{K+1}})
	\]
	for some $\tilde{r}>0$ and $\tilde{x}\in (0,\infty)^{C_{K+1}}$ with $|\tilde{x}|<1$.
	Thus, by comparing  marginals 
	\[
	(N_{C_\ell})_{S_{K+1}} = N_{S_{K+1}} = (N_{C_{K+1}})_{S_{K+1}}, 
	\]
	we obtain that 
	\[
	N_{S_{K+1}}\sim \operatorname{nm}\left(r, (x')^{S_{K+1}}\right) \quad\mbox{and } \quad N_{S_{K+1}}\sim \operatorname{nm}\left(\tilde r, \frac{\tilde x_{S_{K+1}}}{1-|\tilde x_{R}|}\right),
	\]
    where $(x')^{S_{K+1}}$ is defined in \eqref{eq-xC}. 
	In particular, we deduce that $\tilde{r}=r$ and obtain $N_{S_{K+1}}\sim \operatorname{nm}\left(r, \frac{\tilde x_{S_{K+1}}}{1-|\tilde x_{R}|}\right)$. 
	Since $S_{K+1}$ separates $V\setminus C_{K+1}$ and $R$ in $G$, we have
	\begin{align*}
		\mathbb{P}(N=n) &=  \frac{ \mathbb{P}(N_{V\setminus R}=n_{V\setminus R}) \mathbb{P}(N_{C_{K+1}}=n_{C_{K+1}})}{\mathbb{P}(N_{S_{K+1}}=n_{S_{K+1}})} \\
		&=
		\frac{\gbinomm{|n_{V\setminus R}|+r-1}{n_{V\setminus R}}{G'}\Delta_{G'}^r(x') (x')^{n_{V\setminus R}} \binom{|n_{C_{K+1}}|+r-1}{n_{C_{K+1}}} \tilde{x}^{n_{C_{K+1}}}\left(1-\left|\tilde x \right|\right)^r}{
			\binom{|n_{S_{K+1}}|+r-1}{n_{S_{K+1}}} \left(\frac{\tilde{x}_{S_{K+1}}}{1-|\tilde{x}_R|}\right)^{n_{S_{K+1}}}\left(1-\left| \frac{\tilde{x}_{S_{K+1}}}{1-|\tilde{x}_R|} \right|\right)^r}\\
		&\propto \gbinomm{|n|+r-1}{n}{G}
		(x'_{V\setminus C_{K+1}})^{n_{V\setminus C_{K+1}}} (x'_{S_{K+1}}(1-|\tilde{x}_R|))^{n_{S_{K+1}}}  \tilde{x}_R^{n_{R}} = \gbinomm{|n|+r-1}{n}{G} x^n,
	\end{align*}
	where we have used the second equality of the first part of Lemma \ref{lem:explicit-constants} and denoted for $i\in V$,
	\[
	x_i = \begin{cases}
		x'_i, & i\in V\setminus C_{K+1},\\
		\tilde{x}_i, & i\in R,\\
		x'_{S_{K+1}}(1-|\tilde{x}_R|), & i\in S_{K+1}.
	\end{cases}
	\]
	Since $\sum_{n\in\mathbb{N}^V} \mathbb{P}(N=n)=1<\infty$, by Lemma \ref{lem:MGseries}, we conclude that  $x\in M_G$, and thus $N\sim \operatorname{nm}_G(r, x)$. 
\end{proof}

\subsection{Proofs from Section \ref{sec:org7b431d8}}
\begin{proof}[Proof of Theorem \ref{thm-nm-proba-interpretation}]
	We have for \(n\in \mathbb{N}^V\),
	\begin{align*}
		\mathbb{P}( \varepsilon(\ell)=n) &= \frac{\sum_{[w]\in L\colon\varepsilon([w])=n} \pi_w(x)}{\sum_{[w']\in L}\pi_{w'} (x)} \\
		&=\left( \sum_{[w]\in L\colon\varepsilon([w])=n}\prod_{i\in V} x_i^{n_i} \right) \left( \sum_{[w']\in L} \pi_{w'}(x)\right)^{-1}\\
		&=K_G(n)\prod_{i\in V}x_i^{n_i}\,\left( \sum_{[w']\in L} \pi_{w'}(x)\right)^{-1},
	\end{align*}
	where \(K_G(n) = \sum_{[w]\in L\colon\varepsilon([w])=n}1 = \left|\{[w]\in L\colon\varepsilon([w])=n \}\right|\). Thus, by the Cartier-Foata Theorem \ref{CFT}, we get
	$$
	\mathbb{P}( \varepsilon(\ell)=n)=K_G(n) \Delta_G(x)\prod_{i\in V} x_i^{n_i}.
	$$
	The constants \(K_G(n)\) are then uniquely determined by the multivariate power series expansion
	\[\Delta_G^{-1}(x) = \sum_{n\in\mathbb{N}^V} K_G(n)\prod_{i\in V}x_i^{n_i}.\]
 In view of the second part of Definition \ref{Ddel} with $r=1$, we see that  \(K_G(n)=\gbinomm{|n|}{n}{G}\) and the proof is complete.
\end{proof}

\begin{proof}[Proof of Theorem \ref{thm-mult-proba-interpretation}]
	By definition \(\ell\) contains at most one vertex from every maximal clique of \(G\).  For \(n\in \mathbb{N}_{G,1}\) we have $\varepsilon(\ell)=n$ i.f.f.  $\ell=C(n):=\{i\in V,\ n_i=1\}\in \mathcal{C}_{G^*}$. Therefore
	\[\mathbb{P}(\varepsilon(\ell)=n)=\mathbb{P}(\ell=C(n))=\tfrac{y^n}{\delta_G(y)}\]
	and we  note that \(\gbinom{1}{n}{\,G}=1\).
\end{proof}

\subsection{Proofs from Section \ref{sec:orgb65e48a}}
The proof of Theorem  \ref{thm-norcons} is based on Theorem \ref{thm-independent-u-w} and the following lemma, which we state without a proof. It is a simple consequence of the fact that each moral DAG  is related to a perfect elimination ordering of  its  vertices.
\begin{Lemma}\ 
	\begin{enumerate}[label={\arabic*)}, itemjoin={,}, itemjoin*={, and}]
		\item If \(\mathcal{G}\) is a DAG  with arbitrary skeleton, then there exists  at least one  sink vertex in $\mathcal G$.
		\item If \(\mathcal{G}\) is a moral DAG  then any sink vertex in $\mathcal G$ is simplicial  in the skeleton of \(\mathcal{G}\).
		\item Let \(\mathcal{G}\) be a moral DAG.  Let $i_0$ be its sink vertex.  Consider \(V'= V\setminus \{i_0\}\) and \(\mathcal{G}'=\mathcal{G}_{V'}\). Then parent functions $\mathfrak{pa}_\mathcal G$ and $\mathfrak{pa}_{\mathcal G'}$  coincide on \(V'\).
	\end{enumerate}
	\label{lem-sink}
\end{Lemma}

\begin{proof}[Proof of Theorem \ref{thm-norcons}]
	Eq. \eqref{eq-KG-per-i} follows from integrating both sides of \eqref{eq-density-x-to-u-Dir} over \((0,1)^V\), and Eq. \eqref{eq-kG-per-i} can be obtained by integration of \eqref{eq-density-y-to-w-IDir} over \((0,\infty)^V\). 
	
	We will prove \eqref{eq-KG-C-S} by induction with respect to the size of \(V\).  The proof of \eqref{eq-kG-C-S} follows according to the same lines and we skip it.  For \(V=\{i\}\), we see that \(\operatorname{Dir}_G(\alpha,\beta)\) coincides with a Beta \(\operatorname{B}_I(\alpha_i,\beta)\) distribution, i.e. \(K_G(\alpha,\beta)=\frac{\Gamma(\alpha_i+ \beta)}{\Gamma(\alpha_i)\Gamma(\beta)}\) and thus \eqref{eq-KG-C-S} holds true.
	
	Now consider a moral DAG \(\mathcal{G}\) with skeleton \(G\). Assume that \eqref{eq-KG-C-S} holds true for any decomposable graph with number of vertices strictly less than \(|V|\). By Lemma \ref{lem-sink} there exists a  sink vertex in   \(\mathcal{G}\), say \(i_0\). Since \(\mathcal{G}\) is moral, we have  \(\mathfrak{pa}_\mathcal G(i_0) \in \mathcal{C}_G\). Denote \(V' = V \setminus \{i_0\}\). Applying \eqref{eq-KG-per-i} to \(G\), we have 
	\[K_G(\alpha,\beta)= \frac{\Gamma(|\alpha_{\overline{\mathfrak{pa}}(i_0)}| + \beta)}{\Gamma(|\alpha_{\mathfrak{pa}(i_0)}| + \beta) \Gamma(\alpha_{i_0})} \prod_{i\in V'} \frac{\Gamma(|\alpha_{\overline{\mathfrak{pa}}(i)}| + \beta)}{\Gamma(|\alpha_{\mathfrak{pa}(i)}| + \beta) \Gamma(\alpha_i)}.\]
	Define the induced subgraph \(G'=G_{V'}\) and induced directed subgraph \(\mathcal{G}' = \mathcal{G}_{V'}\). As \(i_0\) is a  sink vertex ,  \(\mathcal{G}'\) is a moral DAG with (decomposable)  skeleton \(G'\). Moreover, by Lemma \ref{lem-sink}, the parent function \(\mathfrak{pa}'\) of \(\mathcal{G}'\) coincides on \(V'\) with the parent function \(\mathfrak{pa}\) of \(\mathcal{G}\). By inductive hypothesis applied to the product $\Pi_{i\in V'}$ above, we have
	\begin{equation}
		\label{eq-KG-alpha-beta-induction}
		\begin{aligned}
			K_G(\alpha,\beta) =  \frac{\Gamma(|\alpha_{\overline{\mathfrak{pa}}(i_0)}| + \beta)}{\Gamma(|\alpha_{\mathfrak{pa}(i_0)}| + \beta) \Gamma(\alpha_{i_0})} \times\frac{ \prod_{C\in \mathcal{C}_{G'}^+} \Gamma(|\alpha_C| + \beta)}{\Gamma(\beta)^{m'} \prod_{i\in V'} \Gamma(\alpha_i) \prod_{S\in \mathcal{S}_{G'}^-} \Gamma(|\alpha_S| + \beta)^{\nu_S^{G'}}},
		\end{aligned}
	\end{equation}
	where \(m'\) is the number of connected components of \(G'\)
    . To complete the induction step, we distinguish two cases.
	\begin{enumerate}[label={\arabic*)}, itemjoin={,}, itemjoin*={, and}]
		\item Either \(\mathfrak{pa}(i_0) = \emptyset\), i.e., \(i_0\) is not connected in \(G\) to any vertices from \(V'\). Then \(m=m'+1\), $\nu_S^{G}=\nu_S^{G'}$ for any \(S\in\mathcal{S}_{G'}^- = \mathcal{S}_G^-\) and \(\mathcal{C}_G^+ = \mathcal{C}_{G'}^+ \cup \{\{i_0\}\}\). Moreover, \(\overline{\mathfrak{pa}}(i_0)=\{i_0\}\).  Therefore, $\Gamma(|\alpha_{\overline{\mathfrak{pa}}(i_0)}| + \beta)=\Gamma(\alpha_{i_0} + \beta)$ and $\Gamma(|\alpha_{\mathfrak{pa}(i_0)}| + \beta)=\Gamma(\beta)$ and thus \eqref{eq-KG-alpha-beta-induction} yields 
        \eqref{eq-KG-C-S}. 
		\item Or \(\mathfrak{pa}(i_0) \ne \emptyset\), then \(m=m'\). Since \(\mathfrak{pa}(i_0) \subset  V'\) forms a clique in \(G'\), there are two possibilities:
		\begin{enumerate}
			\item \(\mathfrak{pa}(i_0) \in \mathcal{C}_{G'}^+\). Consequently, $\nu_S^{G'}=\nu_S^{G}$ for all \(S\in\mathcal{S}_{G'}^- = \mathcal{S}_G^-\) and \(\mathcal{C}_G^+ = \{\overline{\mathfrak{pa}}(i_0)\} \cup \left( \mathcal{C}_{G'}^+ \setminus   \{\mathfrak{pa}(i_0)\} \right)\).  Thus, writing
            $$\prod_{C\in \mathcal{C}_{G'}^+} \Gamma(|\alpha_C| + \beta)=\Gamma(| \alpha_{\mathfrak{pa}(i_0)}| + \beta) \prod_{C\in \mathcal{C}_{G'}^+ \setminus \{\mathfrak{pa}(i_0)\} } \Gamma(|\alpha_C| + \beta)$$ together with
            \eqref{eq-KG-alpha-beta-induction} 
            yield \eqref{eq-KG-C-S}.
			\item \(\mathfrak{pa}(i_0) \notin \mathcal{C}_{G'}^+\). Then \(\mathcal{C}_G^+ = \{\overline{\mathfrak{pa}}(i_0)\} \cup \mathcal{C}_{G'}^+\), $S_0 = \mathfrak{pa}(i_0) \in \mathcal{S}_G^-$, \(\nu_{S_0}^G = \nu_{S_0}^{G'}+1\) and \(\nu_{S}^G = \nu_{S}^{G'}\) for $S\in  \mathcal{S}_{G'}^-\setminus\{S_0\}$. Thus \eqref{eq-KG-alpha-beta-induction} reduces to \eqref{eq-KG-C-S}. 
		\end{enumerate}
	\end{enumerate}
\end{proof}

\subsection{Proofs from Section \ref{sec:org6c5fe5d}}

\begin{proof}[Proof of Theorem \ref{thm-independent-u-w}]
	\begin{enumerate}[label={\arabic*)}, itemjoin={,}, itemjoin*={, and}]
		\item Since for \(i\in V\), \(x_i=u^{\mathcal{G}}_i \prod_{j\in \mathfrak{ch}(i)}(1-u_j)\), the Jacobian of \(u^{\mathcal{G}}\) is
		\[\left| \frac{\partial x}{\partial u } \right| = \prod_{i\in V} \prod_{j\in \mathfrak{ch}(i)}(1-u_j) = \prod_{i\in V} (1-u_i)^{|\mathfrak{pa}(i)|}.\]
		By \eqref{eq-deltaG} and 3) of Lemma \ref{lem-clique-poly-new-coord}, the density $f$ of \(U^{\mathcal{G}}\) is
		\begin{equation}
			\label{eq-density-x-to-u-Dir}
			\begin{aligned}
				f(u) & = \left(\prod_{i\in V} (1-u_i)^{|\mathfrak{pa}(i)|}\right)\,K_G(\alpha,\beta) \left( \prod_{i\in V}(1-u_i) \right)^{\beta-1} \prod_{i\in V} \left( u_i \prod_{j\in \mathfrak{ch}(i)}(1-u_j) \right)^{\alpha_i-1} \mathbbm{1}_{(0,1)^V}(u) \\
				& = K_G(\alpha,\beta) \prod_{i\in V} u_i^{\alpha_i-1} (1-u_i)^{|\alpha_{\mathfrak{pa}(i)}|+\beta-1} \mathbbm{1}_{(0,1)}(u_i).
			\end{aligned}
		\end{equation}
		\item Similarly, we have
		\[\left| \frac{\partial y}{\partial w } \right| = \prod_{i\in V} \prod_{j\in \mathfrak{ch}(i)}(1+w_j) = \prod_{i\in V} (1+w_i)^{|\mathfrak{pa}(i)|}.\]
		Now by \eqref{eq-deltaG} and 4) of Lemma \ref{lem-clique-poly-new-coord}, the density $g$  of \(W^{\mathcal{G}}\) is
		\begin{equation}
			\label{eq-density-y-to-w-IDir}
			\begin{aligned}
				g(w) & = \left(\prod_{i\in V}(1+w_i)^{|\mathfrak{pa}(i)|}\right)\,k_G(\alpha,\beta) \left( \prod_{i\in V}(1+w_i) \right)^{-\beta} \prod_{i\in V} \left( w_i \prod_{j\in \mathfrak{ch}(j)}(1+w_j) \right)^{\alpha_i-1} \mathbbm{1}_{(0,\infty)^V}(w) \\
				&= k_G(\alpha,\beta) \prod_{i\in V} w_i^{\alpha_i-1} (1+w_i)^{|\alpha_{\mathfrak{pa}(i)}|-\beta} \mathbbm{1}_{(0,\infty)}(w_i).
			\end{aligned}
		\end{equation}
	\end{enumerate}
\end{proof}

\begin{proof}[Proof of Theorem \ref{thm-G-Markov-GDir}]
	We only detail the case \(X\sim \operatorname{Dir}_G(\alpha,\beta)\), the case \(X\sim \operatorname{IDir}_G(\alpha,\beta)\) can be treated with exactly the same argument.
	
	Choose a moral DAG \(\mathcal{G}\) with skeleton \(G\), by Theorem \ref{thm-independent-u-w}, the components of the random vector \(U^{\mathcal{G}}=u^{\mathcal{G}}(X)\) are independent. By 3) of Lemma \ref{lem-clique-poly-new-coord} we have, for all \(i\in V\),
	\[X_i=U_i^{\mathcal{G}} \prod_{j\in \mathfrak{ch}(i)}(1-U^{\mathcal{G}}_j).\]
	Now for \(A,B\) disconnected in \(G\) and \(i\in A, \ j\in B\),  in view of morality of $\mathcal G$ we have \(\mathfrak{ch}(i)\cap \mathfrak{ch}(j) = \emptyset\). Consequently,   \(X_i\) and \(X_j\) are build on separate subsets of \(\{U_k\}_{k\in V}\), and therefore they are independent.
\end{proof}

\begin{proof}[Proof of Theorem \ref{thm:dirC}]
	Recall definition of the mapping $M_G\ni x\mapsto x^C$ from \eqref{eq-xC}. 
	Consider the function $\psi_C\colon M_G\ni x\mapsto (x^C, x_{V\setminus C})$, which is clearly a bijection. 
	By  the first part of \eqref{eq-Delta_GV-v}  of Lemma \ref{lem-properties-of-clique-polynomials} and the definition of $x^C$ we have
	\[
	\Delta_G(x) = \Delta_{V\setminus C}(x)\left(1-\sum_{i\in C} x_i^C\right) = \Delta_{G_{V\setminus C}}(x_{V\setminus C}) \Delta_{K_C}(x^C).
	\]
	The above factorization together with \eqref{eq-support-MG} implies that $\psi_C(M_G) =  M_{K_C}\times M_{G_{V\setminus C}}$.  Thus, $\psi_C\colon M_G\mapsto M_{K_C}\times M_{G_{V\setminus C}}$ is a diffeomorphism and the joint density of  $\psi_C(X)=(X^C, X_{V\setminus C})$ is given by 
	\begin{equation*}
		\begin{aligned}
			f_{(X^C, X_{V\setminus C})}(x^C,x_{V\setminus C}) &= f_{X}\left(\psi_C^{-1}(x^C,x_{V\setminus C})\right) |J_{\psi_C^{-1}}\left(x^C,x_{V\setminus C}\right)| \\
			&=
			\prod_{i\in C} \left(x_i^C \frac{\Delta_{V\setminus C}(x_{V\setminus C})}{\Delta_{\mathfrak{nb}_{G^*}(i)}(x_{V\setminus C})}\right)^{\alpha_i-1} 
			\prod_{i\in V\setminus C} x_i^{\alpha_i-1} \Delta^{\beta-1}_{V\setminus C}(x_{V\setminus C})\Delta^{\beta-1}_{C}(x^C)\\
			&\qquad\cdot\mathbbm{1}_{M_{K_C}\times M_{G_{V\setminus C}}}(x^C,x_{V\setminus C})		\prod_{i\in C}  \frac{\Delta_{V\setminus C}(x_{V\setminus C})}{\Delta_{\mathfrak{nb}_{G^*}(i)}(x_{V\setminus C})} \\
			&= f_{\operatorname{Dir}_{K_C}(\alpha_C,\beta)}(x^C) \cdot g(x_{V\setminus C}), 
		\end{aligned}
	\end{equation*}
    for some function $g$, 
	where $$|J_{\psi_C^{-1}}\left(x^C,x_{V\setminus C}\right)|= \left| \frac{\partial \psi_C^{-1}}{\partial x}\right|=\prod_{i\in C}  \frac{\Delta_{V\setminus C}(x_{V\setminus C})}{\Delta_{\mathfrak{nb}_{G^*}(i)}(x_{V\setminus C})} $$ is the Jacobian of $\psi_{C}^{-1}$. Thus, the independence follows. 
	The case of $\operatorname{IDir}_G$ is treated analogously. 
\end{proof}

\subsection{Proofs from Section \ref{sec:org9c216d9}}

\begin{proof}[Proof of Theorem \ref{thm:charact}]
	We present a proof of a characterization of $\operatorname{Dir}_G$. The proof of a characterization of $\operatorname{IDir}_G$ goes along the same lines. We will indicate the key differences after the proof of the $\operatorname{Dir}_G$ case.
	
	If $G$ is complete, then this result follows from \cite[Theorem 1]{JM80} as was indicated in Section \ref{sec:org9c216d9}.
	In particular, this implies that the result is true for $|V|=2$. 
	
	In the sequel we assume that $|V|>2$ and that $G$ is not complete.  We adapt to the present setting the induction argument from the proof of Theorem 4 in \cite{PW18}, where it was used for a statistical model with a tree structure.  
	Assume  that the assertion holds for all decomposable graphs with $|V|\leq d$ and consider a graph $G$ with $d+1$ vertices. 
	Let $i_0$ be a  simplicial vertex in $G$. Let $V'=V\setminus\{i_0\}$ and $G'=G_{V'}$. Define the mapping $M_G\ni x\mapsto x'\in M_{G'}$ by 
	\[
	x'_i = \begin{cases}
		\frac{x_i}{1-x_{i_0}}, & i\in \mathfrak{nb}_G(i_0),\\
		x_i, & i\in V\setminus \overline{\mathfrak{nb}}_G(i_0)=\mathfrak{nb}_{G^*}(i_0).
	\end{cases}
	\]
	Let $X'$ be the image of $X$ under the above map.
	
	For any moral DAG $\mathcal{G}$ with skeleton $G$ such that  $i_0$ is a sink of $\mathcal G$  the following assertion holds true:
	\begin{align}\label{eq:assertion}
		\left(u^{\mathcal{G}}_j(x)\right)_{j\in V'} = u^{\mathcal G'}(x'),
	\end{align}
	where $\mathcal G'=\mathcal{G}_{V'}$ is the induced directed graph.  In order to prove \eqref{eq:assertion} denote $u^{\mathcal{G}}:=u^{\mathcal{G}}(x)$,  $\mathfrak{ch}:=\mathfrak{ch}_\mathcal G$ and $\mathfrak{ch}':=\mathfrak{ch}_{\mathcal G'}$. 
	By \eqref{eq-x-in-u}, we have $x_i = u_i^{\mathcal{G}}\prod_{j\in \mathfrak{ch}(i)}(1-u_j^{\mathcal{G}})$, $i\in V$. 
	In particular, $x_{i_0}=u^\mathcal G_{i_0}$. Consequently,
	\[
	x'_i = \begin{cases}
		\frac{x_i}{1-x_{i_0}} = \frac{u_i^{\mathcal{G}}\prod_{j\in \mathfrak{ch}(i)}(1-u_j^{\mathcal{G}})}{1-u^\mathcal G_{i_0}}, & i\in \mathfrak{nb}_G(i_0),\\
		x_i = u_i^{\mathcal{G}}\prod_{j\in \mathfrak{ch}(i)}(1-u_i^{\mathcal{G}}), & i\in \mathfrak{nb}_{G^*}(i_0).
	\end{cases}
	\]
 Since 
	$\mathfrak{ch}(i) =  \mathfrak{ch}'(i)\cup \{i_0\}$ for $i\in \mathfrak{nb}_G(i_0)$ and $\mathfrak{ch}(i) =	\mathfrak{ch}'(i)$ for $i\in \mathfrak{nb}_{G^*}(i_0)$, we get
    $$
    x_i'= u_i^{\mathcal{G}}\prod_{j\in \mathfrak{ch}'(i)}(1-u_j^{\mathcal{G}}),\quad i\in V'.
    $$
	On the other hand, referring again to \eqref{eq-x-in-u}, we can write \(x_i'=u_i^{\mathcal{G}'}\prod_{j\in \mathfrak{ch'}(i)}(1-u_j^{\mathcal{G}'})\), $i\in V'$.  Comparing these two representations of $x_i'$, $i\in V'$, we conclude that \eqref{eq:assertion} holds true. 
	
	By assumption,  components of $u^{\mathcal{G}}(X)$ are independent. Thus, by \eqref{eq:assertion},  in view of the definition of $X'$, we see that the components of $u^{\mathcal{G}'}(X')$ are also independent. 
	
	If we add a vertex $i_0$ to a moral DAG $\mathcal{G}'$ with skeleton $G'$ such that $i_0$ is a  simplicial sink  of the newly obtained directed graph $\mathcal{G}$, then $\mathcal{G}$ is  a moral DAG. Thus, 
	any moral DAG $\mathcal{G}'$ with skeleton $G'$ can be seen as an induced directed subgraph $\mathcal{G}_{V'}$, where $\mathcal{G}$ is as above (i.e.  $i_0$ is a sink in a moral DAG $\mathcal G$). In view of independence of components of $u^{\mathcal{G}'}(X')$, by inductive hypothesis, we conclude that $X'$ has distribution $\operatorname{Dir}_{G'}(\alpha,\beta)$ for some $\alpha\in (0,\infty)^{V'}$ and $\beta>0$. Moreover, since $X_{i_0}=u^\mathcal{G}_{i_0}(X)$ and $X'$ is a function of $(u_i^{\mathcal{G}}(X))_{i\in V'}$, we obtain that $X_{i_0}$ and $X'$ are independent.
	
	The above reasoning, where we used a simplicial vertex $i_0$,  applies to any simplicial vertex in $G$. Since $G$ is assumed non-complete, there exists another simplicial vertex, say $i_1$, which is not a neighbor of $i_0$. In fact, each simplicial vertex belongs to exactly one maximal clique and there are at least two maximal cliques in a non-complete graph. Thus, if we repeat our argument with $i_1$ instead of $i_0$, denote  $G''=G_{V''}$ with $V''=V\setminus \{i_1\}$ and define $X''$ as $X'$ but use $i_1$ instead of $i_0$, then we obtain that $X_{i_1}$ and $X''$ are independent and $X''\sim  \operatorname{Dir}_{G''}(\alpha'',\beta'')$. Since $i_0$ and $i_1$ are not neighbors in $G$, we have $X_{i_0}=X''_{i_0}$ and by Remark \ref{rem:inproof} we obtain $X_{i_0}\sim \operatorname{B}_I(\alpha_{i_0},\beta_{i_0})$ for some $\alpha_{i_0},\beta_{i_0}>0$. 
	
	Since the Jacobian of the diffeomorphism $M_G\ni x\mapsto (x_{i_0},x')\in (0,1)\times M_{G'}$ equals $(1-x_{i_0})^{-|\mathfrak{nb}_G(i_0)|}$  we can write the density $f_X$ of $X$ as 
	\begin{align}
		f_X(x) &= \frac{f_{X_{i_0}}(x_{i_0}) }{(1-x_{i_0})^{|\mathfrak{nb}_G(i_0)}|}\,f_{X'}(x') \nonumber\\ 
		&\propto \frac{x_{i_0}^{\alpha_{i_0}-1}(1-x_{i_0})^{\beta_{i_0}-1}}{(1-x_{i_0})^{|\mathfrak{nb}_G(i_0)}|}\,  \left(\prod_{i\in \mathfrak{nb}_{G}(i_0)} \left(\frac{x_i}{1-x_{i_0}}\right)^{\alpha_i-1} \right)\, \left(\prod_{i\in \mathfrak{nb}_{G^*}(i_0)} x_i^{\alpha_i-1}\right) \Delta_{G'}^{\beta-1}\left(x'\right) \nonumber\\
		&=\, \prod_{i\in V} x_i^{\alpha_i-1} \Delta_G^{\beta-1}(x)(1-x_{i_0})^{\beta_{i_0}-|\alpha_{\mathfrak{nb}_G(i_0)}|-\beta}, \label{eq:last}
	\end{align}
	where in the last equality above we have used identity \eqref{eq:Deltax}.
	Since \eqref{eq:last} has to hold for any simplicial vertex $i_0$, we obtain that $\beta_{i_0}=|\alpha_{\mathfrak{nb}_G(i_0)}|+\beta$ and the proof of the case $\operatorname{Dir}_G$ is complete.
	
	Now we turn to the case of $\operatorname{IDir}_G$. 
	If $G$ is complete and $Y$ has inverted Dirichlet distribution, then  with $X_i = Y_i/(1+\sum_{k=1}^n Y_k)$, $i\in V$, $X$ follows the Dirichlet distribution.  Also   independence of components of a vector $w^{\mathcal{G}}(Y)$, $\mathcal{G}$ being an arbitrary moral DAG  with complete skeleton, is equivalent to  independence of components of vector $u^{\mathcal{G}'}(X)$, where $\mathcal{G}'$ is a moral DAG defined by inverting  directions of all edges in $\mathcal{G}$.  Consequently, in the complete graph case, the  independence  characterization of $\operatorname{IDir}_G$ follows from  the respective  characterization of $\operatorname{Dir}_G$. 
	
	All arguments of the inductive proof remain valid if the mapping $(0,\infty)^V\ni y\mapsto y'\in (0,\infty)^{V'}$ is defined by $y_i'=y_i/(1+y_{i_0})$ for $i\in \operatorname{nb}_G(i_0)$ and $y_i'=y_i$ otherwise. Then, the analog of \eqref{eq:assertion} is $\left(w^{\mathcal{G}}_i(y)\right)_{i\in V'} = w^{\mathcal{G}'}(y')$. Finally, the univariate marginals of the inverted Dirichlet are $\operatorname{B}_{II}$. 
\end{proof}

\begin{proof}[Proof of Theorem \ref{thm:simpleHyper}]
	We prove the result only for the case $\operatorname{Dir}_G$, the case $\operatorname{IDir}_G$ is analogous.
	Recall the notations from Section \ref{sec:org9c216d9} and a definition of the mapping $u^{\mathcal{G}}$ where $\mathcal{G}$ is a moral DAG on $G$, see \eqref{eq-u-calG}.  We will write $\mathfrak{pa}:=\mathfrak{pa}_\mathcal G$, $\mathfrak{de}:=\mathfrak{de}_\mathcal G$ and $\mathfrak{nde}:=\mathfrak{nde}_\mathcal G$. 
    
    Note that
    \begin{align}
		&p^{\mathfrak{nde}(i)}_{k_{\mathfrak{nde}(i)}}(X)
		=\sum_{k_{\mathfrak{de}(i)}\in\mathbb{N}^{\mathfrak{de}(i) }}\mathbb{P}(N=k|X)=\sum_{k_{\mathfrak{de}(i)}\in\mathbb{N}^{\mathfrak{de}(i) }}\,\prod_{\ell\in V}\,\mathbb P(N_\ell=k_\ell|N_{\mathfrak{pa}(\ell)}=k_{\mathfrak{pa}(\ell)},X) \nonumber \\
		&= \left(
		\sum_{k_{\mathfrak{de}(i)}\in\mathbb{N}^{\mathfrak{de}(i) }}\,\prod_{\ell\in\mathfrak{de}(i)}
		p_{k_{\overline{\mathfrak{pa}}(\ell)}}^{\{\ell\}|\mathfrak{pa}(\ell)}(X)
		\right)\prod_{j\in \mathfrak{nde}(i)} p_{k_{\overline{\mathfrak{pa}}(j)}}^{\{j\}|\mathfrak{pa}(j)}(X)
		=\prod_{j\in \mathfrak{nde}_{\mathcal{G}}(i)}\,p_{k_{\overline{\mathfrak{pa}}_\mathcal{G}(j)}}^{\{j\}|\mathfrak{pa}_\mathcal{G}(j)}(X). \label{lastt}
	\end{align}
	By Theorem \ref{thm:Markov} we have (recall the notation we introduced prior to Definition \ref{def-hyper-Markovs}) we have
	$$
	p_{k_{\overline{\mathfrak{pa}}(i)}}^{\{i\}|\mathfrak{pa}(i)}(X)=\operatorname{nm}\left(r+|k_{\mathfrak{pa}(i)}|, u_i^\mathcal{G}(X)\right)(k_i).
	$$
  Combining this with \eqref{lastt} we conclude that $p^{\mathfrak{nde}(i)}_{k_{\mathfrak{nde}(i)}}(X)$ is a function  of $u^{\mathcal G}_j(X)$, $j\in\mathfrak{nde}_\mathcal G(i)$. 
    
 On the other hand, the strong directed hyper Markov property states that for all $i\in V$,
	\[
	p_{\{i\}|\mathfrak{pa}_\mathcal{G}(i)}\mbox{ and }p_{\mathfrak{nde}_\mathcal{G}(i)}\mbox{ are independent}.
	\]
	In view of the above observations it is clearly equivalent to the property: for all $i\in V$, 
	\begin{align}\label{eq:ind}
		u_i^{\mathcal{G}}(X)\mbox{ and } u_{\mathfrak{nde}_{\mathcal{G}(i)}}^{\mathcal{G}}(X)\mbox{ are independent}.
	\end{align}
	Let $d=|V|$ and let $(\sigma(1),\ldots,\sigma(d))$ be a reversed perfect elimination ordering of vertices defined by $\mathcal{G}$, i.e. if $\sigma(i)\to \sigma(j)$ in $\mathcal{G}$, then we have $i>j$.  
	Since $G$ is connected, for each $i\in \{1,\ldots,d\}$ we have $\{ \sigma(k)\}_{k=1}^{i-1}\subset \mathfrak{nde}_{\mathcal{G}}(\sigma(i))$. 
	Thus, \eqref{eq:ind} implies that for all $i=2,\ldots,d$,
	\[
	u_{\sigma(i)}^{\mathcal{G}}(X)\mbox{ and } u_{\{ \sigma(k)\}_{k=1}^{i-1}}^{\mathcal{G}}(X)\mbox{ are independent}.
	\]
In view of the fact that random variables $Z_1,\ldots,Z_d$ are independent if and only if $Z_i$ and $(Z_1,\ldots,Z_{i-1})$ are independent for all $i=2,\ldots,d$, the proof is finished by referring to Theorem \ref{thm:charact}.
\end{proof}

\subsection*{Acknowledgements}
We thank Steffen Lauritzen for his questions regarding hyper Markov properties and for pointing out a connection to classical Dirichlet   negative multinomial  and  Dirichlet multinomial distributions.

I. Danielewska and B. Ko{\l}odziejek were partially supported by the NCN grant UMO-2022/45/B/ST1/00545. \newline  J. Weso\l owski was partially supported by the IDUB grant  1820/366/201/2021, WUT. \newline X. Zeng  acknowledges the Agence Nationale de la Recherche for their financial support via ANR grant RAW ANR-20-CE40-0012-01.

For the purpose of Open Access, the authors have applied a CC-BY public copyright licence to any Author Accepted Manuscript (AAM) version arising from this submission.

\bibliographystyle{plain}
\bibliography{Bibl}
\end{document}